\documentclass[10pt]{article}
\usepackage{graphicx}
\usepackage{amsfonts, amsmath, amsthm, amssymb}
\usepackage{mathrsfs}
\usepackage[utf8]{inputenc}
\usepackage[all]{xy}
\usepackage[english]{babel}
\usepackage{authblk}
\usepackage{xcolor}

\newtheorem{theorem}{Theorem}[section]
\newtheorem{lemma}[theorem]{Lemma}

\theoremstyle{definition}
\newtheorem{remark}[theorem]{Remark}

\numberwithin{equation}{section}

\begin{document}

\title{A Seifert-van Kampen Theorem for Legendrian Submanifolds and Exact Lagrangian Cobordisms}
\author{Mark C. Lowell}

\maketitle

\begin{abstract}

We prove a Seifert-van Kampen theorem for Legendrian submanifolds and exact Lagrangian cobordisms, and use it to calculate the change in the DGA caused by critical Legendrian ambient surgery.

\end{abstract}



\tableofcontents

\section{Introduction}
\label{sec:intro}

Let \(Y\) be a \((2n+1)\)-dimensional manifold, and let \(\zeta \subset TY\) be a hyperplane distribution.   We say the pair \((Y, \zeta)\) is a \textbf{contact manifold} if \(\zeta\) is maximally non-integrable.   If \(\zeta = \ker \alpha\), where \(\alpha\) is a one-form, then maximal non-integrability is equivalent to \(\alpha \wedge (d\alpha)^n\) being non-vanishing.   The primary example of a contact manifold we will use is the \textbf{one-jet bundle of \(M\)}, \(J^1(M) = (T^*M \times \mathbb{R}, \zeta), \zeta = \mbox{ker }\alpha\), where \(\alpha = dz - \theta\), \(z\) is the coordinate of \(\mathbb{R}\), \(\theta\) is the pullback to \(J^1(M)\) of the tautological one-form of \(T^*M\), and \(J^1(M)\) is \(n\)-dimensional, where \(n = \dim M\).

In one-jet bundles, we define the \textbf{Lagrangian projection} \(\pi_\mathbb{C}\), \textbf{front projection} \(\pi_F\), and \textbf{base projection} \(\pi_M\) to be the projections:
\[
\pi_\mathbb{C}:J^1(M) \to T^*M
\]\[
\pi_F:J^1(M) \to M \times \mathbb{R}
\]\[
\pi_M:J^1(M) \to M
\]
If our contact manifolds has a contact form \(\alpha\), we define the \textbf{Reeb vector field} \(R_\alpha\) to be the unique smooth vector field determined by:
\[
\alpha(R_\alpha) = 1
\]\[
\iota_{R_\alpha}(d\alpha) = 0
\]
Where \(\iota_X\) is the inclusion of the vector field \(X\).   For \(J^1(M)\), the Reeb vector field is always \(R_\alpha = \partial_z\).

We say an \(n\)-dimensional submanifold \(\Lambda\) of a contact manifold \((Y, \zeta)\) is \textbf{Legendrian} if \(T\Lambda \subset \zeta\).   This is equivalent to:
{\color{black}\[
\left.\alpha\right|_{T\Lambda} = 0
\]}
When \(Y = J^1(M)\), which for us it always will, we will frequently work with the Lagrangian or front projections of \(\Lambda\), because these are easier to handle since they are lower-dimensional.   Fortunately, we can always recover the entire Legendrian submanifold up to \(\mathbb{R}\)-translation from either projection, because if we regard the submanifold as a collection of graphs \(x \to (x, y(x), z(x))\), then {\color{black}\(\alpha|_{T\Lambda} = 0\)} is equivalent to:
\[
\frac{\partial z}{\partial x_i} = y_i \mbox{ for all }i
\]
We define a \textbf{Reeb chord} of \(\Lambda\) to be a trajectory of \(R_\alpha\) that begins and ends on \(\Lambda\).   Since \(R_\alpha = \partial_z\) for the contact manifold \(J^1(M)\), Reeb chords of Legendrian submanifolds of \(J^1(M)\) are precisely equivalent to self-intersections of \(\pi_\mathbb{C}(\Lambda)\).

Let \(X\) be a \((2n)\)-dimensional manifold, and let \(\omega\) be a two-form on \(X\).   We say \(\omega\) is a \textbf{symplectic form} and \(X\) is a \textbf{symplectic manifold} if \(\omega\) is closed and non-degenerate, that is, \(d\omega = 0\) and \(\omega^n\) is non-vanishing.   We say \((X, \omega)\) is \textbf{exact symplectic} if \(\omega\) is exact, that is, \(\omega = d\beta\) for some one-form \(\beta\).   The \textbf{symplectization} of \(J^1(M)\) is given by \((J^1(M) \times \mathbb{R}, \omega = d\beta), \beta = e^t\alpha\), where \(t\) is the coordinate of \(\mathbb{R}\); this is then a symplectic manifold.   We say an \((n+1)\)-dimensional submanifold \(L \subset J^1(M) \times \mathbb{R}\) is \textbf{Lagrangian} if:
{\color{black}\[
\left.\omega\right|_{TL} = 0
\]}
We say \(L\) is \textbf{exact Lagrangian} if:
{\color{black}\[
\left.\beta\right|_{TL} = df
\]}
For some function \(f\).

We say two Legendrian submanifolds \(\Lambda_0, \Lambda_1\) are \textbf{Legendrian isotopic} if there exists an isotopy \(\Lambda_t\) between them, so that \(\Lambda_t\) is Legendrian for all \(t\).   {\color{black}Similarly, we say that two exact Lagrangian submanifolds are \textbf{exact Lagrangian isotopic} if they are isotopic through exact Lagrangian submanifolds.}   We are interested in classifying the compact Legendrian embedded submanifolds up to Legendrian isotopy.   In general, determining if two Legendrian submanifolds are Legendrian isotopic is extremely difficult.   {\color{black}To support this effort, we calculate invariants of the Legendrian isotopy classes.   In this paper we will examine the Legendrian Contact Homology of a Legendrian submanifold, which is the homology of a unital graded algebra \(\mathcal{A}_K(\Lambda)\) over a field \(K\) generated by the Reeb chords of \(\Lambda\), with a differential given by counting rigid pseudoholomorphic maps of disks into \(T^*M\) with boundary on \(\pi_\mathbb{C}(\Lambda)\)}.   We will explain these terms in section 2.   We will ordinarily assume \(K = \mathbb{Z}_2\), in which case we will suppress the subscript, but our results hold for arbitrary field.

In this paper we prove a Seifert-van Kampen theorem for this Legendrian Contact Homology. If \(U \subset M\), we define \((\mathcal{A}(\Lambda)|_U, \partial|_U)\) to be the graded algebra generated by the Reeb chords lying over \(U\), paired with a graded algebra morphism \(\partial|_U:\mathcal{A}(\Lambda)|_U \to \mathcal{A}(\Lambda)|_U\) given by counting the pseudoholomorphic disks that lie entirely over \(U\).   By abuse of notation, we will use \(\mathcal{A}(\Lambda)|_U\) to refer to the pair.   This map may or may not be a differential, since for arbitrary \(U\), it is possible that \(\partial|_U^2 \neq 0\).   Then:

\begin{theorem}\label{theorem1.1} 
Let \(J^1(M)\) be the one-jet bundle equipped with its standard contact form, let \(\Lambda \subset J^1(M)\) be a compact Legendrian submanifold, and let \(S \subset M\) be a hypersurface that divides \(M\) into two components, \(R_1\) and \(R_2\), and which does not intersect a codimension-2 singularity of the front projection of \(\Lambda\).   Let \(N\) be an arbitrarily small neighborhood of \(S\), and let \(Q_i = R_i \cup N\).   Then, after a Legendrian isotopy that does not change \(\Lambda\) outside of \(\pi_M^{-1}(N)\), \(\mathcal{A}_K(\Lambda)|_N, \mathcal{A}_K(\Lambda)|_{Q_i}\) are well-defined differential graded algebras, and the following diagram is a push-out square:
\[
\xymatrix{
\mathcal{A}_K(\Lambda)   & \mathcal{A}_K(\Lambda)|_{Q_1}\ar[l]^{i_1}\\
\mathcal{A}_K(\Lambda)|_{Q_2} \ar[u]^{i_2} &\mathcal{A}_K(\Lambda)|_N \ar[u]^{j_1}\ar[l]^{j_2}}
\]
Where \(i_1, i_2, j_1, j_2\) are the inclusion maps.
\end{theorem}

The Legendrian isotopy in theorem~\ref{theorem1.1} is called the \textbf{pinching isotopy} and is defined in section~\ref{subsec:proofofmaintheorem}. It has also been referred to as \textbf{dipping} in \cite{Sa}, and is similar to the \textbf{splashing} of \cite{Fu}.   This result has been known for \(\dim \Lambda = 1\) since \cite{Si}, and for \(\dim\Lambda = 2\) and for higher dimensions where the front projection has no codimension-2 singularities since \cite{HS}.   We prove it for all cases.

Note that, since pinching does not change \(\Lambda\) outside of \(N\), we can apply multiple pinchings to separate \(\Lambda\) into more then two components.   In addition, we can pinch inside the pinch zone, to separate \(N\) itself into multiple components.

{\color{black} Given how difficult the infinite-dimensional Legendrian Contact Homology is to work with, it is common to work instead with the Linearized Contact Homologies, which are finite-dimensional invariants of the Legendrian Contact Homology.   The Linearized Contact Homologies are calculated by finding chain maps \(\epsilon: \mathcal{A}(\Lambda) \to \mathbb{Z}_2\), where \(\mathbb{Z}_2\) is turned into a differential graded algebra by equipping it with the trivial differential; we call such a map an \textbf{augmentation}.   An augmentation then induces an algebra isomorphism \(E_\epsilon = \mbox{Id} + \epsilon\).   We can then define a new differential \(\partial^\epsilon = E_\epsilon \circ \partial \circ E_\epsilon^{-1}\).   Then, the restriction of \(\partial^\epsilon\) to words of length 1, which we denote \(\partial^\epsilon_1\), is itself a differential, that is, \((\partial^\epsilon_1)^2 = 0\).   The Linearized Contact Homology of \(\epsilon\) is the homology of the algebra restricted to words of length 1, equipped with the differential \(\partial^\epsilon_1\).   This is a finite-dimensional homology, and thus much easier to work with then the original Legendrian Contact Homology.   For a given Legendrian submanifold, multiple augmentations may exist, and two different augmentations can induce the same or different Linearized Contact Homology; however, the set of Linearized Contact Homologies is an invariant of the stable tame isomorphism class of the differential graded algebra, and therefore an invariant of the Legendrian isotopy class.   Theorem~\ref{theorem1.1} descends to the Linearized Contact Homologies to provide a Mayer-Veitoris theorem:}

\begin{theorem}\label{theorem1.3} Let \(\Lambda \subset J^1(M)\) be a compact Legendrian submanifold that has been pinched along a neighborhood \(N\) of a hyper surface \(S\) that divides \(M\) into \(Q_1, Q_2\) and which does not intersect a codimension-2 singularity of the front projection.   If \(\Lambda\) has an augmentation \(\epsilon\), then \(\epsilon\) induces augmentations \(\epsilon_N, \epsilon_{Q_i}\) of \(\mathcal{A}(\Lambda)|_N, \mathcal{A}(\Lambda)|_{Q_i}\), and there is a long exact sequence:
\[
... \to LCH_k(\left.\mathcal{A}_K(\Lambda)\right|_N, \partial, \epsilon_N) \to
LCH_k(\left.\mathcal{A}_K(\Lambda)\right|_{Q_1}, \partial, \epsilon_{Q_1}) \oplus
\]\[
LCH_k(\left.\mathcal{A}_K(\Lambda)\right|_{Q_2}, \partial, \epsilon_{Q_2}) \to LCH_k(\mathcal{A}_K(\Lambda), \partial, \epsilon) \to 
LCH_{k+1}(\left.\mathcal{A}_K(\Lambda)\right|_N, \partial, \epsilon_N)
\to ...
\]
Where \(LCH_*\) denotes the linearized contact homology.\end{theorem}

Next, let \(\Lambda_+, \Lambda_-\subset J^1(M)\) be Legendrian submanifolds.   {\color{black}We define the \textbf{symplectization} of \(J^1(M)\) to be the exact symplectic manifold \(J^1(M)\times\mathbb{R}\) equipped with the symplectic form \(\omega = d\beta, \beta = e^t\alpha\)}.   We define an \textbf{exact Lagrangian cobordism} from \(\Lambda_+\) to \(\Lambda_-\) to be an exact Lagrangian submanifold \(L \subset J^1(M) \times \mathbb{R}\) such that:

\begin{itemize}

\item There exists \(T > 0\) such that:
\[
\mathscr{E}_+(L) = L \cap \left(J^1(M) \times (T, \infty)\right) = \Lambda_+ \times (T, \infty)
\]\[
\mathscr{E}_-(L) = L \cap \left(J^1(M) \times (-\infty, -T)\right) = \Lambda_- \times (-\infty, -T)
\]

\item If {\color{black}\(df = \beta|_{TL}\)}, then \(f\) is constant on \(\mathscr{E}_+(L)\) and on \(\mathscr{E}_-(L)\).

\item \(L - (\mathscr{E}_+(L) \cup \mathscr{E}_-(L))\) is compact with boundary \(\Lambda_+ \cup \overline{\Lambda_-}\).

\end{itemize}

We refer to \(\mathscr{E}_+(L), \mathscr{E}_-(L)\) as the \textbf{positive and negative cones of \(L\)}.   An exact Lagrangian cobordism \(L\) {\color{black}gives rise to a unital DGA morphism} \(\Phi_L:\mathcal{A}_K(\Lambda_+) \to \mathcal{A}_K(\Lambda_-)\), as discussed in section~\ref{subsec:linearized}.   We have an equivalent of our Siefert-van Kampen theorem for these cobordisms as well:

\begin{theorem}\label{theorem1.4} Let \(J^1(M)\) be the one-jet bundle equipped with its standard contact form, let \(L \subset J^1(M) \times \mathbb{R}\) be an exact Lagrangian cobordism from the Legendrian submanifold \(\Lambda_+ \subset J^1(M)\) to \(\Lambda_- \subset J^1(M)\), and let \(\hat{S} \subset M \times \mathbb{R}\) be a hypersurface that divides \(M \times \mathbb{R}\) into two components \(\hat{R}_1\) and \(\hat{R}_2\), such that {\color{black}\(\hat{S}\) intersects the level set \(\{t = t_0\}\) transversely for all \(t_0\)}, and such that \(\hat{S}\) {\color{black}is disjoint from the projection to \(M \times \mathbb{R}\) of the singularities of the front projection of the} Legendrian lift of \(L\).   Let \(R^\pm_i = \hat{R}_i \cap \{t = \pm T\}\), \(S^\pm = \hat{S} \cap \{t = \pm T\}\), let \(N^\pm\) be an arbitrarily small neighborhood of \(S^\pm\), and let \(Q^\pm_i = R^\pm_i \cup N^\pm\).   Then \(L\) is exact Lagrangian isotopic to an exact Lagrangian cobordism \(L'\) from \(\Lambda'_+\) to \(\Lambda'_-\), where \(\Lambda'_\pm\) are Legendrian isotopic to \(\Lambda_\pm\), {\color{black}\(\Lambda'_\pm\) and \(\Lambda_\pm\) coincide outside of \(N^\pm\)}, and such that the image of the restriction of \(\Phi_{L'}\) to \(\mathcal{A}_K(\Lambda'_\pm)|_{Q^+_i}, \mathcal{A}_K(\Lambda'_\pm)|_{N^+}\) lies in \(\mathcal{A}_K(\Lambda'_\pm)|_{Q^-_i}, \mathcal{A}_K(\Lambda'_\pm)|_{N^-}\).\end{theorem}

{\color{black}Finally, as an application of Theorems~\ref{theorem1.1}, we calculate the change in the differential graded algebra of a Legendrian submanifold caused by a critical Legendrian ambient surgery.   Legendrian ambient surgery is the Legendrian analogue of Morse surgery; Georgios Dmitroglou Rizell calculated the change in the algebra caused by surgeries of index 0 through \(n-2\), but was unable to calculate the change when the index is \(n-1\), which is called the critical Legendrian ambient surgery.   We calculate the change caused by the critical surgery, by using the Legendrian ambient surgery data to ensure that the neighborhood of the surgery is contactomorphic to a standard case, and then pinching around the neighborhood to isolate it from pseudoholomorphic disks outside the neighborhood:}

\begin{theorem}\label{theorem1.5} Let \(\Lambda \subset J^1(M)\) be a Legendrian submanifold, and let \(\hat{\Lambda} \subset J^1(M)\) be the {\color{black}result} of a critical Legendrian ambient surgery on \(\Lambda\).   Then \(\Lambda, \hat{\Lambda}\) are Legendrian isotopic to Legendrian submanifolds \(\Lambda', \hat{\Lambda}'\) with differential graded algebras \(\mathcal{A}(\Lambda')\) and \(\mathcal{A}(\hat{\Lambda}')\), such that:

\begin{itemize}

\item \(\mathcal{A}(\Lambda')\) has one more generator than \(\mathcal{A}(\hat{\Lambda}')\), and, excluding this additional generator, there is a bijection between the generators of \(\mathcal{A}(\Lambda')\) and \(\mathcal{A}(\hat{\Lambda}')\). We denote the generators of \(\mathcal{A}(\hat{\Lambda}')\) by \(\hat{x}_1, ..., \hat{x}_m, \hat{d}^1_0\), and the generators of \(\mathcal{A}(\Lambda')\) by \(x_1, ..., x_m, d^1_0, c\), where \(c\) is the additional generator.

\item Let \(\partial\) be the differential of \(\mathcal{A}(\Lambda')\) and let \(\hat{\partial}\) be the different of \(\mathcal{A}(\hat{\Lambda}')\). For \(x \neq c\), \(\hat{\partial} \hat{x} = \widehat{\partial x}\),  where \(\widehat{\partial x}\) denotes \(\partial x\) with \(x_i\) replaced with \(\hat{x}_i\) for all \(i\). \(c\) does not appear in \(\partial x\) for any \(x\).

\item \(\partial c = 1 + d^1_0\).

\item \(\partial d^1_0 = 0\) and \(\hat{\partial} \hat{d}^1_0 = 0\).

\end{itemize}
\end{theorem}

(The reason we write one of the generators as \(d^1_0\) is to harmonize the notation of Theorem~\ref{theorem1.5} with notation used in section~\ref{sec:surgery}.)

%
%
%

In section~\ref{sec:overview}, we provide an overview of the concepts behind Legendrian Contact Homology and exact Lagrangian cobordisms.   In section~\ref{sec:surgery}, we will use theorem ~\ref{theorem1.1} to prove theorem ~\ref{theorem1.5}.   In section~\ref{sec:proof}, we will prove theorems ~\ref{theorem1.1}, ~\ref{theorem1.3}, and ~\ref{theorem1.4}.

\section{Overview of Legendrian Contact Homology}
\label{sec:overview}

\subsection{Topological Preliminaries}
\label{subsec:topology}

A \textbf{contact manifold} is a \((2n+1)\)-dimensional manifold \(Y\) equipped with a hyperplane distribution \(\zeta\) such that there exists no embedded {\color{black}\((n+1)\)}-dimensional submanifold that is everywhere tangent to \(\zeta\).   A \textbf{contact form} \(\alpha\) is a one-form in \(T^*Y\) such that \(\zeta = \ker \alpha\).   A contact manifold is not necessarily equipped with a contact form, but in our case we are limiting our interest to one-jet manifolds \(Y=J^1(M) = T^*M \times \mathbb{R}\), where \(\zeta\) is the kernel of the contact form \(\alpha = dz - \theta\), where \(\theta\) is the tautological one-form of \(T^*M\) and \(z\) is the \(\mathbb{R}\) coordinate.   In addition, we are particularly interested in the case \(M = \mathbb{R}^n\), \(J^1(M) = \mathbb{R}^{2n+1}\).   In this subcase, we use \(x_1, ..., x_n\) for coordinates on \(M\) and \(y_1, ..., y_n\) to be the corresponding cotangent coordinates.   Then:
\[
\alpha = dz - \sum_{i=1}^n y_idx_i
\]
An \(n\)-dimensional submanifold \(\Lambda \subset Y\) is called \textbf{Legendrian} if \(T \Lambda \subset \zeta\).   This is equivalent to {\color{black}\(\alpha|_{T\Lambda} = 0\)}.

A \textbf{symplectic manifold} is a \((2n)\)-dimensional manifold \(X\) equipped with a closed 2-form \(\omega\) such that \(\omega^n\) is a volume form.   {\color{black}An \textbf{exact symplectic manifold} is a symplectic manifold such that the symplectic form \(\omega = d\beta\) for some one-form \(\beta\).}   We are interested in symplectic manifolds in part because we will frequently be working in \(T^*M\), which is equipped with the canonical symplectic form \(\omega = d\theta\), {\color{black} where \(\theta\) is the tautological one-form}.   Recall that an \(n\)-dimensional submanifold \(L\) of a symplectic manifold is called \textbf{Lagrangian} if {\color{black}\(\omega|_{TL} = 0\)}.   {\color{black}Further, \(L\) is called \textbf{exact Lagrangian} if the symplectic manifold is exact, and \(\beta|_{TL} = df\) for some function \(f:L\to\mathbb{R}\).}  If \(\Lambda \subset J^1(M)\) is an embedded Legendrian submanifold, then \(\pi_\mathbb{C}(\Lambda) \subset T^*M\) is an immersed {\color{black}exact} Lagrangian submanifold, because we can define a function \(z:\pi_\mathbb{C}(\Lambda) \to \mathbb{R}\) that is simply the \(z\) coordinate of that point in \(\Lambda\), and:
{\color{black}
\[
\theta\big|_{T\pi_\mathbb{C}(\Lambda)} = 
(dz - \alpha)\big|_{T\pi_\mathbb{C}(\Lambda)} = 
dz\big|_{T\pi_\mathbb{C}(\Lambda)}
\]\[
(d\theta)\big|_{T\pi_\mathbb{C}(\Lambda)} = 
d(dz)\big|_{T\pi_\mathbb{C}(\Lambda)} = 0
\]}
An \textbf{almost complex structure} on a \((2n)\)-dimensional manifold \(X\) is a smooth linear map \(TX \to TX\) such that \(J \circ J = -\mbox{Id}\).   The standard example is the complex plane \(\mathbb{C}\) with real coordinates \(x,y\) and canonical {\color{black}complex structure} \(i\), where:
\[
i(\partial_x) = \partial_y
\]\[
i(\partial_y) = -\partial_x
\]
Given two manifolds with almost complex structures \((X_1, J_1)\) and \((X_2, J_2)\), we say that a function \(f:X_1 \to X_2\) is \textbf{pseudoholomorphic} if \(f_* \circ J_1 = J_2 \circ f_*\).

Let \((X, J, g, \omega)\) be a \((2n)\)-dimensional manifold equipped with an almost complex structure \(J\), a Riemannian metric \(g\), and a symplectic form \(\omega\).   We say that the triple \((J, g, \omega)\) is \textbf{tame} if \(g\) is complete and there exists constants \(r_0, C_1, C_2\) such that:

\begin{itemize}

\item Every loop \(\gamma \subset X\) contained in a ball \(B_r(x)\) with \(r \leq r_0\) bounds a disc in \(B\) of area less then \(C_1(\mbox{length}(\gamma))^2\).

\item \(||\omega_x||_g \leq 1\) for all \(x \in X\)

\item For every vector \(V \in T_xX\), \(|X|^2 \leq C_2\omega(X, JX)\).

\end{itemize}

We can always equip \(T^*M\) with a tame triple \((J, g, \omega)\) (\cite{AL}, Ch. 5, Sec. 4.1), and from now on will assume we have done so.   Furthermore, the space of almost complex structures \(J\) on \(T^*M\) which is part of a tame triple is contractible (\cite{MS}, Proposition 4.1.)

\subsection{Definition of the DGA}
\label{subsec:dga}

We say that a Legendrian submanifold is \textbf{front generic} if it has the following properties:

\begin{itemize}

\item The base space projection \(\pi_M : \Lambda \to M\) is an immersion outside of a codimension-1 submanifold \(\Sigma_1\).

\item We define \(\Sigma_k\) inductively to be a codimension-1 subset of \(\Sigma_{k-1}\), such that the map \(\pi_M : (\Sigma_{k-1} - \Sigma_k) \to M\) is an immersion.

\item At points \(s \in \Sigma_1 - \Sigma_2\), \(\pi_F:\Lambda \to M \times \mathbb{R}\) has a standard \textbf{cusp edge singularity}.   That is, there exist coordinates \(u_1, ..., u_n\) of \(\Lambda\) around \(s\) and coordinates \(x_1, ..., x_n\) of \(M\) around \(\pi_M(s)\) such that, if \(z\) is the fiber coordinate in \(J^0(M) = M \times \mathbb{R}\), then \(\pi_F(u) = (x_1(u), ..., x_n(u), z(u))\), where:
\[
x_1(u) = \frac{1}{2}u_1^2
\]\[
x_j(u) = u_j\mbox{ for }j = 2, ..., n
\]\[
z(u) = \frac{1}{3}u_1^3 + \beta \frac{1}{2}u_1^2 + \alpha_2u_2 + ... + \alpha_nu_n
\]
Where \(\alpha_2, ..., \alpha_n, \beta\) are constants.

\end{itemize}

As discussed in \cite{Ek}, Section 2.2.1, any Legendrian submanifold \(\Lambda\) can be made front generic after an arbitrarily small Legendrian isotopy.    We will therefore assume from here on that our Legendrian submanifolds are all front generic.

A \textbf{Reeb vector field} \(R_\alpha\) is the vector field of \(J^1(M)\) such that:
\[
\alpha(R_\alpha) = 1
\]\[
(d\alpha)(R_\alpha) = 0
\]
For \(J^1(M)\), \(R_\alpha = \partial_z\).   A \textbf{Reeb chord} of a Legendrian submanifold \(\Lambda\) is a trajectory of \(R_\alpha\) that begins and ends on \(\Lambda\).   For \(R_\alpha = \partial_z\), Reeb chords correspond precisely to self-intersections of \(\pi_\mathbb{C}(\Lambda)\).   These intersections will generically be transverse double points.   We label the quadrants of each Reeb chord in \(\pi_\mathbb{C}(\Lambda)\) with the signs shown in Figure~\ref{fig:signsOfQuadrants}.

\begin{figure}\begin{center}
\includegraphics{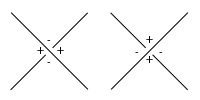}
\caption{Labels of Quadrants of \(\pi_\mathbb{C}(\Lambda)\)\label{fig:signsOfQuadrants}}\end{center}
\end{figure}

We define \(\mathcal{A}_K(\Lambda)\) to be the formal unital algebra over \(K\), where \(K = \mathbb{Z}_2\) or \(\mathbb{Z}_2H_1(\Lambda)\), generated by the Reeb chords of \(\Lambda\) - that is, elements of \(\mathcal{A}_K(\Lambda)\) are sentences made of words \(kb_1...b_m\), where \(k \in K\) and \(b_1, ..., b_m\) are Reeb chords.   We provide a differential to \(\mathcal{A}_K(\Lambda)\) by counting rigid pseudoholomorphic disks in \(T^*M\).

First, for every Reeb chord \(b\) of \(\Lambda\), let \(b^+, b^-\) be the upper and lower end points of the Reeb chord.   We choose a \textbf{capping path} for each chord \(b\), which is a path \(\gamma_b:[0, 1] \to \Lambda\) such that \(\gamma(0) = b^+, \gamma(1) = b^-\).   We define \(D_r\) to be the unit disk in \(\mathbb{C}\) with the canonical complex structure, with \(r\) punctures on its boundary labeled \(q_0, ..., q_{r-1}\).   Then, if \(a, b_1, ..., b_m\) are Reeb chords and \(A\) is an element in \(H_1(\Lambda)\), we define the \textbf{moduli space} \(\mathcal{M}_A(a; b_1...b_m)\) to be the space of maps \(u:(D_{m+1}, \partial D_{m+1}) \to (T^*M, \pi_\mathbb{C}(\Lambda))\) such that:

\begin{itemize}

\item \(u\) is pseudoholomorphic.

\item As \(z\) limits to \(q_0\), \(u(z)\) limits to a positive quadrant of \(a\).

\item As \(z\) limits to \(q_i\), \(i > 0\), \(u(z)\) limits to a negative quadrant of \(b_i\).

\item The restriction of \(u\) to \(\partial D_{m+1}\) has a continuous lift \(\tilde{u}:\partial D_{m+1} \to \Lambda \subset J^1(M)\).

\item The homology class of \(\tilde{u}(\partial D_{m+1} \cup \gamma_{a} \cup \gamma_{b_1} \cup ... \gamma_{b_m}\) equals \(A \in H_1(\Lambda)\).

\end{itemize}

We define \(\mathcal{M}(a; b_1...b_m)\) in the same fashion, except that we drop the requirement on the homology class.

These moduli spaces are open manifolds with boundary.   We define the differential \(\partial\) of \(\mathcal{A}_K(\Lambda)\) by counting 0-dimensional moduli spaces.   If \(K = \mathbb{Z}_2\), then:
\[
\partial a = \sum_{\dim \mathcal{M}(a; b_1...b_m) = 0} \left(\#\mathcal{M}(a; b_1...b_m)\right)b_1...b_m
\]
If \(K = \mathbb{Z}_2H_1(\Lambda)\), then:
\[
\partial a = \sum_{\dim \mathcal{M}_A(a; b_1...b_m) = 0} \left(\#\mathcal{M}_A(a; b_1...b_m)\right)Ab_1...b_m
\]
We extend these differentials to \(\partial:\mathcal{A}_K(\Lambda) \to \mathcal{A}_K(\Lambda)\) using the linearity and the product rules:
\[
\partial (x + y) = (\partial x) + (\partial y)
\]\[
\partial (xy) = (\partial x)y + x(\partial y)
\]
We say that a pseudoholomorphic disk in a 0-dimensional moduli space is \textbf{rigid}.   Note that we could have zero negative punctures of a rigid pseudoholomorphic disk, in which case the disk contributes 1 to the differential.   And:

\begin{theorem}\label{theorem2.1} \(\partial^2 = 0\), so the homology of \(\mathcal{A}_K(\Lambda)\) is well-defined.\end{theorem}

\begin{proof} See \cite{EES2}, Proposition 2.6.\end{proof}

The differential graded algebra \(\mathcal{A}_K(\Lambda)\) is \emph{not} an invariant of the Legendrian isotopy class of \(\Lambda\).   However, its stable tame isomorphism class \emph{is}.   A \textbf{stabilization} of \(\mathcal{A}_K(\Lambda)\), which we denote \(S(\mathcal{A}_K(\Lambda))\), is a differential graded algebra whose generators consist of the generators of \(\mathcal{A}_K(\Lambda)\), plus two generators \(b, c\), and whose differential \(\partial_S\) is defined to be:
\[
\partial_S x = \partial x\mbox{ if }x \in \mathcal{A}_K(\Lambda)
\]\[
\partial_S b = c
\]\[
\partial_S c = 0
\]

An \textbf{elementary automorphism} of a differential graded algebra generated by \(a_1, ..., a_n\) is an automorphism such that there exists some \(j\) so that \(a_i \to a_i\) for \(i \neq j\) and \(a_j \to a_j + u\), where \(u\) is a word in \(a_1, ..., a_{j-1}, a_{j+1}, ..., a_n\).   A \textbf{tame automorphism} is an automorphism that is the composition of a series of elementary automorphisms.   A \textbf{tame isomorphism} between a differential graded algebra generated by \(a_1, ..., a_n\) and a differential graded algebra generated by \(b_1, ..., b_n\) is the composition of a tame automorphism with the isomorphism sending \(a_1 \to b_1, ..., a_n \to b_n\).

We say that two differential graded algebras \(\mathcal{A}_K(\Lambda_1), \mathcal{A}_K(\Lambda_2)\) are \textbf{stable tame isomorphic} if they are tame isomorphic after some series of stabilizations.   And:

\begin{theorem}\label{theorem2.2} The stable tame isomorphism class of \(\mathcal{A}_K(\Lambda)\) is invariant under Legendrian isotopy.\end{theorem}

\begin{proof} See \cite{EES2}, Proposition 2.7.\end{proof}

The homology of \(\mathcal{A}_K(\Lambda)\) is invariant under stable tame isomorphism, and so is in turn an invariant of the Legendrian isotopy class.   This homology is called the Legendrian Contact Homology of \(\Lambda\).

When \(\dim\Lambda = 1\), it is easy to find these pseudoholomorphic disks, because, thanks to the Riemann Mapping Theorem, they correspond to polygons in \(T^*M\) whose vertices are Reeb chords and whose edges lie on \(\pi_\mathbb{C}(\Lambda)\).   It is significantly harder to calculate the differential when \(\dim \Lambda > 1\).   Fortunately, as we discuss in the next subsection, \cite{Ek} provides a technique for calculating the differential in a very broad special case.

\subsection{Gradient Flow Trees}
\label{subsec:flowtrees}

\cite{Ek} provides a technique for calculating \(\partial\) provided that either \(\dim \Lambda = 1\) or 2 or the front projection \(\pi_F\) has no codimension-2 or higher singularities that are not cusp edge intersections.   This is a large and very interesting class of Legendrian submanifolds.   He does this using an object called a gradient flow tree.

Consider the front projection of a Legendrian submanifold.   We can regard this as the graph of some collection of \textbf{height functions} \(f_i:U_i \to \mathbb{R}, U_i \subset M\).   A path \(\gamma:[0,1] \to M\) is a \textbf{flow line} of the height functions \(f_1, f_2\) if \(\gamma'(t) = -\nabla(f_1 - f_2)(\gamma(t))\).   If \(\gamma:[0,1]\to M\) is a flow line, the \textbf{1-jet lift} of \(\gamma\) is an unordered pair \(\{\tilde{\gamma}^1, \tilde{\gamma}^2\}:[0, 1] \to \Lambda \subset J^1(M)\) of continuous paths such that \(\pi_F(\tilde{\gamma}_i(t)) = (\gamma(t), f_i(\gamma(t)))\).   We define the \textbf{flow orientation} of \(\tilde{\gamma}_1\) at a point \(p \in M\) to be the unique lift of the vector \(-\nabla(f_1 - f_2)(\pi_M(p)) \in T_{\pi_M(p)}M\) to \(T_p\Lambda\).

Let \(\Gamma\) be a tree with finitely many edges.   For \(k \geq 2\), at each \(k\)-valent vertex, place a cyclic ordering on the edges.   A map \(F:\Gamma \to M\) is a \textbf{gradient flow tree} if it satisfies the following conditions:

\begin{itemize}

\item If \(e\) is an edge of \(\Gamma\) then the restriction of \(F\) to \(e\) is a flow line.

\item If \(v\) is a \(k\)-valent vertex with ordered edges \(e_1, ..., e_k\), and \(\tilde{\gamma}^1_j, \tilde{\gamma}^2_j\) are the 1-jet lifts of \(F\) restricted to \(e_j\), then we require that \(\pi_\mathbb{C}(\tilde{\gamma}^2_j(v)) = \pi_\mathbb{C}(\tilde{\gamma}^1_{j+1}(v))\), with the convention \(\gamma_{k+1} = \gamma_1\).   In addition, we require that the flow orientation of \(\gamma^2_j\) at \(v\) is directed towards \(v\) if and only if the flow orientation of \(\gamma^1_{j+1}\) is directed away.

\item The Lagrangian projection of the 1-jet lifts of the edges of \(\Gamma\) fit together to form a closed oriented curve \(\bar{\Gamma}\) in \(\pi_\mathbb{C}(\Lambda)\).

\end{itemize}

Let \(\mathcal{T}(a; b_1...b_m)\) be the space of gradient flow trees with a positive puncture at \(a\) and negative punctures at \(b_1, ..., b_m\).   Let \(\tilde{\Gamma}\) be the lift of \(\bar{\Gamma}\) to \(\Lambda\), and let \(\mathcal{T}_A(a; b_1...b_m)\) be the restriction of \(\mathcal{T}(a; b_1...b_m)\) to trees such that the homology class of \(\tilde{\Gamma} \cup \gamma_a \cup \gamma_{b_1} \cup ... \cup \gamma_{b_m}\) equals \(A\) in \(H_1(\Lambda)\).

\begin{theorem}\label{theorem2.3} If \(\dim \Lambda = 1\) or 2 or if the only singularities of \(\pi_F\) are cusp edges, then there exists an almost complex structure such that there is a bijection between \(\mathcal{T}(a; b_1...b_m)\) and \(\mathcal{M}(a; b_1...b_m)\), and between \(\mathcal{T}_A(a; b_1...b_m)\) and \(\mathcal{M}_A(a; b_1...b_m)\), if the disks are rigid.\end{theorem}

\begin{proof} See \cite{Ek}, Theorem 1.1.\end{proof}

For a Reeb chord \(p\), we will use \(I(p)\) to denote the Morse index of the height difference function at \(p\).   If \(I(p) = 0\) we say \(p\) is a \textbf{minimum}, and if \(I(p) = \dim \Lambda\) we say \(p\) is a \textbf{maximum}.

The vertices of a gradient flow tree may correspond to Reeb chords or they may correspond to other joinings of flow lines.   We label the following kinds of non-chord vertices:

\begin{itemize}

\item \textbf{End:} One-valent vertex where a flow line between two sheets meets a cusp edge of that sheet.

\item \textbf{Switches:} Two-valent vertex where a flow line between sheets \(i\) and \(j\), where \(i\) lies above \(j\), meets a cusp edge between sheets \(i\) and \(k\) (or \(k\) and \(j\)), and leaves as a flow line between sheets \(k\) and \(j\) (or \(i\) and \(k\)).

\item \textbf{\(Y^0\) vertices:} Three-valent vertex where a flow line between sheets \(i\) and \(j\) splits into two flow lines, one between sheets \(i\) and \(k\) and one between sheets \(k\) and \(j\), where the \(z\) coordinate of sheet \(k\) lies between the \(z\) coordinates of sheets \(i\) and \(j\).

\item \textbf{\(Y^1\) vertices:} Three-valent vertex where a flow line between sheets \(i\) and \(j\) meets a cusp edge between sheets \(k\) and \(l\), where \(i\) lies above \(k\) and \(l\) lies above \(j\), and leaves as two flow lines, one between sheets \(i\) and \(l\) and one between sheets \(k\) and \(j\).

\end{itemize}

\begin{figure}\begin{center}
\includegraphics{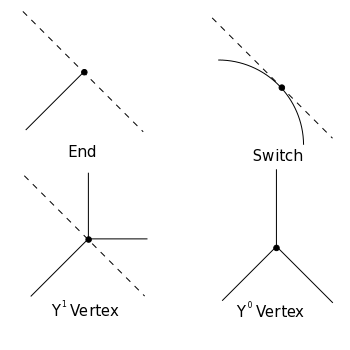}
\caption{Non-Chord Vertices of Gradient Flow Trees\label{fig:picturesOfVertices}}
(\emph{Solid lines are gradient flow lines, dotted lines are cusp edges})\end{center}
\end{figure}

\begin{theorem}\label{theorem2.4} Let \(\Gamma\) be a gradient flow tree with positive Reeb chords \(p_1, ..., p_m\), negative Reeb chords \(q_1, ..., q_r\), and non-chord vertices \(v_1, ..., v_s\).   The dimension of the moduli space containing \(\Gamma\) is then:
\begin{equation}
\label{eqn:GFTdimFormula}\dim \Gamma = (n - 3) + \sum_{i=1}^m \left(I(p_i) - (n - 1)\right) - \sum_{j=1}^r \left(I(q_j) - 1\right) + \sum_{k=1}^s \mu(v_k)
\end{equation}

Where \(I(x)\) is the Morse index of \(x\) if \(x\) is a Reeb chord, \(I(x) = n + 1\) if \(x\) is a positive special puncture, \(I(x) = -1\) if \(x\) is a negative special puncture, \(\mu(v_k) = 1\) if \(v_k\) is an end, 0 if \(v_k\) is a \(Y^0\) vertex, and -1 if \(v_k\) is a switch or \(Y^1\) vertex.\end{theorem}

\begin{proof} See \cite{Ek}, Prop. 3.14.\end{proof}

\begin{theorem}\label{theorem2.5} Any vertex of a rigid gradient flow tree with one positive puncture must be a one-valent vertex at a positive or negative Reeb chord; a two-valent vertex at a positive minimum or a negative maximum; a cusp edge; a switch; a \(Y^0\) vertex; or a \(Y^1\) vertex.\end{theorem}

\begin{proof} See \cite{Ek}, Lemma 3.7.\end{proof}

We will also make use of objects called \textbf{partial flow trees}, which are gradient flow trees that have one or more \textbf{special punctures}, which are 1-valent vertices that are neither cusp edges nor Reeb chords.   A special puncture may be considered to be either positive or negative, depending on the orientation of the 1-jet lift.   If \(v\) is a special puncture, then we treat \(I(v) = n + 1\) if \(v\) is positive, \(I(v) = -1\) if \(v\) is negative.   Critically, given any gradient flow tree \(\Gamma\), and any point \(p\) in a flow line (not a vertex) of that tree, we can separate \(\Gamma\) into a pair of partial flow trees, \(\Gamma_1\) and \(\Gamma_2\), by breaking \(\Gamma\) at \(p\).   Then:
\[
\dim \Gamma = \dim \Gamma_1 + \dim \Gamma_2 - (n + 1)
\]

The advantage of gradient flow trees is that they reduce finding the rigid pseudoholomorphic disks from an infinite-dimensional problem to a finite-dimensional problem.   Although we cannot always use them, given the requirements on \(\pi_F\), they are a very valuable tool when they are useable.   The essence of the proof of theorems 1.1 through 1.5 will be to show that, although we cannot prove that pseudoholomorphic disks are globally equivalent to gradient flow trees if \(\pi_F\) has higher-codimension singularities, we can show they are \emph{locally} equivalent.

%
%
%
%
%
%
%

\subsection{Augmentations and Linearized Contact Homology}
\label{subsec:linearized}

The Legendrian Contact Homology of a Legendrian submanifold \(\Lambda \subset J^1(M)\) is, generally, infinite-dimensional, and often quite difficult to work with.   We therefore define algebraic invariants of the Legendrian Contact Homology which will be finite-dimensional.

First, we split the differential \(\partial\) into components \(\partial_0, \partial_1, \partial_2\), etc., where \(\partial_k\) denotes the restriction of \(\partial\) to words of length \(k\).   Since \(\partial \circ \partial = 0\), we obtain:
\[
\partial_0 \circ \partial_0 = 0
\]\[
(\partial_0 \circ \partial_1) + (\partial_1 \circ \partial_0) = 0
\]\[
(\partial_0 \circ \partial_2) + (\partial_1 \circ \partial_1) + (\partial_2 \circ \partial_0) = 0
\]
Etc.   This means that, if \(\partial_0 = 0\), then \(\partial_1 \circ \partial _1 = 0\), which means that we can define a \textbf{linearized chain complex} given by the restriction of \(\mathcal{A}(\Lambda)\) to words of unit length, and equipped with the differential \(\partial_1\).   The homology of this algebra will then be finite-dimensional.   Unfortunately, \(\partial_0 \neq 0\) in general.   However, \(\mathcal{A}(\Lambda)\) may be stably tame isomorphic to a DGA where \(\partial_0\) \emph{does} equal zero.

Finding such an isomorphism is equivalent to finding a graded algebra map \(\epsilon:\mathcal{A}(\Lambda) \to \mathbb{Z}_2\) such that \(\epsilon(1) = 1\) and \(\epsilon \circ \partial = 0\).   We call such a map an \textbf{augmentation}.   Given an augmentation, we define a DGA map by \(H_\epsilon(x) = x + \epsilon(x)\) if \(x\) is a Reeb chord, and extend to products multiplicatively - that is:
\[
H_\epsilon(xy) = H_\epsilon(x)H_\epsilon(y)
\]
\(H_\epsilon\) is a stable tame isomorphism to the differential graded algebra with the same generators as \(\mathcal{A}(\Lambda)\), but with differential \(\partial_\epsilon = H_\epsilon \circ \partial \circ H_\epsilon^{-1}\).   \((\partial_\epsilon)_0\) will always equal zero, so we can define a linearized chain complex for this augmented algebra.   The homology of this complex is called the \textbf{linearized contact homology}, or \(LCH_*(\mathcal{A}(\Lambda), \partial, \epsilon)\).   Since \(\epsilon\) is not necessarily unique, a given Legendrian may have multiple linearized contact homologies, and two different maps \(\epsilon\) may produce the same linearized contact homology.   However:

\begin{theorem}\label{theorem2.6} The set of linearized contact homologies of \(\Lambda\) is an invariant of the stable tame isomorphism class of \(\mathcal{A}(\Lambda)\), and therefore is an invariant of the Legendrian isotopy class of \(\Lambda\).\end{theorem}

\begin{proof} See \cite{Che}, Theorem 5.2.\end{proof} 

\subsection{Cobordism Maps}
\label{subsec:cobordismmaps}

Let \(\mathbb{R}^+ = (0, \infty)\).   The symplectization \((J^1(M) \times \mathbb{R}, d(e^t\alpha))\) is symplectomorphic to the cotangent bundle with canonical symplectic form \((T^*(M \times \mathbb{R}^+), \omega)\) by the map:
\[
F(x_1, ..., x_n, y_1, ..., y_n, z, t) = (x_1, ..., x_n, e^t, e^ty_1, ..., e^ty_n, z)
\]
Although this is not an \emph{exact} symplectomorphism, the difference between the primitive of the symplectic form of \(J^1(M) \times \mathbb{R}\) and the tautological one-form of \(T^*(M\times\mathbb{R}^+)\) \emph{is} exact:
\[
F^*\left(-\sum_{i=1}^n y_idx_i - zdt\right) = -\sum_{i=1}^ne^ty_i dx_i - e^t z dt = e^t\alpha - d(e^tz)
\]
Therefore, if {\color{black}\((e^t\alpha)|_{TL} = df\)}, then \(\theta|_{F(L)} = d(f + tz)\), where \(\theta\) is the tautological one-form of \(T^*(M \times \mathbb{R}^+)\).   So \(F(L)\) is still an exact Lagrangian cobordism.   We will use \(L\) to refer to \(F(L)\) from here on, and work exclusively in \(T^*(M \times \mathbb{R}^+)\).   We prefer this environment because it is the native environment for our pseudoholomorphic disks.

An exact Lagrangian cobordism \(L\) from \(\Lambda_+\) to \(\Lambda_-\) induces a DGA map \(\Phi_L:\mathcal{A}(\Lambda_+) \to \mathcal{A}(\Lambda_-)\), given by counting rigid pseudoholomorphic disks whose punctures limit to the Reeb chords at infinity.   That is, if \(a\) is a Reeb chord in \(\Lambda_+\) and \(b_1, ..., b_m\) are Reeb chords in \(\Lambda_-\), we define the moduli space \(\widetilde{\mathcal{M}}(a; b_1...b_m)\) to be the space of maps \(u:(D_{m+1}, \partial D_{m+1}) \to (T^*(M\times\mathbb{R}^+), L))\) such that:

\begin{itemize}

\item \(u\) is pseudoholomorphic, that is, \(du + J \circ du \circ J_{\mathbb{C}} = 0\), where \(J_\mathbb{C}\) is the almost complex structure on the domain of \(u\).

\item \(u(p)\) limits to \(a\times\{\infty\}\) as \(p\to q_0\), where \(a\) is a Reeb chord of \(\Lambda_+\).

\item \(u(p)\) limits to \(b_i\times\{-\infty\}\) as \(p \to q_i\), where \(b_i\) is a Reeb chord of \(\Lambda_-\).

\end{itemize}

We then define \(\Phi_L\) by:
\[
\Phi_L(a) = \sum_{\dim\widetilde{\mathcal{M}}(a; b_1...b_m) = 0} \left(\#\widetilde{\mathcal{M}}(a; b_1...b_m)\right)b_1...b_m
\]
Let \(\partial_\pm\) denote the differential of \(\Lambda_\pm\).   Then:

\begin{theorem}\label{theorem2.7} \(\Phi_L \circ \partial_+ = \partial_- \circ \Phi_L\)\end{theorem}

\begin{proof} See \cite{EKH}, Theorem 1.2.\end{proof}

\begin{theorem}\label{theorem2.8} If \(L_1\) and \(L_2\) are exact Lagrangian isotopic, then \(\Phi_{L_1}\) is chain homotopic to \(\Phi_{L_2}\).\end{theorem}

\begin{proof} See \cite{EKH}, Theorem 1.3.\end{proof}

Note that, given an exact Lagrangian cobordism \(L_{12}\) from \(\Lambda_1\) to \(\Lambda_2\), and a second exact Lagrangian cobordism \(L_{23}\) from \(\Lambda_2\) to \(\Lambda_3\), we can form a new exact Lagrangian cobordism \(L_{13}\) from \(\Lambda_1\) to \(\Lambda_3\) by concatenating \(L_{12}\) and \(L_{23}\): we delete the negative cone of \(L_{12}\) and positive cone of \(L_{23}\) and glue them together.   Then:

\begin{theorem}\label{theorem2.9} \(\Phi_{L_{13}} = \Phi_{L_{23}} \circ \Phi_{L_{12}}\)\end{theorem}

\begin{proof} See \cite{EKH}, Theorem 1.2.\end{proof}

\subsection{Lifting Cobordisms}
\label{subsec:liftingcobordisms}

This section is derived from \cite{EKH}, Section 2, with minor modifications to adapt the concept to higher dimensions - principally adding extra coordinates as appropriate.

Recall that, according to the definition of an exact Lagrangian cobordism \(L \subset T^*(M\times \mathbb{R}^+)\), there exists a function \(f:L \to \mathbb{R}\) such that {\color{black}\(\beta|_{TL} = d(f + tz)\)}.   We can lift \(L \subset T^*(M\times\mathbb{R}^+)\) to \(\hat{L} \subset J^1(M \times \mathbb{R}^+)\), with \(\hat{L}\) unique up to \(\mathbb{R}\)-translation, by mapping \(p \in L\) to \((p, f(p) + t(p)z(p)) \in \hat{L}\).   Let \(\hat{\alpha}\) be the canonical contact form of \(J^1(M\times\mathbb{R}^+)\); then:
\[
\left.\hat{\alpha}\right|_{\hat{L}} = d(f + tz) - \theta = \theta - \theta = 0
\]
So \(\hat{L}\) is Legendrian, though it is definitely \emph{not} compact.   However, as we will see in section~\ref{sec:proof}, we do not \emph{need} compactness.   We will prove convergence of the boundary of pseudoholomorphic curves to gradient flow lines in \(\Lambda\), where \(\Lambda\) may denote a compact Legendrian submanifold of \(J^1(M)\) or it may denote the Legendrian lift \(\hat{L} \subset J^1(M\times\mathbb{R})^+\) of an exact Lagrangian cobordism, and where \(M\) may mean \(M\) or may mean \(M \times \mathbb{R}^+\).

This does create some confusion between the \(z\) coordinate of \(\Lambda_\pm\) and the new \(z\) coordinate of \(\hat{L}\) given by integrating \(d(f + tz)\).   In section~\ref{sec:proof}, when we refer to the \(z\) coordinate of \(\Lambda\), we will always mean the \(\mathbb{R}\) coordinate of the one-jet space, whether that one-jet space happens to be \(J^1(M)\) or \(J^1(M\times\mathbb{R}^+)\).

\section{Critical Legendrian Ambient Surgery}
\label{sec:surgery}

Legendrian Ambient Surgery is defined in \cite{Ri}, Sec. 3, as an analogue of Morse surgery for Legendrians.  We begin with the following data: let \(\Lambda \subset J^1(M)\) be an \(n\)-dimensional Legendrian submanifold, and let \(k < n\).   Let \(S \subset \Lambda\) be an embedded \(k\)-sphere with a choice of framing \(F\) of its normal bundle \(NS \subset T\Lambda\).  We define the \textbf{symplectic normal bundle of \(NS\)} over \(S\) to be:
\[
NS^{(d\alpha)} = \left\{v \in \mbox{Ker }\alpha | (d\alpha)(v, NS) = \{0\}\right\}
\]

Let \(D_S \subset J^1(M)\) be an isotropic embedded \((k+1)\)-disk that satisfies the following properties:

\begin{itemize}

\item \(\partial D_S = S\)

\item \(\mbox{Int } D_S \subset J^1(M) - \Lambda\)

\item Any outward normal vector field to \(D_S\) is nowhere tangent to \(\Lambda\)

\item For any vector field \(H\) in \(NS\) such that \((d\alpha)(G, H) > 0\) for any vector \(G\) that is outward normal to \(D_S\), we require that the frame given by adjoining \(H\) to the Lagrangian frame of \((TD_S)^{d\alpha}|_S\) is homotopic to \(F\).

\end{itemize}

\cite{Ri} uses the above data to construct a standard model of a neighborhood of \(S \subset \Lambda\), then performs the surgery on this standard model, obtaining a surgered Legendrian submanifold \(\hat{\Lambda}\).  He calculates the change in \(\mathcal{A}(\Lambda)\) for \(k < n - 1\), but is unable to calculate the change for \(k = n - 1\).  Using pinching techniques, we can calculate the change in the algebra for \(k = n - 1\), using a modified version of the standard model, which is Legendrian isotopic to [Ri]'s version.

\subsection{Preliminary Morse Lemmas} 

We begin with some preliminary Morse lemmas, which will be used in constructing our local neighborhood model.   First, we define:
\[
\mathcal{F}(f_1, \delta) = \left\{f_2:M\to\mathbb{R}\mbox{ }|\mbox{ }f_2\mbox{ is Morse}, |f_1 - f_2|_{C^1} < \delta\right\}
\]

We then have the following Morse lemmas:

\begin{lemma}\label{lemma3.1} Let \(f_1: M \to \mathbb{R}\) be a Morse function, where \(M\) is a compact manifold.   For any \(\epsilon > 0\), there exists \(\delta > 0\) such that, for any \(f_2 \in \mathcal{F}(f_1, \delta)\), there is a bijection between the critical points of \(f_1\) and \(f_2\), the ascending manifold of every critical point of \(f_2\) lies within \(\epsilon\) of the ascending manifold for the corresponding critical point of \(f_1\), and the descending manifold of every critical point of \(f_2\) lies within \(\epsilon\) of the descending manifold of the corresponding critical point of \(f_1\).\end{lemma}

\begin{proof} See Appendix~\ref{sec:MorseLemmas}.\end{proof}

\begin{lemma}\label{lemma3.2} Let \(f:M \to \mathbb{R}\) be a Morse function, where \(M\) is a closed manifold.   Let \(Q\) be a compact codimension-0 subset of \(M\) that includes no critical points of \(f\).   Then for any \(\epsilon > 0\) there exists \(\delta > 0\) such that for any critical point \(q\) and any points \(p_1, p_2 \in Q\), if:
\[
d(p_1, p_2) < \delta
\mbox{, and}
\]\[
p_1, p_2\mbox{ lie in the same component of }\mathcal{A}_f(q)
\]
Then:
\[
d(\mathcal{D}_f(p_1), \mathcal{D}_f(p_2)) < \epsilon
\]
And, for any \(\epsilon > 0\), there exists \(\delta > 0\) such that, for any points \(p_1, p_2 \in Q\), if
\[
d(p_1, p_2) < \delta
\mbox{, and}
\]\[
p_1, p_2\mbox{ lie in the same component of }\mathcal{D}_f(q)
\]
Then:
\[
d(\mathcal{A}_f(p_1), \mathcal{A}_f(p_2)) < \epsilon
\]\end{lemma}

\begin{proof} See Appendix~\ref{sec:MorseLemmas}.\end{proof}

\begin{lemma}\label{lemma3.3} Let \(f_1:M \to \mathbb{R}\) be a Morse function, where \(M\) is a compact manifold.   For any choice of \(\epsilon > 0\) and any compact codimension-0 submanifold \(Q \subset M\) that lies in the descending manifold of maxima of \(f_1\) and contains no critical points of \(f_1\), there exists \(\delta\) such that, for any generic choice of \(f_2 \in \mathcal{F}(f_1, \delta)\) and any \(p \in Q\), the ascending manifold of \(p\) for \(\nabla f_2\) will lie within an \(\epsilon\)-neighborhood of the ascending manifold of \(p\) for \(\nabla f_1\).\end{lemma}

\begin{proof} See Appendix~\ref{sec:MorseLemmas}.\end{proof}

\begin{lemma}\label{lemma3.4} Let \(\Lambda \subset J^1(M)\) be a front-generic Legendrian submanifold whose front projection is defined by sheet functions \(f_1:U_1 \to \mathbb{R}, ..., f_m:U_m \to \mathbb{R}\), where \(U_1, ..., U_m \subset M\).   Let \(\hat{\Lambda} \subset J^1(M)\) be a second front-generic Legendrian submanifold whose front projection is defined by sheet functions \(\hat{f}_1:U_1 \to \mathbb{R}, ..., \hat{f}_m:U_m \to \mathbb{R}\).   Then, for any choice of \(\epsilon > 0\), there exists \(\delta > 0\) such that, if:
\[
\left|\left| f_i - \hat{f}_i\right|\right|_{C^1} < \delta\mbox{ for all }i
\]
Then, there exists a bijection between the rigid gradient flow trees \(\Lambda\) with one positive Reeb chord and the rigid gradient flow trees of \(\hat{\Lambda}\) with one positive Reeb chord, such that a tree \(\Gamma\) of \(\Lambda\) shares the same Reeb chords with its corresponding tree \(\hat{\Gamma}\) of \(\hat{\Lambda}\), and such that the projection of \(\hat{\Gamma}\) to \(M\) lies within an \(\epsilon\)-neighborhood of the projection of \(\Gamma\) to \(M\).\end{lemma}

\begin{proof} See Appendix~\ref{sec:MorseLemmas}.\end{proof}

\subsection{Setup} 

In what follows, most of our Reeb chords will appear in clusters, corresponding to critical points of a preliminary function, which we scale and then perturb by an arbitrarily small amount to produce the functions we will actually use to define our Legendrian submanifold. With one exception, discussed later, our Reeb chord clusters will contain a Reeb chord between every pair of sheets. We will use ordinary letters \(x\) to denote a Reeb chord cluster prior to the surgery, and the superscript \(\hat{x}\) to denote the cluster after the surgery. We will use \(x^i_j\) to denote the Reeb chord within cluster \(x\) that lies between sheets \(i\) and \(j\) before the surgery, and \(\hat{x}_j^i\) to denote the chord after the surgery. We will also have various preliminary functions denote by superscripts \(\bar{f}\), \(\tilde{f}\), etc., and we will use \(\bar{x}, \tilde{x}\), etc. to denote the clusters of Reeb chords for these preliminary functions.

If a sheet in the front projection \(\pi_F(\Lambda)\) crosses \(\pi_F(D_S)\), we can perform a Legendrian isotopy to contract \(D_S\) until this is no longer the case. We can do this because, since \(S \subset \Sigma_1\), in our Legendrian isotopy we can pass \(S\) through any intervening sheet in the front projection. Therefore, we can assume that, after a Legendrian isotopy, there exists local coordinates \(x\) in the base space \(M\) such that \(\pi_M(D_S)\) is contained in a neighborhood \(D_4\), where:
\[
D_r = \{x_1^2 + ... + x_n^2 \leq r^2\}
\]
And \(\pi_F(D_S)\) has the form:
\[
\pi_F(D_S) = \{x_1^2 + ... + x_n^2 \leq 4 \mbox{ and } z = z_0\}
\]
And that, after a Legendrian isotopy, the sheets in the front projection that lie over a neighborhood \(U\) of \(\pi_M(D_S)\), and which do not form the cusps containing \(S\), are graphs of functions \(g_i\) such that:
\[
g_i(x) = x_1v_{1i} + ... + x_nv_{ni}
\]\[
g_i(x) \neq z_0\mbox{ for }x\in U
\]

Label the sheets in \(\pi_F(D_S)\) with (possibly negative) integers by \(z\) order, so that sheet 0 and 1 are the sheets containing \(S\).  We have a natural projection \(\pi_S:(D_r - \{0\}) \to S^n\).   We perform a Legendrian isotopy so that, over \(D_3\), the front projection of sheet \(i\) is the graph of a sheet function \(f_i\) which will be specified below, and over \(D_4 - D_3\) we have an interpolation between those sheet functions and the original form of the Legendrian submanifold.   We will use \(\hat{f}_i\) to denote the sheet functions after the surgery.  Note that, for \(i \neq 0, 1\), \(\hat{f}_i = f_i\).

We begin by defining a preliminary function \(\tilde{f}(x)\):

\begin{itemize}

\item \(\tilde{f}(x)\) has a maximum at \(p = (2, 0, ..., 0)\), with the value \(\tilde{f}(p) = 2\).

\item Let \(K\) be a 0.01-neighborhood of \(p\).  \(\tilde{f}\) has no other critical points in \(K\).

\item Over \(D_{2.51} - D_{2.49}\), \(\tilde{f}\) has the value:
\[
\tilde{f}(x) = \left(
(|x| - 2.5)^2 + 0.01
\right) + 0.0001s(\pi_S(x))
\]
where \(\pi_S\) is the projection to \(S^n\) and \(s:S^n \to [0,1]\) is a function that has a maximum at the north pole, a minimum at the south pole, and no other critical points.

\item Over \(D_{2.48} - K\), \(\tilde{f}\) has the value:
\[
\tilde{f}(x) = 2\left(\frac{l(x) - |p - x|}{l(x)}\right) + 0.01\left(\frac{|p - x|}{l(x)}\right)
\]
Where \(l(x)\) denotes the length of the straight line from \(p\) to \(\partial D_{2.5}\) through \(x\).

\item \(\tilde{f}\) has no other critical points in \(D_3\).

\end{itemize}

Define:
\[
M_i = \left\{\begin{array}{ll} 2^i&i > 1\\
0 &i = 1\\
0&i=0\\
-2^{-i}&i < 0\end{array}\right\}
\]
Then define further preliminary functions \(\tilde{f}_i(x) = M_i\tilde{f}(x)\).   (It is intentional that \(\tilde{f}_1 = \tilde{f}_0 = 0\) at this stage.)   If we consider the graphs of \(\tilde{f}_i\) as the front projection of some Legendrian submanifold, this submanifold is not generic, but it will have clusters of Reeb chords \(\tilde{a}, \tilde{b}\), and \(\tilde{d}\) corresponding to the critical points of \(\tilde{f}\), where \(\tilde{a}\) is the cluster at the maximum \(p\), \(\tilde{b}\) is the cluster at the maximum of \(s\) in \(\partial D_{2.5}\) - corresponding to an index-\((n-1)\) Reeb chord - and \(\tilde{d}\) is the cluster at the minimum of \(s\) in \(\partial D_{2.5}\) - corresponding to a minimum. In each case, we will have as many Reeb chords \(\tilde{a}_j^i, \tilde{b}_j^i, \tilde{d}_j^i\) as we have pairs of sheets, excluding the pair \(0, 1\), and each Reeb chord in a cluster \(\tilde{x}\) will have the same projection to the base space. A cross-section of \(\tilde{f}_i\) is shown in Figure~\ref{fig:CrossSectionOfD3}.

We will use \textbf{formally rigid} to denote the gradient flow trees on these height difference functions, excluding \(\tilde{f}_1 - \tilde{f}_0\), that have formal dimension zero (they are not actually rigid because \(\tilde{f}_i\) are not generic).   Pick an arbitrarily small \(\epsilon > 0\), and let \(V\) be an \(\epsilon\)-neighborhood of the image in the base space \(M\) of the formally rigid gradient flow trees of all of the height difference functions \(\tilde{f}_i - \tilde{f}_j, i > j\), excluding \(\tilde{f}_1 - \tilde{f}_0\).   Note that \(V\) includes, in particular, the gradient flow line from \(\tilde{a}^i_j\) to \(\tilde{b}^i_j\), which coincide for all \(i, j\); this fact will be important later.

Find a disk \(D' \subset D_{2.4} - V\) of radius \(\rho\), choosing \(\rho\) to be arbitrarily small.   In particular, choose \(\rho\) to be small enough that we can approximate the height functions \(\tilde{f}_i, i \neq 0, 1\) linearly over \(D'\).   Let \(\eta\) be the infinum of \(|\nabla\tilde{f}_i|, i \neq 0,1\) over \(D'\).   Let \(\delta_1 > 0\) be a quantity small enough that:
\[
\delta_1 \leq \frac{\eta}{2(1+m) \rho^{3/2}}
\]
Where \(m\) is the supremum of \(|\nabla \tilde{f}|\).

By Lemma~\ref{lemma3.4}, we can find \(\delta_2 > 0\) such that, if we perturb our height functions by an amount less then \(\delta_2\), the image of the rigid gradient flow trees of the perturbed height functions will be within \(\epsilon\) of the images of the rigid gradient flow trees of the height functions \(\tilde{f}_i\), excepting the undefined flow trees on \(\tilde{f}_1 - \tilde{f}_0\).   Let \(\delta\) be the minimum of \(\delta_1, \delta_2\).

For \(i \neq 0,1\), define \(\bar{f}_i = \tilde{f}_i\), and define \(\bar{f}_1 = \frac{1}{3m}\delta\tilde{f}, \bar{f}_0 = -\frac{1}{3m}\delta\tilde{f}\).   Observe that \(\bar{f}_1, \bar{f}_0\) are \(\frac{1}{3}\delta\)-close to \(\tilde{f}_1, \tilde{f}_0\) in the \(C^1\) metric.

Define \(\hat{f}_i\) to be generic perturbations of \(\bar{f}_i\) that are \(\frac{1}{3}\delta\)-close to \(\bar{f}_i\) in the \(C^1\)-metric.   The height difference functions \(\hat{f}_i - \hat{f}_j\) now have three clusters of Reeb chords: the maxima at \(\hat{a}\); the index-\((n-1)\) chords around the north pole of \(\partial (D_{2.5})\) at \(\hat{b}\); and the minima around the south pole of \(\partial (D_{2.5})\) at \(\hat{d}\). Let \(\hat{\Lambda}'\) be the Legendrian submanifold which is equal to \(\hat{\Lambda}\) outside of \(\pi_M^{-1}(D)\), and whose front projection over \(\pi_M^{-1}(D)\) is the graphs of the functions \(\hat{f}_i\) over \(D\). \(\hat{\Lambda}'\) is then isotopic to \(\hat{\Lambda}\).

Observe that, by construction, the rigid gradient flow trees of \(\bar{f}_1 - \bar{f}_0\) coincide with the rigid gradient flow trees of \(\bar{f}_{i+1} - \bar{f}_{i}, i \neq 0\).   Therefore, the rigid gradient flow trees of \(\hat{f}_i\) lie within \(V\) for all \(i\), and thus avoid \(D'\).

\begin{figure}
\begin{center}\includegraphics{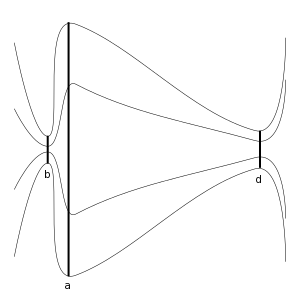}
\caption{Cross-section of \(D_3\)\label{fig:CrossSectionOfD3}}\end{center} 
\end{figure}

Let \(q\) be the center of \(D'\), and let \(\pi_1:(D' - \{q\}) \to \partial D'\) be the natural projection.   Define \(h_i(x) = f_i(\pi_1(x))\).   Define the set:
\[
D'_r = \left\{x \in D' \left| |x - q| \leq r \right.\right\}
\]
We define \({f}^1_0, f^1_1\) as follows:

\begin{itemize}

\item Outside of \(D'_{\rho}\), \(f^1_i = \hat{f}_i\).

\item Within \(D'_{0.9\rho} - D'_{0.8\rho}\), \(f^1_i\) has the form:
\[
f^1_1(x) = \frac{h_1(x)}{(0.1\rho)^{3/2}} \left(|x - q| - 0.8\rho\right)^{3/2}
\]\[
f^1_0(x) = \frac{h_0(x)}{(0.1\rho)^{3/2}} \left(|x - q| - 0.8\rho\right)^{3/2}
\]

\item \(f^1_0, f^1_1\) are not defined inside \(D'_{0.8\rho}\).

\item Within \(D'_\rho - D'_{0.9\rho}\):
\[
|\nabla f^1_1|, |\nabla f^1_0| < \frac{1}{2}\eta
\]
Where \(\eta\) is the infinum of \(|\nabla f_i|, i \neq 0, 1\), over \(D'\).

\end{itemize}

We define \(f_0, f_1\) to be generic \(\frac{1}{3}\delta\)-small perturbations of \(f^1_1, f^1_0\), and for \(i \neq 0, 1\) we define \(f_i = \hat{f}_i\) over our entire base space. Define \(\Lambda'\) to be the Legendrian submanifold equal to \(\Lambda\) outside of \(\pi_M^{-1}(D)\), and whose front projection is equal to the graphs of \(f_i\) over \(\pi_M^{-1}(D)\); \(\Lambda\) and \(\Lambda'\) are then Legendrian isotopic.

Thus, prior to the Legendrian ambient surgery, outside of \(D'\) we will have clusters of Reeb chords \(a, b, d\), with corresponding clusters after the surgery \(\hat{a}, \hat{b}, \hat{d}\).   Within \(D'\), before the surgery, we will have a single index-1 Reeb chord \(c\) between sheets 0 and 1.   However, observe that, for \(i = 0, 1\), over \(D'\):
\[
\nabla f_i = \nabla h_i\left(|x - q| - 0.8\rho\right)^{3/2} + \frac{3}{2}h_i\left(|x - q| - 0.8\rho\right)^{1/2} \nabla |x - q|
\]\[
|\nabla f_i| \leq |\nabla h_i| (0.1\rho)^{3/2} + \frac{3}{2} |h_i| (0.1\rho)^{1/2}\rho
\]
Since \(h_i\) is equal to \(\hat{f}_i\) on the boundary, we can bound \(|h_i|, |\nabla h_i|\) by \(|\hat{f}_i|, |\nabla \hat{f}_i|\):
\[
|\nabla f_i| \leq |\nabla \hat{f}_i| (0.1\rho)^{3/2} + \frac{3}{2} |\hat{f}_i| (0.1\rho)^{1/2}\rho
\]\[
|\nabla f_i| \leq \delta (1+m) \rho^{3/2} \leq \frac{1}{2}\eta
\]
Where \(\eta\) is strictly less then the supremum of \(|\nabla f_i|, i \neq 0, 1\) over \(D'\).   Therefore, no other Reeb chords are generated besides \(c\), so unlike our other Reeb chords, it is solitary, rather than appearing in a cluster of chords.   An ``overhead'' view of \(\pi_F(\Lambda')\) will have the form shown in Figure~\ref{fig:OverheadOfD3}.

\begin{figure}
\begin{center}\includegraphics{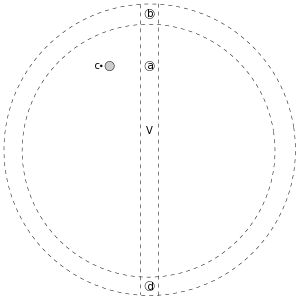}
\caption{Overhead View of \(D_3\)\label{fig:OverheadOfD3}}\end{center} 
\end{figure}

\subsection{DGA Calculations:} We will do this in several steps.   First of all:

\begin{lemma}\label{lemma3.5} If \(\delta, \epsilon, \rho\) are small enough, then there exists no partial flow tree \(\Gamma\) in \(\Lambda'\) or \(\hat{\Lambda}'\) with a positive special puncture over \(D'\) and a negative puncture at an \(a\) or \(\hat{a}\) Reeb chord.\end{lemma}

\begin{proof} We will prove the theorem for \(\Lambda'\); the argument is precisely the same for \(\hat{\Lambda}'\). Let \(A_1\) be the union of the ascending manifolds of every Reeb chord \(a^i_j\) for every pair of height difference functions \(-\nabla(f_k - f_l)\); let \(A_2\) be the union of the ascending manifolds of \(A_1\) for every pair of height difference functions \(-\nabla(f_k - f_l)\); let \(A_3\) be the union of the ascending manifolds of \(A_2\); etc.   Let \(m\) be the total number of sheets of \(\Lambda'\).   We claim that \(A_1, ..., A_m\) are disjoint from \(D'\) if \(\epsilon, \rho\) are small enough.

Recall that \(\bar{f}_i\) denotes the sheet height functions before they are genericized.   Let \(U(a)\) be a Morse neighborhood of \(\bar{a}^1_0\) for \(-\nabla (\bar{f}_i - \bar{f}_j)\).   (Since \(\bar{f}_i\) all equal each other multiplied by a constant, \(\pi_M(\bar{a}^1_0) = \pi_M(\bar{a}^i_j)\) for all \(i, j\), and the \(\bar{f}_i\) functions will have a common Morse neighborhood, except that the functions will take the form \(\bar{f}_i(x) - \bar{f}_j(x) = \bar{f}_i(0) - \bar{f}_j(0) - m_i(x_1^2 + ... + x_n^2)\) instead of \(\bar{f}_i(0) - \bar{f}_j(0) - (x_1^2 + ... + x_n^2)\).)   If \(\delta\) is small enough, the Reeb chords \(a^i_j, \hat{a}^i_j\) will lie in \(U(a)\).   Define an outward-pointing radial vector field of \(U(a)\):
\[
R = x_1\partial_{x_1} + ... + x_n\partial_{x_n}
\]
And observe \(-\nabla(\bar{f}_i - \bar{f}_j) \cdot R > 0\) on \(\partial U(a)\) for all \(i, j\), so all gradient flows of the height difference functions \(-(\bar{f}_i - \bar{f}_j)\) are leaving \(U(a)\).   Therefore, if \(\delta\) is small enough, then \(-\nabla(f_i - f_j) \cdot R > 0\) on \(\partial U(a)\) for all \(i, j\), so all gradient flows of \(-(f_i - f_j)\) on \(\partial U(a)\) are also leaving.

What this means is that, for every \(p \in \partial U(a)\) and every height difference function \(-(f_i - f_j)\), the ascending manifold of \(p\) for that height difference function will lie inside \(U(a)\).   Therefore, if \(\delta\) is small enough, the ascending manifold of every point \(p \in U(a)\) for every choice of height difference function \(-(f_i - f_j)\) will lie in \(U(a)\).   Therefore, \(A_1, ..., A_m \subset U(a)\).   Since \(U(a)\) is disjoint from \(D'\) if \(\delta, \epsilon, \rho\) are small enough, we conclude that \(A_1, ..., A_m\) is disjoint from \(D'\).

Why does this matter?   Well, consider a partial flow tree \(\Gamma\) in \(\Lambda'\) with a positive special puncture over \(D'\) and a negative puncture at a chord \(a^i_j\).   Let \(U(a)'\) be an arbitrarily small open neighborhood of \(U(a)\) that is also disjoint from \(D'\).   Let \(\Gamma'\) be the subtree of \(\Gamma\) that is obtained by restricting \(\Gamma\) to \(U(a)'\) and deleting the components of the restricted tree that do not contain our Reeb chord at \(a^i_j\).   \(\Gamma'\) will then have a positive special puncture \(q\) on \(\partial (U(a)')\), a negative puncture \(a^i_j\), and possibly some \(Y^0\) vertices and/or negative special punctures on \(\partial (U(a))'\).   Consider the sequence of edges \(e_1, ..., e_k\) in \(\Gamma'\) between \(a^i_j\) and \(q\): \(a^i_j\) is a boundary point of \(e_1\), define \(p_1\) to be the joint boundary point of \(e_1\) and \(e_2\), \(p_2\) to be the joint boundary point of \(e_2\) and \(e_3\), and so on, up to \(p_k = q\) is a boundary point of \(e_k\), as shown in Figure~\ref{fig:EdgesInUa}.   Note that, since each edge \(e_i\) must lie between fewer sheets then the edge \(e_{i+1}\), we can conclude that \(k \leq m\).

\begin{figure}
\begin{center}\includegraphics{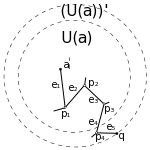}
\caption{Labels of Edges in \(U(a)\)\label{fig:EdgesInUa}}
\end{center}
\end{figure}

Then \(p_1 \in A_1, p_2 \in A_2, ..., p_k \in A_k\), so \(q\) lies in \(A_k\).   But \(A_k \subset U(a)\), and \(q \notin U(a)\), which is a contradiction.   Therefore \(\Gamma\) cannot have a negative puncture at a Reeb chord \(a^i_j\). The same argument applies to \(\hat{\Lambda}'\). \end{proof}

\begin{lemma}\label{lemma3.6} If \(\delta, \epsilon, \rho\) are small enough, then there exists no partial flow tree \(\Gamma\) in \(\Lambda'\) or \(\hat{\Lambda}'\) with a positive special puncture over \(D'\) and a negative puncture at a \(b\) or \(\hat{b}\) Reeb chord.\end{lemma}

\begin{proof} We prove the lemma for \(\Lambda'\); the argument is precisely the same for \(\hat{\Lambda}'\). Let \(B_1\) denote the union of the ascending manifolds of \(b^i_j\) for \(f_i - f_j = \hat{f}_i - \hat{f}_j\) for all \(i, j\).   Let \(B_2\) denote the union of the ascending manifolds of \(B_1\) for \(f_k - f_l\) for all \(k, l\); let \(B_3\) denote the union of the ascending manifolds of \(B_2\); and so on, up to \(B_m\), where \(m\) is the number of sheets of \(\pi_F(\Lambda')\).

Let \(U(b)\) be a Morse neighborhood of \(\bar{b}\), and assume \(\delta\) is small enough that \(b^i_j \in U(b)\) for all \(i, j\).   As in equation~\ref{eqn:defineS} in the proof of Lemma~\ref{lemma3.1}, define \(S\) to be a \(\epsilon_0\)-neighborhood of the ascending manifolds of \(\bar{b}\) for \(-\nabla (\bar{f}_k - \bar{f}_l)\):
\[
S = \left\{(x_1, ..., x_n) \in U(b) | x_2^2 + ... + x_n^2 \leq \epsilon_0^2\right\}
\]
Consider \(\partial S\) as the union of subsets \(V, W\), as in equation~\ref{eqn:defineVW}
\[
V = \left\{(x_1, ..., x_n) \in U(b) | x_2^2 + ... + x_n^2 = \epsilon_0^2\right\}
\]\[
W = \left\{(x_1, ..., x_n) \in \partial U(b) | x_2^2 + ... + x_n^2 \leq \epsilon_0^2\right\}
\]
And let \(R_V, R_W\) be vector fields on \(V, W\):
\[
R_V = -x_2\partial_{x_2} - ... - x_n\partial_{x_n}
\]\[
R_W = x_1\partial x_1
\]
Then, if \(\delta\) is small enough, for every \(-\nabla (f_k - f_l)\):
\[
\left.-\nabla (f_k - f_l)\right|_V \cdot R_V < 0
\]\[
\left.-\nabla (f_k - f_l)\right|_W \cdot R_W < 0
\]
Therefore, for any point \(p \in S\), and any choice \(i, j\), we know that \(\mathcal{A}_{-(f_i - f_j)}(p) \cap U(b) \subset S\), and in particular the ascending manifold must leave \(U(b)\) through \(W\).   Therefore, \(B_1 \cap U(b), B_2 \cap U(b), ..., B_m \cap U(b)\) all lie in \(S\), and \(B_1 \cap (\partial U(b)), B_2 \cap \partial (U(b)), ..., B_m \cap \partial (U(b))\) all lie in \(W\).   We can separate \(W\) into two components: \(W^+\), whose ascending manifolds lie in \(D_{2.49}\), and \(W^-\), whose ascending manifolds are disjoint from \(D_{2.49}\).

Now consider the ascending manifold of \(W^+\) for any choice of \(-\nabla (f_k - f_l)\).   Define \(Q'_\rho\) to be a \(\rho\)-neighborhood of the ascending manifold of \(\bar{b}^i_j\) for \(-\nabla (\bar{f}_k - \bar{f}_l)\), as shown in Figure~\ref{fig:WandQ}, and define \(Q_\rho\) to be the closure of \(Q'_\rho - (U(b) \cup U(a))\).   \(Q_\rho\) then contains no critical points of \(\bar{f}_k - \bar{f}_l\) for any \(k, l\), and lies in the descending manifold of \(\bar{a}\) for \(-\nabla (\bar{f}_k - \bar{f}_l)\).   In addition, \(W^+ \subset \partial Q_\rho\) if \(\epsilon_0 < \rho\).

\begin{figure}
\begin{center}\includegraphics{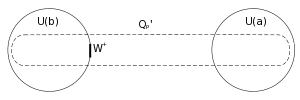}
\caption{Diagram of \(Q_\rho, U(a), U(b),\) and \(W^+\)\label{fig:WandQ}}
\end{center}
\end{figure}

By Lemma~\ref{lemma3.3}, if \(\delta\) is small enough, the ascending manifold of every point in \(W^+\) for \(-\nabla (f_k - f_l)\) will lie within \(\frac{1}{2}\epsilon_1\) of the ascending manifold of that point for \(-\nabla (\bar{f}_k - \bar{f}_l)\).   By Lemma~\ref{lemma3.2}, if \(\epsilon_0\) is small enough, then the ascending manifold of every point of \(W\) for \(-\nabla (\bar{f}_k - \bar{f}_l)\) will lie within \(\frac{1}{2}\epsilon_1\) of the ascending manifold of \(\bar{b}\) for \(-\nabla (\bar{f}_k - \bar{f}_l)\).   Therefore, if \(\delta, \epsilon_0\) are small enough, the ascending manifold of \(W^+\) for every \(-\nabla (f_i - f_j)\) will lie within \(Q'_{\epsilon_1}\).   Therefore, if \(\delta, \epsilon_0\) are small enough, then \(B_1 \subset Q'_{\epsilon_1} \cup (M - D_{2.49})\).

Now consider \(B_2\).   Since \(B_1\) consists of the union of the ascending manifolds of \(W^+\) for every choice of height difference functions \(-\nabla(f_i - f_j)\), then for any choice \(-\nabla(f_i - f_j)\), there exists some ascending manifold in \(B_1\) of that function, and every other point in \(B_1\) lies within \(2\epsilon_1\) of that ascending manifold.   Therefore, by Lemma~\ref{lemma3.2}, we can choose \(\epsilon_1\) to be small enough that \(B_2\) lies within \(\frac{1}{2}\epsilon_2\) of \(B_1\) for any choice of \(\epsilon_2\), and therefore \(B_2 \subset Q'_{\epsilon_2} \cup (M - D_{2.49})\).

We can then repeat this process for \(B_3, ..., B_m\), until we obtain that \(B_m \subset Q'_{\epsilon_m} \cup (M - D_{2.49})\).   What we obtain from this is that, if \(\epsilon_m, \epsilon_{m-1}, ..., \epsilon_1, \epsilon_0,\) and \(\delta\) are all small enough, then \(B_1, ..., B_m\) will be disjoint from \(D'\).   Therefore, for the same reasons as for the \(a\) chords in Lemma~\ref{lemma3.5}, there can exist no partial flow tree with a positive special puncture over \(D'\) and a negative puncture at a \(b\) chord.   The same argument applies to \(\hat{\Lambda}'\).\end{proof}

\begin{lemma}\label{lemma3.7} If \(\epsilon, \rho\) are small enough, then there exist no rigid gradient flow trees of \(\Lambda'\) whose image passes through \(D'\) and which does not have either a switch or a puncture at \(c\).\end{lemma}

\begin{proof}  Let \(\Gamma\) be a rigid flow tree in \(\Lambda'\) which passes through \(D'\) and which has no switch or puncture at \(c\).   Let \(Y(\Gamma)\) denote the image of \(\Gamma\) in the base space \(M\).   We can break \(\Gamma\) at some points \(y_1, ..., y_m \in Y(\Gamma) \cap \partial D'\) into a connected partial flow tree \(\Gamma'\) and a collection of partial flow trees \(\Gamma''_1, ..., \Gamma''_m\), so that \(\Gamma'\) is disjoint from \(D'\), \(\Gamma'\) has negative special punctures at \(y_1, ..., y_m\), and \(\Gamma''_i\) has a positive special puncture at \(y_i\).

Let \(r_i, s_i\) be the sheets of \(\Gamma\) at the point \(y_i\).   Generically, every point \(y_i\), if translated to \(\hat{\Lambda}'\), will be in the ascending manifold of the minimum \(\hat{d}^{r_i}_{s_i}\).   Therefore, we can find a partial flow tree \(\gamma_i\) in \(\hat{\Lambda}'\) - \emph{not} in \(\Lambda'\) - consisting of a flow line on \(-\nabla(f_{r_i} - f_{s_i})\) from \(y_i\) to the minimum \(\hat{d}_{r_i}^{s_i}\).   Since, by construction, the gradient flow \(-\nabla(f_{r_i} - f_{s_i})\) is entering \(D'\) at \(y_i\), \(\gamma_i\) must cross \(D'\).   Observe that \(\Gamma'\) is also a valid partial flow tree in \(\hat{\Lambda}'\), because it is disjoint from \(D'\), and \(\Lambda', \hat{\Lambda}'\) agree except over \(D'\).   Define \(\hat{\Gamma}\) to be the flow tree in \(\hat{\Lambda}'\) obtained by joining \(\Gamma', \gamma_1, ..., \gamma_m\).   By construction, the rigid flow trees of \(\hat{\Lambda}'\) do not cross \(D'\), so \(\hat{\Gamma}\) cannot be rigid.

Recall from section~\ref{subsec:flowtrees} that the formal dimension of the moduli space containing a gradient flow tree \(T\) is given by:
\begin{equation}\label{eqn:dimT}
\dim T = (n - 3) + \sum_{i=1}^m (I(p_i) - (n -1)) - \sum_{j=1}^l(I(q_j) - 1) + \sum_{k=1}^r \mu(v_k)
\end{equation}
Where \(n = \dim M\), \(p_i\) are the positive punctures, \(q_j\) are the negative punctures, \(v_k\) are the other vertices, \(I(x)\) is the Morse index of the Reeb chord \(x\), and \(\mu(z)\) is the Maslov content of a vertex \(z\).   Recall further that, if \(x\) is a positive special puncture, then \(I(x) = n + 1\), and if \(x\) is a negative special puncture, then \(I(x) = -1\).   Finally, recall that, if we divide a flow tree \(T\) into two partial flow trees \(T_1, T_2\), then:
\[
\dim T = \dim T_1 + \dim T_2 - (n + 1)
\]
Therefore, if we divide \(T\) into \(m+1\) partial flow trees \(T_1, ..., T_{m+1}\), then:
\begin{equation}\label{eqn:sumDimTi}
\dim T = \dim T_1 + ... + \dim T_{m+1} - m(n+1)
\end{equation}
We now apply these formulas to our current case.   Each \(\gamma_i\) consists of a positive special puncture, an edge, and a negative puncture at a minimum; therefore:
\[
\dim \gamma_i = (n - 3) + ((n + 1) - (n - 1)) - (0 - 1) = n
\]
And by equation~\ref{eqn:sumDimTi}:
\[
\dim \hat{\Gamma} = \dim \Gamma' + \sum_{i=1}^m \dim \gamma_i - m(n + 1) = \dim \Gamma' - m > 0
\]
Therefore:
\begin{equation}\label{eqn:dimGammaPrime}
\dim \Gamma' > m
\end{equation}
Now consider \(\Gamma''_i, i = 1,..., m\).   Since \(\Gamma''_i\) does not have a switch or a negative puncture at \(c\), and since by Lemma~\ref{lemma3.4} \(\Gamma''_i\) cannot have negative punctures at our \(a\) and \(b\) Reeb chords, \(\Gamma''_i\) can have no Reeb chords other than minima.   Besides minima, it can have \(Y^0\) and \(Y^1\) vertices.   Let \(k\) be the number of minima in \(\Gamma''_i\), and let \(l\) be the number of \(Y^1\) vertices.   Observe that \(k \geq l + 1\).   Therefore:
\[
\dim \Gamma''_i = (n - 3) + ((n + 1) - (n - 1)) - k(0 - 1) + l(-1) = 
n - 1 + k - l 
\]\begin{equation}\label{eqn:dimGammaDoublePrime}
\dim \Gamma''_i \geq n
\end{equation}
Therefore, combining equations~\ref{eqn:sumDimTi},~\ref{eqn:dimGammaPrime}, and~\ref{eqn:dimGammaDoublePrime}, we obtain:
\[
\dim \Gamma > m + mn - m(n+1)
\]\[
\dim \Gamma > 0
\]
Therefore, \(\Gamma\) is not rigid.\end{proof}

\begin{lemma}\label{lemma3.8} If \(\epsilon, \rho\) are small enough, then if \(\Gamma\) is a partial flow tree in \(\Lambda'\) which has a positive special puncture followed by a switch inside \(D'\) and does not have a negative puncture at \(c\), then \(\Gamma\) does not have a \(Y^1\) vertex.\end{lemma}

\begin{proof} The edge emerging from the switch must be on the height difference function \(-(f_1 - f_k)\) or \(-(f_k - f_0)\).   Suppose without loss of generality that it is \(-(f_1 - f_k)\).   Since there are no cusp edges within \(D\) except inside \(D'\), if the partial flow tree has a \(Y_1\) vertex, the height difference function of the edge entering the \(Y^1\) vertex must be on \(-(f_i - f_j)\), where \(1 \geq i > j \geq k\).   But, for a \(Y^1\) vertex, we must have \(i > 1 > 0 > j\).   This is a contradiction.\end{proof}

\begin{lemma}\label{lemma3.9} If \(\epsilon, \rho\) are small enough, than for every rigid gradient flow tree \(\Gamma\) in \(\Lambda'\) which has a switch inside \(D'\) and does not have a negative puncture at \(c\), there is a unique, second rigid gradient flow tree with the same positive and negative punctures.\end{lemma}

\begin{proof} We can find a point \(y_0\) in \(\Gamma\) immediately above a switch, so that we can break \(\Gamma\) at \(y_0 \in D'\) into partial flow trees \(\Gamma', \Gamma''\), where \(\Gamma'\) has a negative special puncture at \(y_0\), \(\Gamma''\) has a positive special puncture at \(y_0\), and \(\Gamma''\) has only a single switch.

Suppose the switch is from the sheets whose height difference function is \(f_0 - f_j\) to the sheets whose height difference function is \(f_1 - f_j\) (the following argument works equivalently if it switches from or to other sheets).   Let \(\Sigma' \subset M\) be the tangency locus of the cusp edge inside \(D'\) for the relevant sheets, and let \(\mathcal{M}(\Gamma'')\) be the component of the moduli space containing \(\Gamma''\).   Since \(\Gamma''\) does not have a negative puncture at \(c\), by Lemmas~\ref{lemma3.4} and~\ref{lemma3.6} it must consist of a single positive puncture at \(y_0\), the switch on \(\Sigma'\), one or more minima at Reeb chords \(d^k_l\), and possibly \(Y^0\) vertices.   Therefore, for any point \(s \in \Sigma'\), we can find a partial flow tree \(\Gamma''_s \in \mathcal{M}(\Gamma'')\) that has a switch at \(s\).

Let \(P(\Gamma'') \subset M\) denote the set of points \(p \in M\) such that there is a partial flow tree in \(\mathcal{M}(\Gamma'')\) with a special positive puncture at \(p\).   Since we can find a partial flow tree \(\Gamma''_s \in \mathcal{M}(\Gamma'')\) for any \(s \in \Sigma'\), the ascending manifold of \(\Sigma'\) for the vector field \(-\nabla(f_1 - f_j)\) is contained in \(P(\Gamma'')\).

Because we have chosen \(\rho\) small enough that we can treat \(f_j\) as linear over \(D'\), and because \(\nabla f_1, \nabla f_0 = 0\) on the cusp edge, \(\Sigma'\) will consist of those points under the cusp edge in \(D'\) where \(\nabla f_j \in T\Sigma_1\), which will generically be diffeomorphic to \(S^{n-2}, n = \dim \Lambda'\).   Therefore the ascending manifold is diffeomorphic to the cylinder \(S^{n-2} \times [0, 1)\); let \(P_{\mbox{sw}}(\Gamma'')\) denote the restriction of \(P(\Gamma'')\) to this ascending manifold.   Furthermore, since any two points on \(\Sigma'\) are at most \(2\rho\) apart, and \(\rho\) is arbitrarily small, by Lemma~\ref{lemma3.2} \(P_{\mbox{sw}}(\Gamma'')\) must lie within an arbitrarily small neighborhood of the ascending manifold of \(s_0 \in \Sigma'\), where \(s_0\) is the switch in \(\Gamma''\).   This ascending manifold is a 1-dimensional curve between \(s_0\) and a maximum chord.   The upshot is that, since \(\rho\) is arbitrarily small, \(P_{\mbox{sw}}(\Gamma'')\) is arbitrarily ``narrow".

Let \(\mathcal{M}(\Gamma')\) denote the component of the moduli space containing \(\Gamma'\), and let \(P(\Gamma') \subset M\) denote the set of points \(p \in M\) such that there is a partial flow tree in \(\mathcal{M}(\Gamma')\) with a special negative puncture at \(p\).   If \(\Gamma\) is rigid, then the intersection \(P(\Gamma') \cap P_{\mbox{sw}}(\Gamma'')\) must be zero-dimensional.   Since \(P_{\mbox{sw}}(\Gamma'')\) is codimension-1, \(P(\Gamma)\) must therefore be codimension-\((n-1)\), that is, 1-dimensional.   And, if \(\rho\) is small enough, \(P_{\mbox{sw}}(\Gamma'')\) will be narrow enough that the boundary points of \(P(\Gamma')\) will not lie in the solid cone whose boundary is \(P_{\mbox{sw}}(\Gamma'') \cup D''\), where \(D''\) is the disk whose boundary is the cusp edge inside \(D'\).   Then, \(\#(P_{\mbox{sw}}(\Gamma'') \cap P(\Gamma'))\) must be even.   And, for any point \(y_1 \in P_{\mbox{sw}}(\Gamma'') \cap P(\Gamma')\), we can find partial flow trees \(\Gamma''_{y_1} \in P_{\mbox{sw}}(\Gamma''), \Gamma'_{y_1} \in P(\Gamma')\) which have special punctures at \(y_1\), and connect them together to obtain a rigid gradient flow tree \(\Gamma_{y_1}\) which has the same negative and positive punctures as \(\Gamma\).   This concludes the proof.\end{proof}

\begin{lemma}\label{lemma3.10} If \(\epsilon, \rho\) are small enough, than for every rigid gradient flow tree \(\Gamma\) in \(\Lambda'\) which has a negative puncture at \(c\), there is a unique, second rigid gradient flow tree with the same positive and negative punctures.\end{lemma}

\begin{proof} The proof works broadly similarly to the proof of Lemma~\ref{lemma3.9}.   Because \(c\) is of Morse index 1, \(\Gamma\) cannot have a 2-valent vertex there.   Since the ascending manifolds of \(c\) of \(-\nabla(f_1 - f_0)\) do not intersect any cusp edges, that means that \(\Gamma\) must either be the flow line from \(a^1_0\) to \(c\) with no other vertices, or the edge of \(\Gamma\) leading to the vertex at \(c\) must begin with a \(Y^0\) vertex, \(y_0\).   If \(\Gamma\) is the flow line, there is a second flow line approaching the other side of \(c\), so assume it is not.

Break \(\Gamma\) at \(y_0\) into \(\Gamma_1, \Gamma_2, \Gamma_3\), where \(\Gamma_1\) has a negative special puncture at \(y_0\), \(\Gamma_2\) and \(\Gamma_3\) have positive special punctures at \(y_0\), and \(\Gamma_3\) is the flow line from \(y_0\) to \(c\).   Recall that \(\mathcal{M}(\Gamma_i)\) denotes the component of the moduli space containing the partial flow tree \(\Gamma_i\), and that \(P(\Gamma_i) \subset M\) denotes the set of points where a partial flow tree in \(\mathcal{M}(\Gamma_i)\) has a special puncture.   Since \(\Gamma\) is rigid, we know that:
\[
\dim \left(P(\Gamma_1) \cap P(\Gamma_2) \cap P(\Gamma_3)\right) = 0
\]
Further, since \(c\) is index-1, its ascending manifold on \(-\nabla(f_1 - f_0)\) is a pair of 1-dimensional curves, \(S^0 \times (0, 1)\), and these are submanifolds of \(P(\Gamma_3)\); call them \(P_c(\Gamma_3)\).   These are codimension-\((n-1)\), so, generically, \(P(\Gamma_1) \cap P(\Gamma_2)\) must be codimension-1.

Consider \(P_c(\Gamma_3) - D'\).   This consists of a pair of curves, each of which has one boundary point at \(a^1_0\), and the other boundary point on \(\partial D'\); label the boundary points on \(\partial D'\) by \(q_1, q_2\).   \(q_1, q_2\) must be within \(2\rho\) of each other.   Since \(f_i = \hat{f}_i\) outside of \(D'\), the ascending manifolds of \(q_1, q_2\) by \(-\nabla(f_1 - f_0)\) equal the ascending manifolds of \(q_1, q_2\) by \(-\nabla(\hat{f}_1 - \hat{f}_0)\), and for the function \(\hat{f}_1 - \hat{f}_0\), \(q_1, q_2\) lie in the same component of the descending manifold of \(\hat{a}^1_0\) and ascending manifold of \(\hat{d}^1_0\).   Therefore, by Lemma~\ref{lemma3.2}, for any \(\epsilon' > 0\) there exists \(\delta' > 0\) such that if \(2\rho < \delta'\), then the two components of \(P_c(\Gamma_3) - D'\) lie within \(\epsilon'\) of each other.   Therefore, since we can make \(\rho\) arbitrarily small, we can ensure that the two curves in \(P_c(\Gamma_3)\) are arbitrarily close together outside of \(D'\) - and since they must necessarily be within \(2\rho\) of each other inside \(D'\), we can bound the distance between the two curves everywhere.

Therefore, the points in \(P(\Gamma_1) \cap P(\Gamma_2) \cap P(\Gamma_3)\) will appear in pairs.   For every such point \(y_1\), we can find partial flow trees \(\Gamma'_1 \in \mathcal{M}(\Gamma_1), \Gamma'_2 \in \mathcal{M}(\Gamma_2), \Gamma'_3 \in \mathcal{M}(\Gamma_2)\) that have special punctures at \(y_1\), and then connect them together with a \(Y^0\) vertex to form a new rigid gradient flow tree \(\Gamma'\) with the same positive and negative punctures as \(\Gamma\).\end{proof}

\begin{lemma}\label{lemma3.11} If \(\epsilon, \rho\) are small enough, then for \(x \neq c\), the differential after the Legendrian ambient surgery is given by:
\[
\partial \hat{x} = \widehat{(\partial x)}
\]
\end{lemma}

\begin{proof} {\color{black}The region \(D_{2.51} - D_{2.49}\) contains a pinch.   In the language of Theorem~\ref{theorem1.1}, \(S = \partial D_{2.50}, N = D_{2.51} - D_{2.49}, Q_1 = M - D_{2.49}, Q_2 = D_{2.51}\).   Therefore, by Theorem~\ref{theorem1.1}, no pseudoholomorphic disks with any puncture in \(\Lambda' - \pi_M^{-1}(D_{2.51})\) can have a puncture in \(\pi_M^{-1}(D_{2.49})\).   The same is true for \(\hat{\Lambda}'\).   Given that \(\Lambda', \hat{\Lambda}'\) coincide outside of \(D_{2.49}\), for any Reeb chord \(x\) outside of \(D_{2.51}\), \(\partial \hat{x} = \widehat{\partial x}\); that is, there is no change in the differential.}

{\color{black}Now consider the Reeb chords inside \(\pi_M^{-1}(D_{2.51})\).}   Because the only difference between \(\Lambda'\) and \(\hat{\Lambda}'\) is over \(D'\), and by construction all of the rigid flow trees of \(\hat{\Lambda}'\) avoid \(D'\), all of the rigid flow trees of \(\hat{\Lambda}'\) are also rigid flow trees of \(\Lambda'\).   However, it is possible that there are now new rigid flow trees for \(\Lambda'\) that pass through \(D'\).   By Lemmas~\ref{lemma3.7},~\ref{lemma3.9}, and~\ref{lemma3.10}, any such new rigid flow trees appear in pairs, and so their contribution to the differential is canceled out.   This concludes the proof.\end{proof} 

{\color{black}\textbf{Proof of Theorem~\ref{theorem1.5}:}} The fact that \(\partial \widehat{x} = \widehat{\partial x}\) for \(x \neq c\) is immediate from Lemma~\ref{lemma3.11}. \(\partial c = 1 + d^1_0\) follows from the following facts: first, by Theorem~\ref{theorem1.1}, \(\partial c\) can only include Reeb chords that lie within \(D_3\). Second, because \(c\) lies on sheets 0 and 1, and the only cusp edges in \(D_3\) join sheets 0 and 1, any gradient flow trees emerging from \(c\) must be flow lines that terminate either in a minimum or a cusp edge. Therefore, \(\partial c = 1 + d_0^1\).

\addtocontents{toc}{\protect\newpage}
\section{Proof of Theorems 1.1, 1.2, and 1.3}
\label{sec:proof}

\subsection{Proof of Theorems 1.1 and 1.3}
\label{subsec:proofofmaintheorem}

Recall that, in Theorem 1.1, we defined \(S\) to be a hypersurface in \(M\) that divides \(M\) into \(R_1\) and \(R_2\) and which does not intersect a codimension-2 singularity of the front projection, \(N\) to be an arbitrarily small neighborhood of \(S\), and \(Q_i = R_i \cup N\).   We can generically assume that \(S\) intersects cusp edges transversely, and that \(N\) does not contain any codimension-2 or higher singularities of \(\pi_F\).   Pick a quantity \(\delta > 0\) that is less then the action of any existing Reeb chord (new Reeb chords will be created by the pinching isotopy), and choose a tubular neighborhood \(S\times[-1,1]\) of \(S\) such that \(N = S \times [-1, 1]\), and let \(\mu\) denote the tubular coordinate.

We define \(\lambda_{\epsilon, \delta}: [-\epsilon,\epsilon] \to \mathbb{R}\) to be a smooth inverted bump function such that:
\[
1 \geq \lambda_{\epsilon, \delta}(\mu) > 0
\]\[
\lambda_{\epsilon, \delta}(\epsilon) = \lambda(-\epsilon) = 1
\]\[
\lambda_{\epsilon, \delta}(0) = \delta
\]\[
\frac{\partial\lambda_{\epsilon, \delta}}{\partial\mu} \mbox{ has the same sign as } \mu
\]
We will ordinarily suppress the subscripts.   A graph of this function has the form shown in figure~\ref{fig:graphOfLambda}.

\begin{figure}\begin{center}
\includegraphics{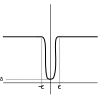}
\caption{Graph of \(\lambda_{\epsilon,\delta}(\mu)\)\label{fig:graphOfLambda}}\end{center}
\end{figure}

Then define \(l: M \to \mathbb{R}\) be a smooth function such that:
\[
l(x) = \left\{\begin{array}{ll}
1&\mbox{for }x \notin S \times [-1, 1]\\
\lambda(\mu)&\mbox{for }x\in S \times [-1, 1]
\end{array}\right\}
\]
We then define a Legendrian isotopy \(\Lambda_t\) in terms of the front projection \(\pi_F(\Lambda)\).   Let \(f_1, ..., f_m\) be the sheet functions \(M \to \mathbb{R}\) of the front projection, and define:
\[
f^t_i(x) = ((1 - t) + tl(x))f_i(x)
\]
This defines a Legendrian isotopy \(\Lambda'_t\), where \(\Lambda'_0 = \Lambda\).   Although this isotopy captures the primary features we need, we need to make some further small modifications in order to preserve front genericity.   Choose a quantity \(\epsilon' < \mbox{min }(\delta, \epsilon / 2)\) and perform a \(C^2\)-small perturbation of \(\Lambda'_t\) of order \(\epsilon'\) in an \(\epsilon'\)-neighborhood of any cusp edges that intersect \(N\), to make it front generic, as defined in section~\ref{subsec:dga}.   This gives us a new Legendrian isotopy \(\Lambda_t\).   We call this isotopy a \textbf{pinching isotopy along S}.

We need to modify our approach to cope with exact Lagrangian cobordisms, since these are not compact and do not have Reeb chords.   Recall that we define \(\hat{L}\) to be the Legendrian lift of an exact Lagrangian cobordism \(L\); that \(\hat{S} \subset M \times \mathbb{R}^+\) is a hypersurface such that each time-slice is a hypersurface in \(M\) and \(\pi_M^{-1}(\hat{S})\) is disjoint from \(\Sigma_1\); and that \(\hat{N}\) is an arbitrarily small neighborhood of \(\hat{S}\).   A Legendrian isotopy of \(\hat{L}\) descends to an exact Lagrangian \emph{homotopy} of \(L\).   This is an exact Lagrangian isotopy if and only if there are no Reeb chords at any time in the Legendrian isotopy.   This is the reason for the restriction that \(\hat{S}\) cannot cross any cusp edges or self-intersections of \(\pi_F(\hat{L})\) in the statement of Theorem~\ref{theorem1.4}; we claim that this is a sufficient (though not necessary) condition to allow us to find a pinching isotopy of \(\hat{L}\) that descends to an exact Lagrangian isotopy.

Instead of \(\lambda(\mu)\), we use \(\lambda(\mu, \tau)\), where \(\tau\) is the cylindrical coordinate.   Let \(T_-, T_+\) be the cylindrical coordinates such that \(L \cap (J^1(M) \times (0, T_-)), L \cap (J^1(M) \times (T_+, \infty))\) are the cones over \(\Lambda_-, \Lambda_+\).   Let \(\delta > 0\) be some quantity smaller then the action of the smallest Reeb chord of \(\Lambda_- = \partial \mathcal{E}_-(L)\).   We require that \(\lambda(\mu, \tau)\) obey the same requirements as \(\lambda(\mu)\), and, in addition:
\[
\frac{\partial \lambda}{\partial \tau}(\mu, \tau) \geq 0
\]\[
\lambda(\mu, \tau) = \frac{t}{T_-}\lambda(\mu, T_-)\mbox{ for } \tau \leq T_-
\]\[
\lambda(\mu, \tau) = \frac{t}{T_+}\lambda(\mu, T_+)\mbox{ for } \tau \geq T_+
\]
We then define an isotopy on \(\hat{L}\) in a precisely analogous fashion as we did for \(\Lambda\).

\begin{lemma}\label{lemma5.20} For any exact Lagrangian cobordism \(L\), given a hypersurface \(\hat{S} \subset M \times \mathbb{R}^+\) such that each time-slice is a hypersurface in \(M\), and there is a neighborhood \(\hat{N}\) of \(\hat{S}\) such that no cusp edges lie above \(\hat{N}\), the isotopy induced by \(\hat{L}\) descends to an exact Lagrangian isotopy.\end{lemma}

\begin{proof} Let \(f_i, f_j\) be sheet functions of \(\pi_F(\hat{L})\) prior to the pinching isotopy, \(f_i > f_j\).   Since, by definition, \(\hat{N}\) may not cross a cusp edge, we further know that:
\[
\left.\frac{\partial}{\partial \tau} (f_i - f_j)\right|_{\hat{N}} > 0
\]

During the pinching, within \(\hat{S} \times [-\epsilon, \epsilon]\) \(f_i, f_j\) are replaced by \(((1- t) + t\lambda(\mu, \tau))f_i(x, t), ((1 - t) + t\lambda(\mu, \tau))f_j(x, t)\).   Then, within \(\hat{N}\), observe that:
\[
\frac{\partial}{\partial \tau}(((1- t) + t\lambda(x, t))(f_i - f_j) = 
t\frac{\partial\lambda}{\partial \tau}(f_i - f_j) + (((1- t) + t\lambda(x, t)) \frac{\partial}{\partial t}(f_i - f_j)
\]
By construction, all of the terms of on the right hand side of the equation are positive, so there are no Reeb chords within \(\hat{N}\) at any value of \(t\).\end{proof}

Once we have performed the pinching isotopy, we can then prove the main theorem, outsourcing the detailed analysis to appendix~\ref{sec:maintechnical}:

\bigskip

{\color{black}\textbf{Proof of Theorem~\ref{theorem1.1}:} In Appendix A, we define \(s_\sigma\) to be the scaling of the \(z\) and cofiber coordinates of \(J^1(M)\) by \(\sigma\) as \(\sigma \to 0\). We define \(\Lambda_\sigma\) to be an arbitrarily small perturbation of \(s_\sigma(\Lambda)\), to ensure that certain technical conditions on the boundaries are met, and we show that for any pseudoholomorphic disk \(u:(\Delta_m, \partial \Delta_m) \to (J^1(M), \Lambda)\), there is a corresponding pseudoholomorphic disk \(u_\sigma:(\Delta_m, \partial\Delta_m) \to (J^1(M), \Lambda_\sigma)\). In Lemmas~\ref{WFGLemma} and \ref{GFGlemma}, we show that there exists a sequence \(\sigma \to 0\) such that:
\[
(\pi_M \circ u_{\sigma_i})(\partial\Delta_r) \cap (M - \pi_M(\Sigma_2))
\]
converges to a collection of partial flow trees \(F_i:\Gamma_i \to M\) as \(\sigma_i \to 0\), such that every partial flow tree has only a single positive puncture.   As a consequence, \(F_i(\Gamma_i)\) may not have a flow line along a pair of sheets whose front projections cross each other, since this would imply the existence of a second positive puncture.

Now suppose we have pinched \(\Lambda\) along \(S\) to form \(\Lambda'\) using arbitrarily small \(\delta, \epsilon\), and suppose we have a pseudoholomorphic disk \(u\) such that \(u\) has punctures over both \(M - Q_1\) and \(M - Q_2\).   By hypothesis, we have chosen \(S\) so that it does not intersect \(\pi_M(\Sigma_2)\).   Since \(N\) is arbitrarily narrow, we can therefore assume that \(N\) does not intersect \(\pi_M(\Sigma_2)\).

By construction, \(\partial N\) will be disconnected: we will have one set of components facing \(Q_1\) and one set facing \(Q_2\), each a slight displacement of \(S\).   Let \(\partial_1N, \partial_2N\) be these two disjoint subsets of \(\partial N\).   If \(u\) has punctures in both \(M - Q_1\) and \(M - Q_2\), then there exists some tree \(F_i:\Gamma_i \to M\) which intersects both \(\partial_1N\) and \(\partial_2N\).   Then, by choosing appropriate edges in \(\Gamma_i\), we can find a piecewise smooth path \(\gamma:[0, 1] \to N\) such that \(\gamma(0) \in \partial N_1, \gamma(1) \in \partial N_2\), and, except at finitely many points \(v_1, ..., v_m\), \(\gamma' = -\lambda \nabla(f_i - f_j)\) for some height functions \(f_i, f_j\) and some strictly positive, continuous function \(\lambda:[0, 1] \to \mathbb{R}^+\).   And, since we established that the flow lines of \(F_i(\Gamma_i)\) may not flow along a pair of sheets whose front projections cross, we know that \(f_i > f_j\) at every point in \([0, 1]\) where \(\gamma'(t)\) is defined.   Let \(h:([0,1] - \{v_1, ..., v_m\}) \to \mathbb{R}^+\) be the function that maps \(t\) to \(f_i(\gamma(t)) - f_j(\gamma(t))\), where \(f_i, f_j\) are the corresponding height functions at \(\gamma(t)\).   Because \(\gamma'(t) = -\nabla(f_i - f_j)\), we know that \(h'(t) < 0\).   Furthermore, since \(\Gamma_i\) is a partial flow tree, \(h(v_k - \epsilon) > h(v_k + \epsilon)\) for \(\epsilon > 0\).   Therefore, \(h\) is monotone decreasing, and since \(f_i > f_j\), we know that \(h(t) > 0\).   This is a contradiction, since the value of \(h(t_0)\) for \(\gamma(t_0) \in S\) must be strictly less than its value at both \(\partial_1N\) and \(\partial_2N\).

Therefore, any pseudoholomorphic disk that has a puncture at a Reeb chord over \(Q_1\) cannot have a puncture at a Reeb chord over \(Q_2 - N\), and any disk that has a puncture over \(Q_2\) cannot have a puncture over \(Q_1 - N\).   A partial flow tree can begin in \(Q_1\) and flow into \(N\), but then cannot escape it.   Therefore, the differential of any Reeb chord lying over \(Q_i\) will consist solely of Reeb chords lying over \(Q_i\).  Furthermore, since the actions of the Reeb chords over \(N\) are arbitrarily small, the differential of any Reeb chord lying over \(N\) will consist of Reeb chords lying over \(N\).   Therefore:
\[
\partial\big|_{Q_i}^2(x) = \partial^2(x) = 0 \mbox{ for any } x \in Q_i
\]\[
\partial\big|_{N}^2(x) = \partial^2(x) = 0 \mbox{ for any } x \in N
\]
So \(\mathcal{A}_K(\Lambda)|_{Q_i}, \mathcal{A}_K(\Lambda)|_{N}\) are well-defined differential graded algebras.   We also have canonical inclusion maps:
\[
\xymatrix{
\mathcal{A}_K(\Lambda)   & \mathcal{A}_K(\Lambda)\big|_{Q_1}\ar[l]^{i_1}\\
\mathcal{A}_K(\Lambda)\big|_{Q_2} \ar[u]^{i_2} &\mathcal{A}_K(\Lambda)\big|_N \ar[u]^{j_1}\ar[l]^{j_2}}
\]
It is clear that \(i_1 \circ j_1 = i_2 \circ j_2\), so the diagram commutes.

\bigskip

\textbf{Proof of Theorem~\ref{theorem1.4}:} The proof is precisely analogous to the proof of Theorem~\ref{theorem1.1}.   By the same steps, we see that any pseudoholomorphic disk with a positive puncture in \(\Lambda_+'|_{Q_1^+}\) cannot have negative punctures in \(\Lambda_-'|_{Q_2^- - N^-}\), and the same for \(\Lambda_+'|_{Q_2^+}\) and \(\Lambda_-'|_{Q_1^- - N^-}\), and for \(\Lambda_+'|_{N^+}\) and \(\Lambda_-'|_{N^-}\).   The result follows.
}

\subsection{Proof of Theorem 1.2}
\label{subsec:proofoflinearizedtheorem}

The argument in this section is taken essentially from \cite{HS}, Sec. 3.

\textbf{Proof of Theorem~\ref{theorem1.3}:} \(\Lambda\) is a Legendrian submanifold pinched as in theorem 1.1 along a neighborhood \(N\) of a hypersurface \(S\) dividing \(M\) into \(Q_1, Q_2\), and that \(\epsilon\) is an augmentation of \(\mathcal{A}(\Lambda)\).

Define \(i_1, i_2, j_1, j_2\) to be the inclusion maps:
\[
\xymatrix{
\mathcal{A}(\Lambda)   & \mathcal{A}(\Lambda)|_{Q_1}\ar[l]^{i_1}\\
\mathcal{A}(\Lambda)|_{Q_2} \ar[u]^{i_2} &\mathcal{A}(\Lambda)|_N \ar[u]^{j_1}\ar[l]^{j_2}}
\]
Observe that \(i_1 \circ j_1 = i_2 \circ j_2\).   Define \(\epsilon_{Q_1} = \epsilon \circ i_1, \epsilon_{Q_2} = \epsilon \circ i_2, \epsilon_N = \epsilon \circ i_1 \circ j_1 = \epsilon \circ i_2 \circ j_2\).   These are obviously augmentations, since \(\epsilon_X \circ \partial|_X = (\epsilon \circ \partial)|_X = 0\).

Now, consider the maps \(j_1 \oplus j_2\) and \(i_1 + i_2\).   Observe that:
\[
(i_1 + i_2) \circ (j_1 \oplus j_2) = (i_1 \circ j_1) + (i_2 \circ j_2) = 0
\]
Since we are working over \(\mathbb{Z}_2\).   Therefore, the image of \(j_1 \oplus j_2\) lies in the kernel of \(i_1 + i_2\).   Now observe that, if \((x, y)\) is in the kernel of \(i_1 + i_2\), that implies that \(x = y\), meaning that \(x\) is in the intersection of \(\mathcal{A}(\Lambda)|_{Q_1}\) and \(\mathcal{A}(\Lambda)|_{Q_2}\), which is \(\mathcal{A}(\Lambda)|_N\).   Therefore the kernel of \(i_1 + i_2\) equals the image of \(j_1 \oplus j_2\) and we have a long exact sequence.

Just as we have a Seifert-van Kampen theorem for both Legendrian submanifolds and exact Lagrangian cobordisms, there is an extension of the Mayer-Veitoris theorem to exact Lagrangian cobordisms.   Let \(L\) be an exact Lagrangian cobordism from \(\Lambda_+\) to \(\Lambda_-\), and let \(\Phi_L\) be the cobordism map.   If \(\Lambda_-\) has an augmentation \(\epsilon\), this induces an augmentation \(\epsilon \circ \Phi_L\) on \(\Lambda_+\).   Then, after splashing, theorem 1.4 holds for the linearized chain complexes of \(\Lambda_+, \Lambda_-\) as well as the differential graded algebra.

\appendix
\section{Convergence of Disk Boundaries to Flow Lines}
\label{sec:maintechnical}

This section is derived, with some modifications and a great many omissions, from \cite{Ek}.   {\color{black}Sections \ref{subsec:deformation_metric} and \ref{subsec:deformation_cusprounding} are derived entirely from \cite{Ek}, Section 4.2.   Section \ref{subsec:monotonicitylemma} is derived from \cite{Ek}, Section 5.1, but we make the reasoning more explicit.}   The remaining sections are derived from \cite{Ek}, Sections 5.2 and 5.3, but with modifications to allow higher-dimensional singularities in the front projection.   For simplicity's sake, we will provide the original source for each lemma in parentheses next to the lemma.

Let \(s_\sigma:J^1(M) \to J^1(M), s_\sigma(x, y, z) = (x, \sigma y, \sigma z)\) be the scaling of the cofiber and \(z\) components by \(\sigma > 0\).   We define:
\[
\tilde{\Lambda}_\sigma = s_\sigma(\Lambda)
\]

We begin by defining some arbitrarily small modifications to \(\tilde{\Lambda}_\sigma\) in sections~\ref{subsec:deformation_metric} and~\ref{subsec:deformation_cusprounding}.   These modifications will produce \(\Lambda_\sigma\), which is what we will work with after those sections.   \(\Lambda_\sigma\) allows us to obtain a modified version of the monotonicity lemma in section~\ref{subsec:monotonicitylemma}, which will still hold even as \(\sigma\) varies.   In section~\ref{subsec:slitmodel}, we define the slit model of the disk.   In section~\ref{subsec:boundingderivative}, we use the slit model and the monotonicity lemma to bound the derivatives of the map of the pseudoholomorphic disk in terms of \(\sigma\).   In section~\ref{subsec:domainsubdivision}, we add punctures to the boundary of the disk, which allows us to restrict our attention to the disk away from an arbitrarily small neighborhood of the codimension-2 singularities of \(\pi_F\).   In section~\ref{subsec:convergenceofdiskboundaries}, we show that, away from this neighborhood, the boundary of the disk converges to a flow line of a height difference function as \(\sigma \to 0\){\color{black}.

\subsection{Deformations of the Legendrian - The Metric}
\label{subsec:deformation_metric}

Recall that section~\ref{subsec:dga} defines \(\Sigma_k \subset \Lambda\) as the codimension-\(k\) components of the singularities of the front projection \(\pi_F\).   Define \(U(k, d)\) to be a product neighborhood of \(\Sigma_k\) of radius \(d\) inside \(\Lambda\), and define \(N(k, d)\) to be a product neighborhood of \(\Sigma_k\) of radius \(d\) in \(T^*M\).   {\color{black}Define \(U_\sigma(k, d), N_\sigma(k, d)\) to be:
\[
U_\sigma(k, d) = U(k, \sigma d) - U(k + 1, \sigma d)
\]\[
N_\sigma(k, d) = N(k, \sigma d) - N(k + 1, \sigma d)
\]}

Let \(\tilde{g}\) be a metric on \(M\) such that the self-intersections of {\color{black}\(\pi_M(\Sigma_1)\)} are orthogonal.   Let \(\epsilon_1 > 0\) be a deformation parameter.   Define \(b\) to be the restriction of \(\tilde{g}\) to \(\pi_M(\Sigma_1)\).   Using \(b\) and \(\tilde{g}\) we can produce a metric {\color{black}\(g_\sigma\)} on {\color{black}\(U_\sigma(1, \epsilon_1)\)} which can be made arbitrarily close to \(\tilde{g}\) by choosing \(\epsilon_1\) small enough, and such that if \(m \in \Sigma_1 - \Sigma_2\), {\color{black}and \(\sigma\) is small enough,} then there exists a coordinate patch around \(m\) disjoint from \(\Sigma_2\) with coordinates \((q, s) \in \mathbb{R}^{n-1} \times \mathbb{R}\), such that:
\[
\Sigma_1 \mbox{ corresponds to } \{s = 0\}
\]\[
g_{q,s} = \sum_{i,j} (b_{i,j})(q) dq_i \otimes dq_j + ds \otimes ds
\]
Now that we have established this product metric and coordinates, we move on to cusp rounding inside {\color{black}\(U_\sigma(1, \epsilon_1)\)}.

\subsection{Deformations of the Legendrian - Cusp Rounding}
\label{subsec:deformation_cusprounding}

Let \(m \in \pi_M(\Sigma_1)\), and {\color{black}assume \(\sigma\) is small enough that we can find} a coordinate patch \((q,s) \in \mathbb{R}^{n-1}\times\mathbb{R}\) around \(m\).   We have two kinds of sheets over \(m\): those on which \(\pi_M\) is an immersion, and those with a cusp edge singularity.   We call the former an \textbf{unfolded sheet}, and the latter a \textbf{folded sheet}.

From section~\ref{subsec:deformation_metric}, we have local coordinates \((q, s, \kappa, \nu, z)\) of \(J^1(M)\), where \((q,s) \in \mathbb{R}^{n-1} \times \mathbb{R}\), \(\kappa_i\) is the cofiber coordinate of \(q_i\), \(\nu\) is the cofiber coordinate of \(s\), and \(z\) is the \(\mathbb{R}\) coordinate.   Then an unfolded sheet can be parameterized locally {\color{black}in \(\pi_M(U_\sigma(1, \epsilon_1))\)} as the graph of a function \(f\):
\[
(q, s) \to \left(q, s, \sigma \partial_qf, \sigma \partial_sf, \sigma f(q, s)\right)
\]
We perform a small Legendrian isotopy for \(|s| < \epsilon_1\), replacing \(f\) with its Taylor polynomial of degree 1 in \(s\):
\[
f(q, s) = a(q) + s h(q) + ...
\]\begin{equation}\label{eqn:unfoldedsheetform}
(q, s) \to \left(q, s, \sigma \partial_q a + \sigma s \partial_qh, \sigma h, \sigma\left(a + sh\right)\right)
\end{equation}

This is a Legendrian isotopy that introduces no new Reeb chords.   {\color{black}We interpolate between this isotopy and \(\tilde{\Lambda}_\sigma\) over \(\pi_M(U(2, \sigma\epsilon_1) - U(2, \frac{1}{2}\sigma\epsilon_1))\).   (This is a modification of \cite{Ek}, which assumes that no higher singularities other than swallowtails in dimension 2 and crossings of cusp edges in arbitrary dimension exist, allowing them to define the cusp rounding without interpolation.)}

Next we perform a similar isotopy for folded sheets.   We assume without loss of generality that the projection of the folded sheet lies in \(\{s \geq 0\}\).   Then the sheet can be parameterized locally {\color{black}in \(\pi_M(U_\sigma(1, \epsilon_1))\)} as the graph of a function \(f\):
\[
(q, s) \to \left(q, \frac{1}{2}s^2, \sigma \partial_q f, \sigma \partial_s f, \sigma f\right)
\]
Where:
\[
\frac{\partial f}{\partial s}(q, 0) = 0
\]\[
\frac{\partial^3f}{\partial s^3}(q, 0) \neq 0
\]
We perform an isotopy over {\color{black}\(\pi_M(U_\sigma(1, \epsilon_1))\)} to replace \(f\) with its Taylor polynomial of degree 3 in \(s\), obtaining:
\[
(q, s) \to \left(q, \frac{1}{2}s^2, \sigma\left(\partial_qa + \frac{1}{2}s^2\partial_q\beta + \frac{1}{3}s^3\partial_q\alpha\right), \right.\]\begin{equation}\label{eqn:foldedsheetform}
\left.\sigma(\beta + s\alpha), \sigma\left(a + \frac{1}{2}s^2\beta + \frac{1}{3}s^3\alpha\right)\right)
\end{equation}
Where \(\alpha, \beta, a\) depend only on \(q\), and \(\alpha(q) \neq 0\), {\color{black}and interpolate between this isotopy and \(\tilde{\Lambda}_\sigma\) over \(\pi_M(U(2, \sigma\epsilon_1) - U(2, \frac{1}{2}\sigma\epsilon_1))\).}

We follow this with an additional isotopy inside an even smaller neighborhood of the cusp edge.   For notational simplicity, we begin with the case \(\dim \Lambda = 1\), where the cusp edge has the form:
\[
s \to \left(\frac{1}{2}s^2, \sigma s, \frac{1}{3}\sigma s^3\right)
\]
Where \(-\delta \leq s \leq \delta\).   Fix \(\epsilon_2 \ll \delta\).   We define \(\Theta_\sigma\) to be the projection of the above cusp edge to \(T^*\mathbb{R}\) for \(0 \leq x \leq l\), where \(x\) is the base space coordinate and \(y\) is the cofiber coordinate, and where \(l > \sigma\epsilon_2\).   We define \(c_\sigma\) to be the curve in \(T^*\mathbb{R}\) given by a half-circle of radius \(\sigma\epsilon_2\) centered at \((\sigma\epsilon_2, 0)\), connected to a pair of horizontal line segments \(y = \pm \sigma\epsilon_2\) from \(x = \sigma\epsilon_2\) to \(x = l\).   Explicitly, this has the formula:
\begin{equation}\label{eqn:definecsigma}
s \to \left(\begin{array}{ll}
(\frac{1}{2}s^2, -\sigma \epsilon_2)& s\in (-\infty, -\sqrt{2\sigma\epsilon_2})\\
(\frac{1}{2}s^2, \frac{1}{2}s\sqrt{4\sigma\epsilon_2 - s^2})& s \in (-\sqrt{2\sigma\epsilon_2}, \sqrt{2\sigma\epsilon_2})\\
(\frac{1}{2}s^2, \sigma\epsilon_2)& s \in (\sqrt{2\sigma\epsilon_2}, \infty)
\end{array}\right)
\end{equation}
Restricted to \(0 \leq x \leq l\).

\begin{figure}\begin{center}
\includegraphics{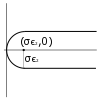}
\caption{Graph of \(c_\sigma\) in \(T^*\mathbb{R}\)}
\end{center}\end{figure}

The area bounded by the curve \(c_\sigma\) is \(\frac{1}{2}\pi\epsilon_2^2\sigma^2 + 2l\sigma\epsilon_2\).   The area bounded by the curve \(\Theta_\sigma\) is \(\frac{4\sqrt{2}}{3}\sigma l^{3/2}\).   The two curves intersect each other at \(x = \epsilon_2^2 / 2\).   Therefore, we can find a Hamiltonian isotopy supported in \(- \epsilon_2^2 \leq x \leq 10\epsilon_2^2\) which deforms \(\Theta_\sigma\) into a smoothed version of \(c_\sigma\) for \(0 \leq x \leq \frac{1}{2}\epsilon_2^2\) and so that the first and second derivatives of the new curve are bounded by \(K\sigma\) for some \(K > 0\).   Therefore, we can lift this isotopy to a Legendrian isotopy, giving us a new curve \(\tilde{c}_\sigma\) parameterized by:
\begin{equation}\label{eqn:1Droundedcuspform}
s \to \left(\frac{1}{2}s^2, \gamma_\sigma(s), \Psi_\sigma(s)\right)
\end{equation}
Where \(\gamma_\sigma, \Psi_\sigma\) have the properties:
\[
\gamma_\sigma(0) = \Psi_\sigma(0) = 0
\]\[
\gamma_\sigma(s) = \sigma s\mbox{ for } |s| \geq 100\epsilon_2^2
\]\[
\Psi_\sigma(s) = \frac{1}{3}\sigma s^3\mbox{ for } |s| \geq 100\epsilon_2^2
\]
We can extend this Legendrian isotopy naturally to \(\dim \Lambda > 1\), giving us a Legendrian isotopy restricted to a very small neighborhood of our cusp edge which changes our folded sheets from equation~\ref{eqn:foldedsheetform} to the form:
\[
(q, s) \to \left(q, \frac{1}{2}s^2,\sigma\left(\partial_qa + \frac{1}{2}s^2\partial_q\beta\right) + \Psi_\sigma(s)\partial_q\alpha, \right.
\]\begin{equation}\label{eqn:manyDroundedcuspform}\left.\sigma\beta + \alpha\gamma_\sigma(s), \sigma\left(a +\frac{1}{2}s^2\beta\right) + \Psi_\gamma(s)\alpha\right)
\end{equation}

{\color{black}Once again, we interpolate between the result of this isotopy and \(\tilde{\Lambda}_\sigma\) over \(\pi_M(U(2, \sigma\epsilon_1) - U(2, \frac{1}{2}\sigma\epsilon_1))\).}

Formally, we denote this deformed Legendrian by \(\Lambda_\sigma(\epsilon_0, \epsilon_1, \epsilon_2)\).   We will ordinarily omit the deformation parameters, and write it simply as \(\Lambda_\sigma\).

To sum up: our deformed Legendrian submanifold \(\Lambda_\sigma\) will be equal to the cofiber-scaling of \(\Lambda\) away from the {\color{black}singularities of \(\pi_F\)}.   Near the {\color{black}singularities but outside of a continuously-shrinking neighborhood of \(\Sigma_2\)}, sheets that are not folded in the cusp edge will be approximated linearly, while sheets that are folded will be approximated by their degree-3 Taylor polynomial.   Finally, in an even smaller neighborhood inside of that neighborhood, the cusp edges will be replaced by semicircles of radius \(\mathcal{O}(\sigma)\) in the Lagrangian projection.

\subsection{The Monotonicity Lemma}
\label{subsec:monotonicitylemma}

The purpose of these deformations is to ensure that we have a useable version of the monotonicity lemma.   This section is based on the treatment in \cite{Ek}, Section 5.1, though made more explicit, and in \cite{AL}, Chapter 5, Section 4, though \cite{AL} considers only fixed Lagrangians.

Let \(p \in T^*M\), and let \(B(p, r)\) be an \(r\)-ball around \(p\).   The standard monotonicity lemma says that if we have a \(J\)-holomorphic map \(u:(D, \partial D) \to (B(p, r), \partial B \cup \pi_\mathbb{C}(\Lambda_\sigma))\), then there exists a constant \(C' > 0\) such that:
\[
\mbox{Area}(u(D)) \geq C'r^2
\]
The problem is that \(C'\) depends on \(\pi_\mathbb{C}(\Lambda_\sigma)\), which in turn depends on \(\sigma\).

Even with our deformations, we cannot retrieve the full version of the monotonicity lemma.   What we \emph{can} do, however, is get a version that is good enough.

\begin{remark}\label{r_0remark} Recall that by the definition of a tame almost complex symplectic manifold, there exists constants \(r_0\) and \(C_1\) such that if \(\gamma\) is a loop in \(T^*M\) contained in a ball \(B(x, r), r \leq r_0\), then \(\gamma\) bounds a disc in \(B(x, r)\) of area less then \(C_1\mbox{length}(\gamma)^2\).\end{remark}

Then:

\begin{lemma}\label{lemma5.2} There exists a constant \(C_2\) such that, for any Riemannian surface with boundary \(D\) and any pseudoholomorphic map \(f:D \to T^*M\), if \(f(D) \subset B(p, r_0)\) then:
\begin{equation}\label{eqn:upperboundonarea}
\textnormal{area}(f(D)) \leq C_2\left(\textnormal{length}(f(\partial D))\right)^2
\end{equation}
In addition, there exists a constant \(C_3\) such that, if \(f(\partial D) \subset \partial B(p, r_0)\), and \(p \in f(D)\), then:
\begin{equation}\label{eqn:standardmonotonicity}
\textnormal{area}(f(D)) \geq C_3 r^2
\end{equation}\end{lemma}

\begin{proof} See \cite{AL}, Chapter 5, Proposition 4.3.1.i and ii.   In particular, the second statement is the general form of the monotonicity lemma for symplectic manifolds.\end{proof}

Note that \(C_1, C_2, C_3\) do not depend on \(\sigma\), or even on \(\Lambda\).   They are properties of the symplectic manifold \(T^*M\) and its associated symplectic form, almost complex structure, and Riemannian metric.

\begin{lemma}\label{lemma5.4} If \(\epsilon_2\) is small enough, then there exists a constant \(C_4\) independent of \(\sigma\) such that, if \(D\) is a disk in \(\mathbb{C}\) and \(u:(D, \partial D) \to (B(x, r), \partial B \cup \pi_\mathbb{C}(\Lambda_\sigma))\) is a pseudoholomorphic map, {\color{black} \(B(x, r)\) is disjoint from \(N(2, \sigma\epsilon_1)\), and} \(r \leq \sigma r_0\), where \(r_0\) is the \(r_0\) defined in remark \ref{r_0remark}, then:
\begin{equation}\label{eqn:upperboundonareawithlegendrian}
\mbox{Area}(f(D)) \leq C_4r^2
\end{equation}\end{lemma}

\begin{proof} In what follows, we will refer to the ``cusp edges" and ``sheets" of the Legendrian, even though we are working with the Lagrangian projection \(\pi_\mathbb{C}(\Lambda)\).   These both refer to the projection of the corresponding objects in the front projection to the Lagrangian projection.   {\color{black}Note that, since we require \(B(x, r)\) to be disjoint from \(N(2, \sigma\epsilon_1)\), we know that \(B(x, r)\) contains no higher singularities of the front projection, and we can assume that the cusp edges and sheets are ``rounded'' in the sense of section~\ref{subsec:deformation_cusprounding}.}

Choose \(\delta \leq r_0\) small enough that no solid ball in \(T^*M\) of radius \(\delta\) intersects more then one sheet of \(\pi_\mathbb{C}(\Lambda_1)\) unless it contains a cusp edge or a double point, in which case it intersects at most two sheets.   In addition, choose \(\delta\) small enough that we can always find a coordinate chart of \(T^*M\) containing any solid ball of radius \(\delta\).   The subscript \(\Lambda_1\) is used to indicate \(\sigma = 1\).   If \(\epsilon_2\) is small enough, and \(\delta < \epsilon_2\), then any ball \(B(p, \sigma\delta), p \in \pi_\mathbb{C}(\Lambda_\sigma)\), will intersect \(\pi_\mathbb{C}(\Lambda_\sigma)\) in at most one sheet unless it contains a cusp edge or a double point, in which case it will intersect at most two sheets, thanks to the cusp rounding conducted in section~\ref{subsec:deformation_cusprounding}.

We can regard \(\pi_\mathbb{C}(\Lambda_\sigma)\) as the union of the graphs of a collection of sheet functions:
\[
f_i: \pi_M(U_i) \to \mathbb{R}^n
\]
Where \(U_i \subset \Lambda_\sigma\).   Because \(\Lambda_\sigma\) is compact, outside of the \(\sigma\epsilon_2\)-neighborhood of the cusp edges, we can bound \(|\nabla f_i| \leq C'\sigma\) for some constant \(C'\).   In addition, inside the \(\sigma\epsilon_2\)-neighborhood of \(\pi_\mathbb{C}(\Sigma_1)\), but outside of a \(\sigma\epsilon_2\)-neighborhood of \(\pi_\mathbb{C}(\Sigma_2)\), we can instead describe \(\pi_\mathbb{C}(\Lambda_\sigma)\) as the graph of a function \(g_i(v_1, u_2, ..., u_n)\), and bound \(|\nabla g_i| \leq C''\).

Let \(\rho_1, \rho_2 \in \partial B(p, r) \cap \pi_\mathbb{C}(\Lambda_\sigma), r \leq \sigma\delta\).   We begin with the case where \(B(p, r)\) does not intersect a double point or the \(\sigma\epsilon_2\)-neighborhood of a cusp edge.   Let \(f_i:\pi_M(B) \to \mathbb{R}^n\), where \(f_i = (f^1_i, ..., f^n_i)\), be the sheet function.   Pick a coordinate chart in \(M\) around \(\pi_M(p)\), with \(\pi_M(\rho_i) = (\rho^1_i, ..., \rho^n_i)\).   Then we can define a path \(\gamma:[0,1] \to \pi_\mathbb{C}(\Lambda_\sigma)\) that connects \(\rho_1\) to \(\rho_2\) by:
\[
\gamma(t) = \left(t\rho^1_1 + (1-t)\rho^1_2, ..., t\rho^n_1 + (1-t)\rho^n_2, f_i(t\pi_M(\rho_1) + (1-t)\pi_M(\rho_2))\right)
\]\[
\gamma'(t) = \left(\pi_M(\rho_1) - \pi_M(\rho_2), \nabla f^1_i \cdot (\pi_M(\rho_1) - \pi_M(\rho_2)), ..., \nabla f^n_i \cdot (\pi_M(q_1) - \pi_M(q_2))\right)
\]\[
|\gamma'(t)|^2 \leq (1+(C'\sigma)^2)(2r)^2
\]
Therefore, \(\mbox{length}(\gamma) \leq 2r\sqrt{1+(C'\sigma)^2}\).

A similar argument allows us to bound the length of a path by \(4r\sqrt{1+(C'\sigma)^2} \) if the ball intersects a double point.

If the ball intersects the \(\sigma\epsilon_2\)-neighborhood of the cusp edge, then pick a coordinate chart in \(T^*M\) corresponding to the product neighborhood of the cusp edge, with \(\rho_i = (\rho^1_i, ..., \rho^n_i, \nu^1_i, ..., \nu^n_i)\), and let \(g\) be the cusp edge function.   We define \(\pi'\) to be the projection from \((u, v) \to (u_2, ..., u_n, v_1)\).   We can define a path \(\gamma:[0,1] \to \pi_\mathbb{C}(\Lambda_\sigma)\) that connects \(\rho_2\) to \(\rho_1\) by:
\[
\gamma(t) = \left(g_1(t\pi'(\rho_1) + (1-t)\pi'(\rho_2)), t\rho^2_1 + (1-t)\rho^2_2, ..., t\rho^n_1 + (1-t)\rho^n_2,\right.
\]\[
\left. t\nu^1_1 + (1-t)\nu^1_2, g_2(t\pi'(\rho_1) + (1-t)\pi'(\rho_2)), ...
\right)
\]\[
\gamma'(t) = \left(\nabla g_1 \cdot (\pi'(q_1) - \pi'(q_2)), q^2_1 - q^2_2, ..., q^n_1 - q^n_2, r^1_1 - r^1_2, \nabla g_2 \cdot (\pi'(q_1) - \pi'(q_2)), ...\right)
\]\[
|\gamma'(t)|^2 \leq (1 + (C'')^2)(2r)^2
\]
Therefore, \(\mbox{length}(\gamma) \leq 2r \sqrt{1 + (C'')^2}\).   \(C''\) is independent of \(\sigma\) because, as a result of the cusp rounding and the fact we are working in an arbitrarily small neighborhood of the cusp edge, \(g\) is independent of \(\sigma\).   Therefore, we can find \(C\) such that, for any ball of radius \(r \leq \delta\sigma\), any any two points \(\rho_1, \rho_2 \in B(p, r) \cap \pi_\mathbb{C}(\Lambda_\sigma)\), there exists a path in \(\pi_\mathbb{C}(\Lambda_\sigma)\) linking \(\rho_1\) to \(\rho_2\) of length less then or equal to \(Cr\).

Now let \(B(p, r), r \leq \delta\sigma\) be an arbitrary ball intersecting \(\Lambda_\sigma\), and let \(u\) be a pseudoholomorphic map of an open disk \((D, \partial D) \to (B(p, r), \partial B \cup \pi_\mathbb{C}(\Lambda_\sigma))\).   Then \(\partial(u(D)) - \pi_\mathbb{C}(\Lambda_\sigma)\) is a union of arcs \(\alpha_j\) in \(\partial B(p, r)\) from \(\pi_\mathbb{C}(\Lambda_\sigma)\) to \(\pi_\mathbb{C}(\Lambda_\sigma)\), which have length less then or equal to \(\pi r\).   Let \(\alpha'_j\) be arcs in \(\pi_\mathbb{C}(\Lambda_\sigma) \cap B(p, r)\) closing \(\alpha_j\).   Then they have length less then or equal to \(Cr\).   Therefore, by lemma~\ref{lemma5.2}, \(\alpha_j \cup \alpha'_j\) bound a union of disks \(W\) of area less then or equal to \(C_2r^2\), where \(C_2\) is a constant that depends only on \(T^*M\).   Therefore, \(W \cup u(D)\) is a surface with boundary on \(\partial B(p, r) \cap \pi_\mathbb{C}(\Lambda_\sigma)\).   

\(B(p, r) \cap \pi_\mathbb{C}(\Lambda_\sigma)\) must be a disk or pair of disks by our restrictions on \(r\), so it is contractible.   By the fact that \(B(p, r) \cap \pi_\mathbb{C}(\Lambda_\sigma)\) is contractible and the fact that \(||w|| \leq 1\), we obtain that the area of \(u(D)\) is less then or equal to \(C_4r^2\), where \(C_4\) is another constant that depends only on \(T^*M\).\end{proof}

\begin{lemma}[Monotonicity Lemma]\label{lemma5.5}If \(\epsilon_2\) is small enough, then there exists constants \(C, \delta > 0\) which are independent of \(\sigma\) such that, if \(0 < r \leq \sigma\delta\), then for any non-constant \(J\)-holomorphic map \(u:(D, \partial D) \to (B(p, r), \partial B \cup \pi_\mathbb{C}(\Lambda_\sigma))\), such that \(p\) is in the image of \(u\), {\color{black}\(B(p, r)\) is disjoint from \(N(2, \sigma\epsilon_1)\),} and \(D\) is an open surface, then:
\begin{equation}\label{eqn:scaledmonotonicitylemma}
\textnormal{Area}(u(D)) \geq Cr^2
\end{equation}\end{lemma}

\begin{proof} Define \(S_t = u^{-1}(B(x, t))\) for \(t \leq r\), \(a(t) = \mbox{area} (u|_{S^t})\).   Since \(u\) is smooth, Sard's theorem (\cite{St}, Theorem II.3.1) implies that \(S_t\) is a subsurface with \(C^1\) boundary \(\partial S_t = u^{-1}(\partial(B(x, t)))\) for almost all \(t\).   Define \(l(t) = \mbox{length}(u(\partial S_t))\).   \(a(t)\) is an absolutely continuous function, and \(a'(t) = l(t)\) for almost all \(t\).   Then by lemma~\ref{lemma5.4}, we have \(a(t) \leq C_4(l(t))^2\) for almost all \(t \leq r\).   Since \(u\) is not constant and the image of \(u\) contains \(p\), \(a(t) > 0\) for \(t > 0\).   Thus, for \(t > 0\) we know:
\[
\frac{d}{dt}\left(\sqrt{a(t)}\right) = \frac{a'(t)}{2\sqrt{a(t)}} \geq \frac{l(t)}{2\sqrt{C_4(l(t))^2}} = \frac{1}{2\sqrt{C_4}}
\]
Therefore, integrating the above from \(t = 0\) to \(r\), we obtain:
\[
\mbox{Area}(u(D)) = a(r) \geq \frac{r^2}{4C_4}
\]
Which implies equation~\ref{eqn:scaledmonotonicitylemma}.\end{proof}

\subsection{The Slit Model of the Disk}
\label{subsec:slitmodel}

Recall that \(D_m\) is the unit disk in \(\mathbb{C}\) with \(m\)-punctured boundary.   We define \(\Delta_m\) to be the subset of \(\mathbb{R} \times [0, m-1]\) given by deleting \(m-2\) horizontal slits of width \(\epsilon < 1\), where each slit ends in a half-circle.   \(\Delta_m\) has a canonical complex and symplectic structure inherited from \(\mathbb{R}^2 = \mathbb{C}\).

\begin{lemma}(\cite{Ek}, Lemma 2.2)\label{lemma5.6} \(\Delta_m\) is biholomorphic to \(D_m\).\end{lemma}

\begin{proof}  See \cite{Ek}, Lemma 2.2.\end{proof}

We will work in the slit model of the disk for most of what follows.

\subsection{Lemmas Bounding the Derivative}
\label{subsec:boundingderivative}

Recall from the introduction to this section that \(s_\sigma(x, y, z) = (x, \sigma y, \sigma z)\) is the scaling of the cofiber and \(z\) components by \(\sigma > 0\).   {\color{black}Let \(\tilde{J}_\sigma = (s_\sigma)_*^{-1} \circ J \circ (s_\sigma)_*\).   Then, there is a bijection between \(J\)-holomorphic disks on \(\pi_\mathbb{C}(\Lambda)\) and \(\tilde{J}_\sigma\)-holomorphic disks on \(\pi_\mathbb{C}(\tilde{\Lambda}_\sigma)\).   Furthermore, there exists a fiber-preserving diffeomorphism \(\Theta_\sigma:T^*M \to T^*M\) so that \(\Theta_\sigma(\tilde{\Lambda}_\sigma) = \Lambda_\sigma\), and \(\Theta_\sigma\) is \(\sigma\)-close to the identity map by the supremum norm on diffeomorphisms.   Define \(J_\sigma = (\Theta_\sigma)_*^{-1} \circ \tilde{J}_\sigma \circ (\Theta_\sigma)_*\).   Then there is a bijection between \(J_\sigma\)-holomorphic disks on \(\pi_\mathbb{C}(\Lambda_\sigma)\) and \(\tilde{J}_\sigma\)-holomorphic disks on \(\pi_\mathbb{C}(\tilde{\Lambda}_\sigma)\).

Since the space of almost complex structures \(J\) on \(T^*M\) which is part of a tame triple is contractible (\cite{AL}, Ch. 5, Sec. 4.1), the almost complex structure \(\tilde{J}_\sigma\) is part of a tame triple.   Given that \(\tilde{J}_\sigma\) is part of a tame triple and \(J_\sigma\) is \(\sigma\)-close to \(\tilde{J}_\sigma\), \(J_\sigma\) is also part of a tame triple.}

Let \(u:(\Delta_m, \partial \Delta_m) \to (T^*M, \pi_\mathbb{C}(\Lambda))\) be a pseudoholomorphic map.  Let \(u_\sigma\) be the corresponding map \((\Delta_m, \partial\Delta_m) \to (T^*M, \Lambda_\sigma)\), which can be parameterized by \(u_\sigma = (p_\sigma, q_\sigma)\), where \(p_\sigma\) is the map to the base space and \(q_\sigma\) is the map to the cofiber.

Recall from section~\ref{subsec:dga} {\color{black}and \ref{subsec:deformation_metric}} that \(\Sigma_k \subset \Lambda\) denotes the codimension-\(k\) component of the singularities of the front projection \(\pi_F\), that \(U(k, d) \subset \Lambda\) is the product neighborhood of radius \(d\) around \(\Sigma_k\) in \(\Lambda\), that \(N(k, d) \subset T^*M\) is the product neighborhood of radius \(d\) around \(\Sigma_k\) in \(T^*M\).

We analogously have \((\Sigma_k)_\sigma, (U(k, d))_\sigma \subset \Lambda_\sigma, (N(k, d))_\sigma \subset T^*M\), etc.; however, in what follows we will omit the subscript.   There is a projection map \(\pi_\Sigma:T^*U(1, d) \to T^*\Sigma_1\).   Define \(u^T = \pi_\Sigma \circ u, u^T_\sigma = \pi_\Sigma \circ u_\sigma\), with \(u, u_\sigma\) restricted to the pre-images of \(T^*U(1, d)\).

Throughout all that follows, we will restrict \(u_\sigma\) to vertical lines contained in the pre-image of \(T^*M - T^*N(2, {\color{black}\sigma}\epsilon_3)\), where \(\epsilon_3\) is arbitrarily small.   We denote the restricted domain of vertical lines by \(D_\sigma\).

Next, we put bounds on \(|Du_\sigma|\).   To do this, we will use \(O\)-notation, where \(f(\sigma) = \mathcal{O}(\sigma)\) means that:
\[
\lim_{\sigma \to 0} \frac{f(\sigma)}{\sigma} \mbox{ is finite}
\]

\begin{lemma}(\cite{Ek}, Lemma 5.2 and 5.4)\label{lemma5.7} For any \(J\)-holomorphic map \(u: (\Delta_m, \partial\Delta_m) \to (T^*M, \Lambda)\), there exist constants \(C', C'' > 0\) independent of \(\sigma\) such that the symplectic area of \(u_\sigma(D_\sigma)\) is less then \(C'\sigma\) and the length of \(u_\sigma(D_\sigma \cap \partial\Delta_m)\) is less then \(C''\) for any \(\sigma\).\end{lemma}

\begin{proof}  By Stokes' Theorem:
\[
\int_{u(D_\sigma)} \omega = \int_{\partial u(D_\sigma)} \beta
\]
{\color{black}Which we can then separate into compnents:
\[
\int_{\partial u_{\color{black}\sigma}(D_\sigma)} \beta = 
\int_{\partial u_{\color{black}\sigma}(D_\sigma) \cap (M - N_\sigma(1, \epsilon_1))} \beta + 
\int_{\partial u_{\color{black}\sigma}(D_\sigma) \cap N_\sigma(1, \epsilon_1)} \beta
\]
Over \(M - N_\sigma(1, \epsilon_1)\), \(u_\sigma\) is equal to the \(\sigma\)-scaling of \(u\).   Since \(\beta\) is the tautological one-form and \(|q_\sigma|\) is scaled by \(\sigma\), the first integral is less than or equal to \(C'_1\sigma\) for some \(C'_1 > 0\).   Over \(N_\sigma(1, \epsilon_1)\), we apply Lemma~\ref{lemma5.4}, which implies that the second integral is less than or equal to \(C'_2\sigma^2\), for some \(C'_2 > 0\).   This provides us with the result.}

We obtain the second result as follows: let \(l_\sigma\) denote the length of \(u_\sigma(\partial\Delta_m \cap D_{\color{black}\sigma}) \subset T^*M\).   For \(\rho > 0\) small enough we can find \(m\) disjoint solid balls of radius \(\rho\sigma\) whose centers lie on \(u_\sigma(\partial\Delta_m \cap D_\sigma)\), where \(m = l_\sigma / 2\sigma\) rounded down.   Lemma~\ref{lemma5.5} then shows that:
\[
\mbox{area}(D_\sigma) \geq \left(\frac{l_\sigma}{2\sigma}\right)(C(\rho\sigma)^2) = \frac{1}{2}Cl_\sigma\rho^2\sigma
\]
Since \(\mbox{area}(D_\sigma) \leq C\sigma\), this shows that \(l_\sigma\) must have some maximum.\end{proof}

\begin{lemma}(\cite{EES}, Theorem 9.4)\label{lemma5.8} Fix \(q \geq 1\) and \(\delta_{k-1} > \delta_k \geq 0\).   Let \(A\) be the open unit disk in \(\mathbb{C}\) or the half disk with boundary on \(\mathbb{R}\).   For any compact \(K \subset A\), there exsts a constant \(C\) such that for all holomorphic maps \(u\in W^{k, 2+\delta_{k-1}}(A, \mathbb{C}^n)\), we have:
\begin{equation}\label{eqn:Wbound}
||u||_{W^{k+1,2+\delta_k}(K)} \leq C||u||_{W^{k,2+\delta_k}(A)}
\end{equation}\end{lemma}

\begin{proof}  See \cite{EES}, Theorem 9.4.\end{proof}

Define \(G_r\) to be the open \(r\)-disk centered at the origin in \(\mathbb{C}\), and \(H_r = G_r \cap \{x \geq 0\}\).

There exists \(d > 0\) such that, if \(p, q \in \Sigma_1\), and \(\pi_M(p) = \pi_M(q) \in \pi_M(\Sigma_2)\), then the distance between \(p, q\) is greater then \(d\) when \(\sigma = 1\).   Let \(S_d \subset \Sigma_1\) be an arbitrary open set with diameter less then \(d\).   Choose \(\delta > 0\) small enough that \(S_d\) has a tubular neighborhood of radius \(\delta\).   Then let \(N(\Pi(S_d), \delta)\) denote the restriction to \(S_d\) of a \(\delta\)-tubular neighborhood in the product neighborhood.   Define \(V(S_d, \delta) \subset \Lambda_\sigma\) to be the connected segment of \(\pi_M^{-1}(N(\Pi(S_d), \delta))\) containing \(S_{\color{black}d}\).

\begin{lemma}(\cite{Ek}, Lemma 5.6)\label{lemma5.9} There exists \(C > 0\) such that if \(u:G_2\to T^*M\) is pseudo-holomorphic then:
\[
\sup_{G_1} |Du| \leq C||Du||_{L^2,G_2}
\]
In addition, for all \(K > 0\) large enough there exists \(C > 0\) such that if \(u: (H_2, \partial H_2) \to (T^*M, \pi_\mathbb{C}(\Lambda_\sigma))\) is pseudoholomorphic, and if \(u(\partial H_2)\) lies outside a \(K\sigma\)-neighborhood of \(\pi_\mathbb{C}(\Sigma_1) \subset \pi_\mathbb{C}(\Lambda_\sigma)\), then:
\[
\sup_{H_1} |Du| \leq C||Du||_{L^2,H_2}
\]
Finally, for all \(d > 0\) small enough, if \(u(H_2, \partial H_2) \to (T^*M, V(S, \delta))\) is a pseudoholomorphic map such that \(u(\partial H_2) \subset T^*(N(\Pi(S_d), \delta))\) then:
\[
\sup_{H_1} |Du^T| \leq C ||Du^T||_{L^2, H_2}
\]\end{lemma}

\begin{proof} We begin with the first case.   If \(||Du||_{L^2, G_2} = \infty\) the inequality is trivially true, so we assume it does not.   Recall from \cite{Dr}, Theorem 27.18, that the Sobolev inequality for \(k > n / q\) is:
\[
||f||_{C^{k - \lfloor n/p\rfloor-1, \gamma}(U)} \leq C||f||_{W^{k,p}(U)}
\]
Where \(n\) is the dimension of the space, \(U\) is a bounded open subset of \(\mathbb{R}^n\) with \(C^1\) boundary, \(C^{m,\gamma}(U)\) is a Holder space, and \(\gamma\) is defined by:
\[
\gamma = \left\lfloor\frac{n}{p}\right\rfloor + 1 - \frac{n}{p}
\]
Where \(\lfloor x \rfloor\) denotes the floor of \(x\).   In our case \(f = Du\), \(U = G_1\) is a subset of \(\mathbb{C} = \mathbb{R}^2\), so \(n = 2\), and we choose \(k = 1\).   Therefore, for \(q > 2\):
\[
||Du||_{C^{0, \gamma}(G_1)} \leq C_1 ||Du||_{W^{1,q}(G_1)}
\]
Since \(||v||_{C^{0, \gamma}(U)} \geq ||v||_{C^{0}(U)}\) for any \(\gamma, U\), this gives us:
\begin{equation}\label{eqn:supofG1}
\sup_{z \in G_1}|Du| \leq C_1 ||Du||_{W^{1,q}(G_1)}
\end{equation}
Then, we apply Lemma~\ref{lemma5.8} to \(Du\) using \(K=G_1, A = G_r, 1 < r < 2\).   From this we obtain:
\begin{equation}\label{eqn:WofG1}
||Du||_{W^{1,q}(G_1)} \leq C_2(r)||Du||_{W^{0,q}(G_r)} = C_2(r)||Du||_{L^{q}(G_r)}
\end{equation}
Further, we have:
\[
||Du||^q_{L^q(G_r)} = \int_{G_r} |Du|^q \leq \left(\sup_{z\in G_r} |Du|^{q-2}\right) \left(\int_{G_r} |Du|^2\right) \leq
\left(\sup_{z\in G_r} |Du|^{q-2}\right) \left(\int_{G_2} |Du|^2\right)
\]\begin{equation}\label{eqn:supofGr}
||Du||_{L^q(G_r)} \leq \sup_{z\in G_r} |Du|^{(q-2)/q}\cdot ||Du||_{L^2(G_2)}^{2/q} 
\end{equation}
Combining equations~\ref{eqn:supofG1},~\ref{eqn:WofG1}, and~\ref{eqn:supofGr} we have:
\[
\sup_{z\in G_1} |Du| \leq C_3(r) ||Du||_{L^2(G_2)}^{2/q}
\]
Where \(C_3(r)\) is defined by:
\begin{equation}\label{eqn:C3(r)}
C_3(r) = C_1C_2(r)\sup_{z\in G_r} |Du|^{(q-2)/q}
\end{equation}
The limit of \(C_2(r)\) as \(r \to 1\) is given by:
\begin{equation}\label{eqn:limC2(r)}
\lim_{r \to 1} C_2(r) = \frac{||Du||_{W^{1, q}(G_1)}}{||Du||_{W^{0, q}(G_1)}}
\end{equation}
Combining equation~\ref{eqn:C3(r)}, equation~\ref{eqn:limC2(r)}, and the fact that the supremum of \(|Du|\) will be strictly declining as \(r\) shrinks, tells us that the limit of \(C_3(r)\) as \(r \to 1\) must be defined and finite.   Therefore the same inequality holds with \(r = 1\), giving us:
\[
\sup_{z\in G_1} |Du| \leq C ||Du||_{L^2(G_2)}
\]
The process is precisely analogous for the other cases.\end{proof}

Recall from {\color{black}earlier in the section} that \(D_\sigma\) is the closed set in \(\Delta_m\) made of vertical lines \(l\) such that \(u(l)\) lies outside a {\color{black}\(\sigma\)}-small neighborhood of \(\Sigma_2\).   Let \(A_r(p)\) denote the points in \(D_\sigma\) which are connected to \(p\) by a path in \(D_\sigma\) of length at most \(r\).

\begin{lemma}(\cite{Ek}, Lemma 5.7)\label{lemma5.10} If \(\epsilon_2\) and \(\sigma\) are small enough, if \(d_1, d_2 > 0\) are small enough, and if \(C_1 > 0\) is big enough, then there exists \(C = C(\epsilon_2) > 0\) such that:

\begin{itemize}

\item Let \(p \in D_\sigma\), and at least \(4d_2\) distance from \(\partial D_\sigma\).   Then:
\[
\sup_{B_{d_2}(p)} |Du_{\color{black}\sigma}| \leq C\sigma
\]
\item If \(p \in \partial\Delta_m \cap D_\sigma\) and \(u(A_{4d_2}(p) \cap \partial \Delta_m)\) is outside of a \(C_1\sigma\)-neighborhood of \(\pi_\mathbb{C}(\Sigma_1) \subset \pi_\mathbb{C}(\Lambda_\sigma)\), then:
\[
\sup_{A_{d_2}(p)} |Du_{\color{black}\sigma}| \leq C\sigma
\]
\item If \(p \in \partial\Delta_m \cap D_\sigma\) and \(u(A_{4d_2}(p) \cap \partial \Delta_m) \subset V(S, \delta)\) and \(u(A_{4d_2}(p)) \subset T^*N(\Pi(S), \delta)\) then:
\[
\sup_{A_{d_2}(p)} |Du^T_{\color{black}\sigma}| \leq C\sigma
\]

\end{itemize}\end{lemma}

\begin{proof}  We begin with the second case.   There is a biholomorphic map with uniformly bounded derivatives from \(A_{4d_2}(p)\) to \(H_{4d_2}\).   Then, by Lemma~\ref{lemma5.9}, there is a constant \(C'\) such that for all \(z \in H_{2d_2}, |Du_\sigma(z)| \leq C'||Du_\sigma||_{L^2(H_{4d_2})}\), and:
\[
||Du_\sigma||^2_{L^2(H_{4d_2})} = \int_{H_{4d_2}} |Du_\sigma|^2 = \int_{H_{4d_2}} g(Du_\sigma, Du_\sigma) = \int_{H_{4d_2}} \omega(Du_\sigma, JDu_\sigma) = \int_{H_{4d_2}} u_\sigma^*\omega
\]
Therefore \(||Du_\sigma||_{L^2(H_{4d_2})}\) is the square root of the symplectic area of \(u_\sigma(H_{4d_2}(p))\).   Since \(H_{4d_2}(p) \subset D_\sigma\), and the symplectic area of \(u_\sigma(D_\sigma)\) is less than the action of the positive Reeb chords, and the action of the positive Reeb chords is scaled by \(\sigma\), we obtain \(|Du_\sigma(z)| \leq C''\sigma^{1/2}\) for all \(z \in H_{2d_2}\).

We define a norm on the space of linear operators.   For a given point \(p \in {\color{black}T^*M}\) and a linear operator {\color{black}\(J_p:T_p(T^*M) \to T_p(T^*M)\)}, we define:
\[
|J_p| = \sup_{v \in T_p{\color{black}(T^*M)}, |v| = 1} \left\{ |Jv| \right\}
\]
And for an operator \({\color{black}J:T(T^*M) \to T(T^*M)}\), we define:
\[
|J| = \sup_{p\in {\color{black}N(\Lambda)}} \left\{|J_p|\right\}
\]
{\color{black}Where \(N(\Lambda)\) is an arbitrary compact neighborhood in \(T^*M\) containing \(\pi_\mathbb{C}(\Lambda_\sigma)\) for all \(\sigma\).}

Pick a complex coordinate chart around \(u_\sigma(p)\) that agrees with the ambient almost complex structure at \(u_\sigma(p)\).   We denote the complex structure from the coordinate chart by \(\tilde{J}\).   Then, for any \(q \in B_{C''\sigma^{1/2}}(u(p))\), \(|\tilde{J}_q - (J_\sigma)_q| = \mathcal{O}(\sigma^{1/2})\).

Let the sheet of \(\pi_\mathbb{C}(\Lambda_\sigma)\) in \(T^*M\) containing \(u_\sigma(p)\) be the graph \(\Gamma_{\sigma\alpha}\) of a 1-form \(\sigma\alpha\).   We claim that for all \(\sigma\) small enough, we can find a diffeomorphism \(\Theta\) such that:

\begin{itemize}

\item \(\Theta(\Gamma_{\sigma\alpha})\) equals the 0-section, and is therefore real analytic

\item \(d\Theta + J_\sigma \circ d\Theta \circ \tilde{J} = 0\) along \(\Gamma_{\sigma\alpha}\)

\item Let \(d_{C^1}\) be the sup norm metric on \(C^1\).   \(d_{C^1}(\Theta, \mbox{Id}) \leq \eta\) for arbitrarily small \(\eta\).

\end{itemize}

We prove the existence of \(\Theta\) by calculating its inverse.   Let \(z = x + iy\) be the coordinates of the local complex chart such that the 0-section corresponds to \(\{y = 0\}\).   Define \(\psi(x) = (x, \sigma\alpha_x)\).   Then, for very small \(y\), define:
\[
\Theta^{-1}(x, y) = \psi(x) + \sum_i \left(y_i \cdot J_{\psi(x)} \psi_* \partial_{x_i}\right)
\]
And extend the map arbitrarily to the rest of the chart.   Then \(\Theta^{-1}(x, 0) = \psi(x)\), so it sends the zero section to the graph \(\Gamma_{\sigma\alpha}\), and it is \(C^1\)-close to the identity.   And, along the zero-section:
\[
\Theta^{-1}_* \partial_{x_i} = \psi_* \partial_{x_i}
\]\[
\Theta^{-1}_*\partial_{y_i} = (J_\sigma)_{\psi(x)} \psi_* \partial_{x_i}
\]
So \(d\Theta^{-1} + \tilde{J} \circ d\Theta^{-1} \circ J_\sigma = 0\).   This gives us \(\Theta\).

Define \(J_\Theta = \Theta_* \circ J_\sigma \circ \Theta^{-1}_*\).   Then \(|J_\Theta - \tilde{J}| = \mathcal{O}(\eta)\), and \(\Theta \circ u\) is \(J_\Theta\)-holomorphic.

Define \(\hat{u}_\sigma = \sigma^{-1}\Theta \circ u_\sigma\) and \(\hat{J}(x, y) = J_\Theta(\sigma x, \sigma y)\).   Then \(|\hat{J} - \tilde{J}| = \mathcal{O}(\eta)\) and \(\hat{u}_\sigma\) is \(\hat{J}\)-holomorphic.   Furthermore, \(\hat{u}_\sigma\) is \(\hat{J}\)-holomorphic on \(\partial H_{4d_2}(p)\).   We extend \(\hat{u}_\sigma\) by copying it over the 0-section and the boundary, obtaining a map \(\tilde{u}_\sigma: D\to\mathbb{C}^n\), where \(D\) is a disc, \(\tilde{u}_\sigma\) is \(C^1\), and \(d\tilde{u}_\sigma + \hat{J} \circ d\tilde{u}_\sigma \circ \tilde{J} = 0\), where \(\hat{J}\) is extended over the new area by \(\hat{J}(z) = \hat{J}\overline{(\hat{u}(\bar{z}))}\), where the bars denote the complex conjugate.

Now define a map \(F:\mathbb{C}^n \to \mathbb{C}^n\) by:
\[
F(z_1, ..., z_n) = \left(e^{iz_1}, ..., e^{iz_n}\right)
\]
And define \(f_\sigma = F \circ \tilde{u}_\sigma\).

Let \(\hat{y}_\sigma, \tilde{y}_\sigma\) be the imaginary components of \(\hat{u}_\sigma, \tilde{u}_\sigma\).   Since \(\hat{u}_\sigma\) is pseudoholomorphic, the function \(|\hat{y}_\sigma|^2\) is subharmonic, and is therefore bounded.   Since \(\tilde{u}_\sigma\) is the doubling of \(\hat{u}_\sigma\), the same is true of \(|\tilde{y}_\sigma|^2\).   Therefore, \(|f_\sigma|\) is uniformly bounded, and the derivatives of \(F\) are uniformly bounded in a neighborhood of the image of \(\tilde{u}_\sigma\).   Furthermore:
\[
df_\sigma \circ \tilde{J} - J_F \circ df_\sigma = 0
\]
Where \(J_F = dF \circ \hat{J} \circ dF^{-1}\).   Since \(d = \partial_{\tilde{J}} + \bar{\partial}_{\tilde{J}}\), this implies:
\[
-\partial_{\tilde{J}} f_\sigma + \bar{\partial_{\tilde{J}}} f_\sigma + J_F \circ (\bar{\partial_{\tilde{J}}} f_\sigma - \partial_{\tilde{J}} f_\sigma) = 0
\]\[
(\tilde{J} + J_F) \circ \bar{\partial_{\tilde{J}}}f_\sigma + (\tilde{J} - J_F) \circ \partial_{\tilde{J}}f_\sigma = 0
\]
Define \(q(z) = (\tilde{J} + J_F)^{-1}(\tilde{J} - J_F)\).   This then becomes:
\[
\overline{\partial_{\tilde{J}}}f_\sigma + q(z) \partial_{\tilde{J}}f_\sigma = 0
\]
Since \(F\) is \(\tilde{J}\)-holomorphic:
\[
|J_F - \tilde{J}| \leq |dF| |\hat{J} - \tilde{J}| |dF^{-1}| = \mathcal{O}(\eta)
\]
We obtain:
\[
q(z) = \mathcal{O}(\eta)
\]

Now let \(\beta:G_1 \to \mathbb{R}\) be a cutoff function which equals 1 on \(G_{\frac{1}{2}}\) and equals 0 outside \(G_{\frac{3}{4}}\), and define \(f^1_\sigma = \beta f_\sigma\).   Then, define:
\[
g_\sigma^1 = \bar{\partial}_{\tilde{J}} f^1_\sigma + q(z)\partial_{\tilde{J}} f^1_\sigma = 
\bar{\partial}_{\tilde{J}} (\beta f_\sigma) + q(z)\partial_{\tilde{J}} (\beta f_\sigma) =
\left(\bar{\partial}_{\tilde{J}} \beta + 
q(z)\partial_{\tilde{J}} \beta\right) f_\sigma
\]
Since \(\beta\) is compactly supported, and therefore its derivatives are bounded, and since \(q(z) = \mathcal{O}(\eta)\), we conclude \(|g^1_\sigma| \leq C|f_\sigma|\).   Therefore:
\begin{equation}\label{eqn:barpartialf1}
|\bar{\partial}_{\tilde{J}}f^1_\sigma| \leq C_0\eta|\partial_{\tilde{J}}f^1_\sigma| + C_1|f^1_\sigma|
\end{equation}
Provided \(\eta\) is small enough, we can rearrange equation~\ref{eqn:barpartialf1} to obtain:
\[
2|\bar{\partial}_{\tilde{J}}f^1_\sigma|^2 \leq C'_1|f_\sigma|^2
\]
Then add \(|\partial_{\tilde{J}}f^1_\sigma|^2 - |\bar{\partial}_{\tilde{J}}f^1_\sigma|^2\) to both sides to obtain:
\begin{equation}\label{eqn:finequality}
|\partial_{\tilde{J}}f^1_\sigma|^2 + |\bar{\partial}_{\tilde{J}}f^1_\sigma|^2 \leq \left(
|\partial_{\tilde{J}}f^1_\sigma|^2 - |\bar{\partial}_{\tilde{J}}f^1_\sigma|^2
\right)
+C'_1|f_\sigma|^2
\end{equation}
Since \(\omega\) is exact, by Stokes theorem:
\[
\int\left(|\partial_{\tilde{J}}f^1_\sigma|^2 - |\bar{\partial}_{\tilde{J}} f^1_\sigma|^2\right)dA = \int (f^1_\sigma)^*\omega = 0
\]
Therefore, if we integrate both sides of equation~\ref{eqn:finequality}, we obtain:
\[
||Df^1_\sigma||^2_{L^2} \leq C \sup |f_\sigma|^2
\]
Which is less than or equal to some constant \(K'\) since \(f_\sigma\) is bounded.   Note that, since the derivatives of \(F\) are uniformly bounded on a neighborhood of the image of \(\tilde{u}_\sigma\), we have:
\[
||D\hat{u}_\sigma||_{L^2} \leq C ||Df^1_\sigma||_{L^2}
\]
Therefore \(||D\hat{u}_\sigma||_{L^2}\) is bounded by some constant, and therefore, by scaling, \(||Du_\sigma||^2_{L^2} = \mathcal{O}(\sigma^2)\).   Combining this with Lemma~\ref{lemma5.10}, we obtain the result.   The first and third part of the lemma follow analogously.\end{proof}

\subsection{Subdivision of the Domain}
\label{subsec:domainsubdivision}

Pick \(\delta> 0\) small enough such that a \(4\delta\)-neighborhood of \(\Sigma_1 - \Sigma_2\) follows the form given in section A.1, and so that the intersection of \(u(\partial\Delta_m)\) with \(T^*N(1, c\delta)\) is transverse for \(c = 1, 2, 3, 4\).   Let \(b^c_i\) denote the points in \(\partial \Delta_m\) such that \(u(b^c_i) \in \partial (T^*N_{\color{black}\sigma}(1, c\delta))\).   Choose a monotone decreasing sequence \(\sigma_i\) so that \(\sigma_0 = 1, \sigma_i \to 0\), and there is a function \(\delta(\sigma)\) so that \(\delta(\sigma)\) lies in an arbitrarily small neighborhood of \(\delta\) and the intersection of \(u_{\sigma_i}(\partial\Delta_m)\) with \(T^*N(1, c\delta(\sigma_i))\) is transverse for \(c = 1, 2, 3, 4\). We then have a sequence of points \(b^c_i(\sigma_i)\) as well. For every map \(u_{\sigma_i}\), puncture the disk \(\Delta_m\) at every \(b^c_j(\sigma_i), b^c_{j+1}(\sigma_i)\) such that there is \(b^{c-1}_k(\sigma_i)\) so that \(b^c_j(\sigma_i) < b^{c-1}_k(\sigma_i) < b^c_{j+1}(\sigma_i)\) for \(c = 2\) or \(4\).   This gives us a new disk with more punctures in the domain but, except for the punctures, the same image; we will denote the domain of this map by \(\Delta_r\), but by an abuse of notation continue to use \(u_\sigma\) to denote the function. We will also suppress the \(\sigma_i\) in \(\delta(\sigma_i), b_j^c(\sigma_i)\), and write simply \(\delta, b_j^c\). We will also generally suppress the subscript \(i\) in \(\sigma_i\).

We define the \textbf{boundary minimum} of a component \(I\) of \(\partial \Delta_r\) to be the point \((x, y) \in I\) with minimum \(x\) value.

In addition, recall from section~\ref{subsec:boundingderivative} that \(D_\sigma\) denotes the restriction of \(\Delta_r\) to vertical lines contained in the pre-image of \(T^*M - {\color{black}T^*N(2, \sigma\epsilon_3)}\), where \(\epsilon_3\) is arbitrarily small. We pass to a subsequence of \(\sigma_i\) where the topology of each component of \(D_{\sigma_i}\) is constant. At each stage, we will continue to pass to a subsequence as necessary to ensure these properties hold, without mentioning it explicitly.

\begin{lemma}(\cite{Ek}, Lemma 5.9)\label{lemma5.11} There exists a constant \(C > 0\) that does not depend on \(\sigma\) such that the number of added punctures is less than or equal to \(C\).\end{lemma}

\begin{proof} Each added puncture of \(\Delta_r\) corresponds to \(b^c_j\) such that there exists \(b^{c-1}_k\) so that \(b^c_j < b^{c-1}_k < b^c_{j+1}\) on \(\partial \Delta_m\) for \(c = 2, 3, 4\).   The length of the segment between \(b^{c-1}_k, b^c_{j+1}\) is bounded from below by the infinum of \(\delta(\sigma)\), and by lemma~\ref{lemma5.7} the length of \(u(\partial \Delta_r \cap D_\sigma)\) is bounded above for any \(\sigma\), so the result follows.\end{proof}

We now label the boundary components of \(\Delta_r\) by the following types:

\bigskip

\textbf{out:} \(u(I) \subset T^*(M- N(1, 3\delta))\)

\textbf{0:} \(u(I) \subset T^*(N(1, 4\delta) - N(1, \delta))\)

\textbf{in:} \(u(I) \subset T^*N(1,2\delta)\)

\bigskip

Then, for \(I\) a boundary component and \(0 < \rho < \frac{1}{4}\), define \(N_\rho(I)\) to be a \(\rho\)-neighborhood of \(I\) in \(\Delta_{\color{black}r}\), and:
\[
\Omega_\rho = D_\sigma - \bigcup_{I \subset \partial \Delta_r} N_\rho(I)
\]

Fix a small \(\epsilon > 0\).

\begin{lemma}(\cite{Ek}, Lemma 5.10)\label{lemma5.12} There exists \(C\) such that if \(\sigma > 0\) is small enough, then:
\[
\sup_{z \in \Theta_\epsilon} |Du(z)| \leq C\sigma
\]\[
\Theta_\epsilon = \Omega_\epsilon \cup \left(\bigcup_{I \in \mbox{\normalfont out} \cup 0} N_\epsilon(I)\right)
\]\end{lemma}

\begin{proof}  This is a direct consequence of lemma~\ref{lemma5.10}.\end{proof}

Let \(D'_\sigma \subset D_\sigma\) be the subset containing all vertical line segments in \(D_\sigma\) which connect a point in a boundary component of type \textbf{in} to some other boundary point.   Note \(\partial D'_\sigma - \partial D_\sigma\) is a collection of vertical line segments.

\begin{lemma}(\cite{Ek}, Lemma 5.11)\label{lemma5.13} For any \(0 < a < 1\) and sufficiently small \(\sigma > 0\), the distance from any \(p \in I\), where \(I\) is of type \textbf{out}, to \(D'_\sigma\) is larger than \(\sigma^{-a}\).   In particular, if \(l\) is a vertical line segment in \(\partial D'_\sigma - \partial D_\sigma\) and \(q \in \partial l\) then \(q\) is either a boundary minimum on a segment of type \textbf{in} or it lies on a boundary segment of type \textbf{0}.\end{lemma}

\begin{proof} Suppose \(p \in I\), \(I\) of type \textbf{out}, and the distance from \(p\) to \(D'_\sigma\) is less than \(\sigma^{-a}\).   Then there exists a path in \(\Omega_\epsilon \cup N_\epsilon(I) \subset D_\sigma\) of length less than \(\sigma^{-a} + 5r\), where \(r\) is the number of punctures of \(\Delta_r\), from \(p\) to a point \(q\) midway between two horizontal boundary segments of length 1, at least one of which is of type \textbf{in}.   Call the segment of type \textbf{in} \(I_1\) and the other \(I_2\).

Since \(|Du_\sigma| = \mathcal{O}(\sigma)\) along this path and \(u_\sigma(p) \in T^*(M - N(1, 3\delta))\), we know that:
\[
u_\sigma(q) \in T^*(M - N(1, 3\delta - (5r\sigma + \sigma^{1-a}))
\]
Therefore, for \(\sigma\) small enough, \(u_\sigma(q) \in T^*(M - N(1, \frac{5}{2}\delta))\).   Let \(I_3 \subset \Delta_r\) be a horizontal segment of length 1 that intersects \(q\).   For \(\sigma\) small enough, \(u_\sigma(I_3) \subset T^*(M - N(1, \frac{5}{2}\delta))\).   By definition, \(u_\sigma(I_1) \subset T^*N(1, 2\delta)\).   Call the region between \(I_1\) and \(I_3\) \(R\), and observe that \(R = [0, 1] \times [0, \frac{1}{2}]\).   Observe further that the image under \(u_\sigma\) of every vertical path in \(R\) has length at least \(\frac{1}{2}\delta\).   Therefore, if \(t\) is the vertical coordinate of \(\Delta_r\):
\[
\iint_R \left|\frac{\partial u_\sigma}{\partial t}\right|^2 dA \geq \frac{1}{2\delta}
\]
From this, we conclude that the \(L^2\)-norm of \(|Du|\) is bounded below.   This contradicts lemma~\ref{lemma5.7}, which states that the symplectic area of \(u(D_\sigma)\) is bounded above by \(C\sigma\) for some constant \(C\).\end{proof}

Let \(F_l\) be the region of points of distance \(l\) or less from \(D'_\sigma\).   Choose \(\rho'\) so that \(\log(\sigma^{-1}) \leq \rho' \leq 2\log(\sigma^{-1})\) and \(\partial F_{\rho'} - \partial D_\sigma, \partial F_{\frac{1}{2}\rho'} - \partial D_\sigma\) are vertical line segments disjoint from boundary minima.   Define \(D'_1(\sigma) = F_{\rho'}, D_0(\sigma) = D_\sigma - (F_{\frac{1}{2}\rho'})\).   Note that if \(p \in \partial D_0 \cap \partial D_\sigma\) then \(p\) is in a boundary component of type \textbf{0} or \textbf{out}, and if \(p \in \partial D'_1 \cap \partial D_\sigma\) then \(p\) is in a boundary component of type \textbf{0} or \textbf{in}.   Lemma~\ref{lemma5.12} implies that:
\[
\sup_{D_0} |Du_\sigma| \leq C\sigma
\]

\begin{lemma}(\cite{Ek}, Lemma 5.12)\label{lemma5.14} For small enough \(\sigma\), \(u(D'_1(\sigma)) \subset T^*U(1, \frac{9}{2}\delta)\) and \(u(D_0(\sigma)) \subset T^*(M - N(1, \frac{1}{2}\delta))\).\end{lemma}

\begin{proof}  Let \(q \in D_0(\sigma)\).   Then \(q\) is linked by a path \(\gamma\) of length less than 5\(r\), where \(r\) is the number of punctures of \(\Delta_r\), to a point \(p \in \partial D_\sigma \cap \partial D_0(\sigma)\).   \(p\) must lie in a component of type \textbf{0} or \textbf{out}, so \(u_\sigma(p) \in T^*(M - N(1, \delta))\).   Since \(|Du_\sigma| \leq C\sigma\) along the path, for \(\sigma\) small enough the second statement follows.

\(\partial D'_1(\sigma)\) consists of boundary segments of types \textbf{0} and \textbf{in} and vertical lines ending on boundary components of type \textbf{0}.   We know from the definition that the image of the boundary segments will lie in \(T^*N(1, 4\delta)\), while the bound on \(|Du_\sigma|\) allows us to ensure that \(u(\partial D_1(\sigma)) \subset T^*N(1, \frac{17}{4}\delta)\) for \(\sigma\) small enough.   Then, if \(u(D'_1(\sigma))\) does not lie inside \(T^*N(1, \frac{9}{2}\delta)\), we may bound its area from below for any \(\sigma\) by Lemma~\ref{lemma5.5}, our modified monotonicity lemma.\end{proof}

We can then repeat this process: we add additional punctures at the intersections of \(u(\partial\Delta_r)\) with \(T^*N(2, c\delta), c = 1,2,3,4\).   We label the boundary components of \(D_1'(\sigma)\) with \textbf{out', 0', in'}.   We bound \(u^T_\sigma\) in a neighborhorhood \(\Theta'_\epsilon\), and use equivalents of Lemmas 5.13 and 5.14 for \(u^T_\sigma\) to split \(D'_1(\sigma)\) into \(D_1(\sigma)\), which stays away from \(\Sigma_2\), and \(D_2'(\sigma) = D_2(\sigma)\), which does not.   We only do this once, unlike in \cite{Ek}.   We have thus divided \(\Delta_r = D_0(\sigma) \cup D_1(\sigma) \cup D_2(\sigma)\), so that \(D_0(\sigma)\) maps to a neighborhood away from the singularities, \(D_1(\sigma)\) maps to a neighborhood of the cusp edges, and \(D_2(\sigma)\) maps to a neighborhood of \(\Sigma_2\).

\subsection{Convergence of Disk Boundaries to Flow Lines}
\label{subsec:convergenceofdiskboundaries}

Let \(W_j(\sigma)\) be a neighborhood of the boundary minima of \(D_j(\sigma)\) such that:

\begin{itemize}
\item \(\partial W_j(\sigma)\) is a union of arcs in \(\partial D_j(\sigma)\) and of vertical line segments;

\item Each component of \(W_j(\sigma)\) contains at least one boundary minimum;

\item And the width of each component of \(W_j(\sigma)\) is at most \(\log \sigma^{-1}\).

\end{itemize}

Now consider a sequence:
\[
u_{\sigma}: (\Delta_m, \partial \Delta_m) \to (T^*M, \pi_\mathbb{C}(\Lambda_\sigma)), \sigma \to 0
\]

Let \(\sigma_0\) be the value of \(\sigma\) for which the number of components of \(D_0(\sigma) - W_0(\sigma)\) reaches its maximum.   (Since the total number of added punctures is bounded by Lemma~\ref{lemma5.11}, we must be able to select \(W_0(\sigma)\) in such a way that a maximum is eventually achieved.)   Restrict \(\sigma\) to \(\sigma < \sigma_0\), and consider a vertical line segment \(l \subset D_0(\sigma) - W_0(\sigma)\).   Let \(X\) be the component of \(D_0(\sigma) - W_0(\sigma)\) containing \(l\).   If \(X\) is an infinite or half-infinite rectangle, it can be parameterized by \((-\infty, 0] \times [0,1]\) or \([0, \infty) \times [0,1]\) or \(\mathbb{R} \times [0, 1]\).   In each case, \(l\) is specified by a choice of \(x\), so our choice of \(l\) is well-defined for any \(\sigma < \sigma_0\).   If \(X\) is a finite rectangle, it can be parameterized by \(X = [0, d_\sigma] \times [0, 1]\), where \(d_\sigma\) depends on \(\sigma\).   Pick some value of \(\sigma\), which we call \(\hat{\sigma}\).   Then, for \(\sigma = \hat{\sigma}\), \(l\) is specified by a choice of \(x \in [0, d_{\hat{\sigma}}]\).   Then we can define a function \(x_\sigma = (d_\sigma x / d_{\hat{\sigma}})\).   This lets us consider our choice of \(l\) to be well-defined for any value of \(\sigma < \sigma_0\).   We can similarly show that a choice of vertical line segment \(l \subset D_1(\sigma) - W_1(\sigma)\) is well-defined for any \(\sigma < \sigma_0\).

Recall from section~\ref{subsec:boundingderivative} that we can parameterize \(u_\sigma\) by \((q_\sigma, p_\sigma)\), where \(q_\sigma\) is the point in the base space and \(p_\sigma\) is the cofiber coordinate.

Let \(l \subset D_0(\sigma) - W_0(\sigma)\).   By the definition of \(D_0(\sigma)\), the image of \(l\) is outside of a neighborhood of \(\Sigma_1\).   We can therefore find a neighborhood of the image of \(l\) in which \(\pi_\mathbb{C}(\Lambda) \subset T^*M\) can be parameterized as the graph of some collection of functions \(M \to T^*M\).   We refer to these graphs as sheets, by analogy to the sheets of \(\pi_F(\Lambda)\).   Let \(b_1, b_0\) be functions of the sheets containing the image of \(\partial l\).

\begin{lemma}(\cite{Ek}, Lemma 5.13)\label{lemma5.17} For all sufficiently small \(\sigma > 0\), along any vertical line segment \(l \subset D_0(\sigma) - W_0(\sigma)\):
\[
\frac{1}{\sigma} \nabla_t p_\sigma(0, t) - \left(b_1(q_\sigma(0, 0)) - b_0(q_\sigma(0, 0))\right) = \mathcal{O}(\sigma)
\]\[
\frac{1}{\sigma} \nabla_\tau p_\sigma(0, t) = \mathcal{O}(\sigma)
\]
Where \(\nabla_*\) denotes the connection, the subscripts \(t, \tau\) are used to indicate \(\partial_t, \partial_\tau\), and \(t, \tau\) are the vertical, horizontal coordinates respectively of \(D_0(\sigma) - W_0(\sigma)\).\end{lemma}

\begin{proof} Let \(\Theta_{c_\sigma} = [-c_\sigma, c_\sigma] \times [0, 1] \subset D_0(\sigma) - W_0(\sigma)\) be an arbitrarily small neighborhood around \(l\), with \(c_\sigma \leq \sigma \log(\sigma^{-1})\).   By Lemma~\ref{lemma5.10}, \(|Du_\sigma| \leq C\sigma\) on \(\Theta_{c_\sigma}\).   Therefore, we can pick a radius \(R\) such that \(\pi_M(u_\sigma(\Theta_c)) \subset T^*D_{\sigma R}\), where \(D_{\sigma R}\) is the geodesic disk in \(M\) of radius \(\sigma R\) centered around \(\pi_M(u_\sigma(0, 0))\).   We think of \(T^*D_R\) as a subset of \(\mathbb{C}^n\), \((T^*D_R, J_1) = (\{q+ip : |q| \leq R\}, i)\).      Let \(g_0\) be the flat metric on \(D_R\), and let \(J_1\) be the corresponding standard complex structure on \(T^*D_R\), with the coordinates chosen such that \((J_1)_{u_\sigma(0, 0)} = (J_\sigma)_{u_\sigma(0, 0)}\) and \(u_\sigma(0, 0) = 0\).   Recall that \(q_\sigma\) is the base space component and \(p_\sigma\) is the cofiber component of \(u_\sigma\); there exists some \(K\) such that \(|p_\sigma(\Theta_c)| \leq K\sigma\) for all \(\sigma\).   Define \(U_\sigma\) to be \(\{q + ip : |q| \leq R\sigma, |p| \leq K\sigma\}\).

Define an almost complex structure \((\hat{J}_\sigma)_{(q, p)} = (J_\sigma)_{(\sigma q, \sigma p)}\), and define \(\hat{u}_\sigma = \sigma^{-1}u_\sigma:\Theta_c \to \mathbb{C}^n\).   \(\hat{u}_\sigma\) is \(\hat{J}_\sigma\)-holomorphic.

We claim that \(|\hat{J}_\sigma - J_1|_{C^2(U_\sigma)} = \mathcal{O}(\sigma)\).   We will prove the statement in dimension 2; the proof can be extended to higher dimensions by adding appropriate summation signs.   Define \(f_1, f_2\) by:
\[
(J_\sigma)_{z}\partial_q = f_1(z, \sigma)\partial_q + f_2(z, \sigma)\partial_p
\]
We know that \(f_1(0, \sigma) = 0, f_2(0, \sigma) = 1\) because \((J_1)_{u_\sigma(0, 0)} = (J_\sigma)_{u_\sigma(0, 0)} = (\hat{J})_{(0, 0)}\).   We will ordinarily omit the \(\sigma\) coordinate of \(f_1, f_2\).   Since \(U_\sigma\) is compact, \(|f'_1(\sigma z)|, |f'_2(\sigma z)|, |f''_1(\sigma z)|, |f''_2(\sigma z)|\) all obtain some maximum; let \(K_1\) be some number greater then all of them.

Using the Taylor expansion, we obtain:
\[
|(\hat{J}_\sigma)_z(\partial_q) - (J_1)_z(\partial_q)|^2 = (f_1(\sigma z))^2 + (1 - f_2(\sigma z))^2 = \mathcal{O}(\sigma^2)
\]
Therefore \(|(\hat{J}_\sigma)_z(\partial_q) - (J_1)_z(\partial_q)| = \mathcal{O}(\sigma)\).   From this, we obtain:
\[
D \left(|(\hat{J}_\sigma)_z(\partial_q) - (J_1)_z(\partial_q)|^2\right) = 2\sigma f_1(\sigma z)f'_1(\sigma z) - 2\sigma(1 - f_2(\sigma z))f'_2(\sigma z)
\]\[
\left(D |(\hat{J}_\sigma)_z(\partial_q) - (J_1)_z(\partial_q)|\right)\left(|(\hat{J}_\sigma)_z(\partial_q) - (J_1)_z(\partial_q)|\right) \leq 
\]\[
K_1 \sigma f_1(\sigma z) - K_1\sigma(1 - f_2(\sigma z)) = \mathcal{O}(\sigma^2)
\]
Therefore \(D |(\hat{J}_\sigma)_z(\partial_q) - (J_1)_z(\partial_q)| = \mathcal{O}(\sigma)\).   From this we further obtain:
\[
D^2 \left(|(\hat{J}_\sigma)_z(\partial_q) - (J_1)_z(\partial_q)|^2\right) = 2 \sigma^2 f_1(\sigma z)f''_1(\sigma z) +
\]\[
2\sigma^2 (f'_1(\sigma z))^2 - 2\sigma^2(1 - f_2(\sigma z))f''_2(\sigma z) + 2\sigma^2(f''_2(\sigma z))^2
\]\[
\left(D^2 |(\hat{J}_\sigma)_z(\partial_q) - (J_1)_z(\partial_q)|\right)\left(|(\hat{J}_\sigma)_z(\partial_q) - (J_1)_z(\partial_q)|\right) + 
\left(D |(\hat{J}_\sigma)_z(\partial_q) - (J_1)_z(\partial_q)|\right)^2 \leq
\]\[
\sigma^2 K_1 f_1(\sigma z) +\sigma^2 K_1^2 - \sigma^2(1 - f_2(\sigma z))K_1 + \sigma^2K_1^2
\]\[
\left(D^2 |(\hat{J}_\sigma)_z(\partial_q) - (J_1)_z(\partial_q)|\right)\left(|(\hat{J}_\sigma)_z(\partial_q) - (J_1)_z(\partial_q)|\right) + \mathcal{O}(\sigma^2) \leq \mathcal{O}(\sigma^2)
\]
Therefore \(D^2 |(\hat{J}_\sigma)_z(\partial_q) - (J_1)_z(\partial_q)| \leq \mathcal{O}(\sigma^2)\).   Combining these, we obtain:
\[
|(\hat{J}_\sigma)(\partial_q) - (J_1)(\partial_q)|_{C^2(U_\sigma)} = \mathcal{O}(\sigma)
\]
We can then repeat this calculation for \(\partial_p\).   Combining the two, we obtain \(|\hat{J}_\sigma - J_1|_{C^2} = \mathcal{O}(\sigma)\).   This calculation can be extended to higher dimension in a straight-forward manner by adding appropriate summation signs.

Recall that we defined the functions \(\sigma b_0, \sigma b_1: D_R \to T^*D_R\) to be the two sheets of \(\pi_\mathbb{C}(\Lambda_\sigma)\) over \(D_R\) corresponding to the restriction of \(u_\sigma|_{\partial\Theta_c}\) to \(D_R \subset \mathbb{C}^n\).   After scaling \(\mathbb{C}^n\) by \(\sigma^{-1}\), we replace \(x \to \sigma b_i(x)\) with \(x \to b_i(\sigma x)\).   Define \(L_i = \{x + iy: y = b_i(0)\} \subset \mathbb{C}^n\); note that these are Lagrangian subspaces.   Observe that:
\begin{equation}\label{eqn:pminusb1}
\hat{p}_\sigma(\tau + i) - b_1(q(0 + i)) = b_1(\sigma q(\tau + i)) - b_1(\sigma q(0 + i))
\end{equation}
In the next step, we will assume again that we are working in two dimensions without loss of generality.   We can rewrite the righthand side of equation~\ref{eqn:pminusb1} using its Taylor expansion around \(\sigma q(0 + i)\) to obtain:
\[
\hat{p}_\sigma(\tau + i) - b_1(q(0 + i)) = b'_1(\sigma q(0 + i))\sigma\left(q(\tau + i) - q(0 + i)\right) + \mathcal{O}(\sigma^2)
\]

Therefore, for any \(\tau\), \(\hat{p}_\sigma(\tau + i) - b_1(q(0 + i)) = \mathcal{O}(\sigma)\), as does its derivative and double derivative by \(\tau\).   We can repeat this calculation for \(\hat{p}_\sigma(\tau + 0i) - b_0(q(0 + 0i))\).   Therefore there exists a function \(f_\sigma:\Theta_c \to \mathbb{C}^n\) such that:
\[
f_\sigma(0, 0) = 0
\]\[
\sup_{\Theta_c} |D^k f_\sigma| = \mathcal{O}(\sigma)\mbox{ for } k = 1,2,3
\]\[
\hat{u}_\sigma(\tau + 0i) + f_\sigma(\tau + 0i) \in L_0
\]\[
\hat{u}_\sigma(\tau + i) + f_\sigma(\tau + i) \in L_1
\]

And \(\hat{u}_\sigma + f_\sigma\) is \(J_1\)-holomorphic to first order on \(\partial\Theta_c\), that is, \(\bar{\partial}_{J_1}(\hat{u}_\sigma + f_\sigma) = \mathcal{O}(\sigma)\).   Define \(u^1_\sigma = \hat{u}_\sigma + f_\sigma:\Theta_c \to \mathbb{C}^n\), and define \(u^0_\sigma:\Theta_c\to\mathbb{C}^n\) to be the linear solution to:
\[
\bar{\partial}_{J_1}u^0_\sigma= 0
\]\[
u^0_\sigma(\tau + 0i) \in L_0
\]\[
u^0_\sigma(\tau + i) \in L_1
\]\[
u^0_\sigma(\tau + it) = \left((b_1 - b_0)\tau, (1-t)b_0 + tb_1\right)
\]
(There is a typo in the definition of \(u^0_\sigma\) in \cite{Ek}, and we believe this is what he means.)

Define \(v_\sigma = u^1_\sigma - u^0_\sigma:\Theta_c \to \mathbb{C}^n\).   Then \(v_\sigma(\partial\Theta_c) \subset \mathbb{R}^n\), \(v_\sigma\) is \(J_1\)-holomorphic to first order on \(\partial\Theta_c\), \(v_\sigma(0, 0) = 0\), and:
\[
\sup_{\Theta_c} \left|
D^k\left(\bar{\partial}_{J_1} v_\sigma\right)
\right| = \mathcal{O}(\sigma), k = 0,1,2
\]
Let \(\mathcal{H}_{k,p,-\gamma}(\mathbb{R}\times[0,1], \mathbb{C}^n)\) denote the Hilbert space with the weight function \(w(\tau) = e^{-\gamma}\) for \(|\tau| \leq 1\) and \(w(\tau) = e^{-\gamma|\tau|}\) for \(|\tau| \geq 1\).   Let \(-\gamma = -3\).   Define \(\mathcal{H}_{3,2,-\gamma}(\mathbb{R}\times[0,1], \mathbb{C}^n; \mathbb{R}^n, 0_2)\) to be the space of functions \(F\) with boundary on \(\mathbb{R}^n\) and with three derivatives in the weighted Sobolev space \(L^2\), and such that the restriction to the boundary of \(\bar{\partial}_{J_1}F\) and its first derivatives vanishes.   We similarly define \(\mathcal{H}_{2, 2,-\gamma}(\mathbb{C}^n;0_1)\) to be the space of functions with 2 derivatives in the weighted Sobolev space \(L^2\) and which vanish to first order along the boundary.   As shown in \cite{EES}, Prop. 6.3, \(\bar{\partial}_{J_1}: \mathcal{H}_{3,-\gamma}(\mathbb{R}\times[0,1], \mathbb{C}^n; \mathbb{R}^n, 0_2) \to \mathcal{H}_{2,-\gamma}(\mathbb{C}^n;0_1)\) is a Fredholm operator of index \(n\) with kernel spanned by the constant functions.   We will generally write \(\bar{\partial}_{J_1} = \bar{\partial}\).

Let \(W \subset \mathcal{H}_{3,-\gamma}(\mathbb{R}\times[0,1],\mathbb{C}^n;\mathbb{R}^n,0_2)\) be the subspace of non-constant functions \(F\) with \(F(0) = 0\).   Then there exists a constant \(C\) such that:
\[
||w||_{\mathcal{H}_{3,2,-\gamma}} \leq C||\bar{\partial}w||_{\mathcal{H}_{2,2,-\gamma}}
\]

Now let \(B:\Theta_c\to\mathbb{C}\) be a cutoff function such that:
\begin{itemize}

\item \(B\) is real-valued and holomorphic to first order on \(\partial\Theta_{c_\sigma}\)

\item \(B(z) = 1\) on \([-\frac{1}{4}c_\sigma, \frac{1}{4}c_\sigma]\times[0,1]\)

\item \(B(z) = 0\) outside \([-\frac{1}{2}c_\sigma, \frac{1}{2}c_\sigma]\times[0,1]\)

\end{itemize}
Then:
\[
||\bar{\partial}(B v_\sigma)||^2_{\mathcal{H}_{2,2,-\gamma}} = \int\int_{\mathbb{R}\times[0,1]} \left(|\bar{\partial}v_\sigma|^2 + |D\bar{\partial}v_\sigma|^2 + |D^2\bar{\partial}v_\sigma|^2\right)
\]
We split this into cases.   First, observe that:
\[
 \int_{0}^1\int_{-\infty}^\infty |\bar{\partial}(Bv_\sigma)|^2w(\tau) d\tau dt =
\left|\left|\left.
\bar{\partial}(B v_\sigma)\right|_{[0,1]\times[-\frac{1}{4}c,\frac{1}{4}c]}
\right|\right|^2_{H_{2,2,-\gamma}} 
\]\[
+ \int^1_0\int_{|\tau| > \frac{1}{4}c} \left|(\bar{\partial}v_\sigma)\right|^2 |B(\tau)|^2e^{-\gamma|\tau|}d\tau dt +
\int^1_0\int_{|\tau| > \frac{1}{4}c} \left|(\bar{\partial}B)\right|^2 |v_\sigma|^2e^{-\gamma|\tau|}d\tau dt
\]

Given that \(B(\tau) = 1\) for \(|\tau| \leq \frac{1}{4}c_\sigma\), the first term is equal to:
\[
\left|\left|\left.
\bar{\partial}(B v_\sigma)\right|_{[0,1]\times[-\frac{1}{4}c,\frac{1}{4}c]}
\right|\right|^2_{H_{2,2,-\gamma}} = 
\left|\left|\left.
\bar{\partial}v_\sigma\right|_{[0,1]\times[-\frac{1}{4}c,\frac{1}{4}c]}
\right|\right|^2_{H_{2,2,-\gamma}}
\]
Since \(|\bar{\partial}v_\sigma| = \mathcal{O}(\sigma)\), this tells us that the first term of \(||\bar{\partial}Bv_\sigma||^2_{\mathcal{H}_{2,2,-\gamma}}\) equals \(\mathcal{O}(\sigma^2)\).

The second term, we calculate similarly; the bound on \(|\bar{\partial}v|\) tells us that it also equals \(\mathcal{O}(\sigma^2)\), and then no other term depends on \(t\), giving us:
\[
\int^1_0\int_{|\tau| > \frac{1}{4}c} \left|(\bar{\partial}v_\sigma)\right|^2 |B(\tau)|^2e^{-\gamma|\tau|}d\tau dt \leq
2K_0\sigma^2\int_1^\infty |B(\tau)|^2e^{-\gamma|\tau|}d\tau = \mathcal{O}(\sigma^2)
\]

Finally, the third term we calculate as follows: since \(|Du_\sigma|, |D^2u_\sigma|\) are bounded, \(|Dv_\sigma|, |D^2v_\sigma|\) are bounded, and thus \(|v_\sigma|, |Dv_\sigma| = \mathcal{O}(|c_\sigma|)\).   Therefore:
\[
\int^1_0\int_{|\tau| > \frac{1}{4}c} \left|(\bar{\partial}B)\right|^2 |v_\sigma|^2e^{-2\gamma|\tau|}d\tau dt = 
2K_1|c|^2\int_{\frac{1}{4}c_\sigma}^{\infty} e^{-2\gamma\tau}d\tau = 
-\frac{2K_1c^2e^{-\frac{1}{2}c_\sigma\gamma}}{2\gamma}
\]
This equals \(\mathcal{O}(c^2e^{-\frac{1}{2}\gamma c})\).   Since \(c_\sigma \leq \sigma \log (\sigma^{-1})\), this equals \(\mathcal{O}(\sigma^2 (\log (\sigma^{-1}))^2 \sigma^{3/2} e^{-\frac{3}{2}\sigma}) = \mathcal{O}(\sigma^{7/2}\log(\sigma^{-1})) \leq \mathcal{O}(\sigma^2)\).

We can then repeat this process for \(B (Dv_\sigma), B (D^2v_\sigma)\) to show that:
\[
||B v_\sigma||_{H_{2,3,-\gamma}} = \mathcal{O}(\sigma)
\]
Since this controls the supremum norm over \([-\frac{1}{4}c_\sigma, \frac{1}{4}c_\sigma] \times [0,1]\), we obtain:
\[
\sup_{[0,1]\times[-\frac{1}{4}c_\sigma,\frac{1}{4}c_\sigma]} |D^kv_\sigma| = \mathcal{O}(\sigma), k = 0,1
\]
If we write \(\hat{u}_\sigma(\tau, t) = (\hat{q}_\sigma(\tau, t), \hat{p}_\sigma(\tau,t))\), we obtain:
\[
\nabla_t \hat{p}_\sigma(0, t) = \frac{\partial \hat{p}_\sigma}{\partial t}(0, t) + \Gamma(\hat{q}_\sigma(0, t))\left(\frac{\partial \hat{q}_\sigma}{\partial t}, \hat{p}_\sigma\right)
\]\[
\nabla_\tau \hat{p}_\sigma(0, t) = \frac{\partial \hat{p}_\sigma}{\partial \tau}(0, t) + \Gamma(\hat{q}_\sigma(0, t))\left(\frac{\partial \hat{q}_\sigma}{\partial \tau}, \hat{p}_\sigma\right)
\]
Where \(\Gamma\) denotes the linear operator in \(\nabla_u v = D_uv + \Gamma(u, v)\).

Recall that \(\hat{u}_\sigma = \sigma^{-1}u_\sigma\).   Further, \(\hat{u}_\sigma - u^0_\sigma = v_\sigma - f_\sigma\), where \(\sup |D^kf_\sigma| = \mathcal{O}(\sigma)\) for \(k=1,2,3\), and \(\frac{\partial p^0_\sigma}{\partial \tau} =\frac{\partial q^0_\sigma}{\partial t} = 0, \frac{\partial p^0_\sigma}{\partial t} = b_1(0) - b_0(0)\).   Therefore:
\[
\frac{\partial\hat{p}_\sigma}{\partial t} = \frac{\partial p^0_\sigma}{\partial t} + \mathcal{O}(\sigma) = (b_1(0) - b_0(0)) + \mathcal{O}(\sigma)
\]\[
\frac{\partial\hat{p}_\sigma}{\partial \tau} = \mathcal{O}(\sigma)
\]\end{proof}

{\color{black}\begin{lemma}\label{lemma5.16}For any component \(C(\sigma) \subset W_0(\sigma)\), \(\pi_M \circ  u_\sigma(C(\sigma) \cap \partial\Delta_r)\) converges to a point as \(\sigma \to 0\).\end{lemma}

\begin{proof} By Lemma~\ref{lemma5.12}, \(|Du_\sigma| = \mathcal{O}(\sigma)\) in \(W_0(\sigma)\).  The width of \(C(\sigma)\) is at most \(\log(\sigma^{-1})\), so the length of \(\partial\Delta_r \cap C(\sigma)\) is \(\mathcal{O}(\log(\sigma^{-1})\).   Therefore, the length of \(u_\sigma(C(\sigma) \cap \partial\Delta_r)\) is \(\mathcal{O}(\sigma \log (\sigma^{-1}))\), so \(C(\sigma)\) converges to a collection of discrete points in \(p_1, ..., p_m \in T^*M\).

Since \(u_\sigma\) is \(J_\sigma\)-holomorphic, lemma~\ref{lemma5.17} implies that \(\nabla_tq_\sigma = \mathcal{O}(\sigma^2)\) in \(D_0(\sigma) - W_0(\sigma)\), which means that for any pair of points \(p^+, p^- \in \partial\Delta_r\) linked by a vertical line in the interior of \(D_0(\sigma) - W_0(\sigma)\), we know that \(\pi_M \circ u_\sigma(p^+), \pi_M \circ u_\sigma(p^-)\) converge to the same point in \(M\) as \(\sigma\to 0\).

Since \(u_\sigma\) is continuous, these two facts imply that \(\pi_M \circ u_\sigma (C(\sigma))\) converges to a point.\end{proof}}

\begin{lemma}\label{lemma5.19}For all sufficiently small \(\sigma > 0\), along any vertical line segment \(l \subset D_1(\sigma) - W_1(\sigma)\):
\[
\frac{1}{\sigma} \nabla_t p^T_\sigma(0, t) - \left(b_1(q^T_\sigma(0, 0)) - b_0(q^T_\sigma(0, 0))\right) = \mathcal{O}(\sigma)
\]\[
\frac{1}{\sigma} \nabla_\tau p^T_\sigma(0, t) = \mathcal{O}(\sigma)
\]\end{lemma}

\begin{lemma}\label{lemma5.18}{\color{black}For any component \(C(\sigma) \subset W_1(\sigma)\), \(\pi_M \circ  u_\sigma(C(\sigma) \cap \partial\Delta_r)\) converges to a point as \(\sigma \to 0\). }\end{lemma}

\textbf{Proof of Lemmas~\ref{lemma5.19} and~\ref{lemma5.18}:} The proofs are precisely analogous to the proofs of Lemmas~\ref{lemma5.17} and~\ref{lemma5.16}, but with \(u_\sigma\) replaced with \(u^T_\sigma\); the base space replaced with \(\pi_M(\Sigma_1)\); \(b_i\) replaced by their restriction to \(T^*\pi(\Sigma_1)\); \(\Lambda_\sigma\) replaced with its projection to \(T^*\pi(\Sigma_1)\); and \(J_\sigma\) replaced by its restriction to \(T^*\pi(\Sigma_1)\).

\bigskip

%
%
%
%
%
%
%
%
%
%
%
%

\begin{lemma}\label{WFGLemma} If \(u:(\Delta_r, \partial\Delta_r) \to (T^*M, \pi_\mathbb{C}(\Lambda))\) is a pseudoholomorphic disk with a single positive puncture, then there exists a sequence \(\sigma \to 0\) such that \(\pi_M \circ u_{\sigma}(\partial \Delta_r)\) restricted to \(M - \pi_M(\Sigma_2)\) converges to a collection of partial flow trees as \(\sigma \to 0\).\end{lemma}

\begin{proof}By Lemma~\ref{lemma5.17}, we know that, for any vertical line segment \(l \subset D_0(\sigma) - W_0(\sigma)\):
\[
\frac{1}{\sigma} \nabla_t p_\sigma(0, t) - \big(b_1(q_\sigma(0, 0) - b_0(q_\sigma(0, 0)))\big) = \mathcal{O}(\sigma)
\]\[
\frac{1}{\sigma} \nabla_\tau p_\sigma(0, t) = \mathcal{O}(\sigma)
\]
Where \(b_i\) are the gradients of sheet height functions.   Since \(u_\sigma\) is \(J_\sigma\)-holomorphic, this implies that:
\[
\nabla_\tau q_\sigma(0, t) - \sigma\big(b_1(q_\sigma(0, 0) - b_0(q_\sigma(0, 0)))\big) = \mathcal{O}(\sigma^2)
\]\[
\nabla_t q_\sigma = \mathcal{O}(\sigma^2)
\]
Therefore, every edge in \(\partial \Delta_r \cap (D_0(\sigma) - W_0(\sigma))\) converges to the flow line of a pair of height functions.   Furthermore, any pair of edges linked by a vertical line in \(D_0(\sigma) - W_0(\sigma)\) converge to the same flow line.   Furthermore, by Lemma~\ref{lemma5.16}, we know that for every component \(C(\sigma) \subset W_0(\sigma)\), \(\pi_M \circ u_\sigma(C(\sigma))\) converges to a point as \(\sigma \to 0\).   This tells us that \(\pi_M \circ u_\sigma(\partial\Delta_r)\) converges to a partial flow tree or trees over \(D_0(\sigma)\).   Lemma~\ref{lemma5.18} and \ref{lemma5.19} tell us that the same is true over \(D_1(\sigma)\).

Thus, \(\pi_M\circ u_\sigma(\partial\Delta_r)\) converges to a collection of partial flow trees outside of an \(\epsilon_3\)-neighborhood of \(\pi_M(\Sigma_2)\).   Since \(\epsilon_3\) is arbitrarily small, we can take the limit as \(\epsilon_3 \to 0\), and conclude that \(((\pi_M\circ u_\sigma)(\partial\Delta_r)) - \pi_M(\Sigma_2)\) converges to a collection of partial flow trees \(F_i: \Gamma_i \to M\) with special punctures on \(\Sigma_2\).   We can assume that these special punctures are one-valent because, if we have any multi-valent special punctures, we can cut the vertex corresponding to the special puncture into multiple vertices, one for each flow line adjacent to the vertex.   The existence of the arcs in \(\pi_M(\sigma_2)\) follows from observing that the image of \(\pi_M \circ u_\sigma(\partial \Delta_r)\) remains connected as \(\sigma \to 0\); the arcs are the components of \(\pi_M \circ u_\sigma(\partial \Delta_r) \cap \pi_M(\Sigma_2)\).   
\end{proof}

We define a \textbf{\(\Sigma_2\)-complex} to be a set of arcs \(\gamma_1, ..., \gamma_k \in \Sigma_2\) that are in the limit of the image of \(((\pi_M \circ u_\sigma)(\partial \Delta_r)) \cap \pi_M(\Sigma_2)\) as \(\sigma \to 0\), such that \(\gamma_1\) shares endpoints with \(\gamma_2\) and \(\gamma_k\); \(\gamma_2\) shares endpoints with \(\gamma_1\) and \(\gamma_3\); etc.   We refer to the endpoints of the arcs as special punctures, since they will form the special punctures of partial flow trees \(F_i:\Gamma_i \to M\).   We can similarly consider the special punctures to be positive or negative based on the orientation of the arcs entering the puncture.   A positive special puncture in a \(\Sigma_2\)-complex will connect to a negative special puncture in a partial flow tree and vice-versa.

Given a disk \(u_\sigma:\Delta_r \to T^*M\), we can treat the limit of \(\pi_M \circ u_\sigma(\partial\Delta_r)\) as \(\sigma \to 0\) as a graph, consisting of the partial flow trees \(\Gamma_i\), with the special punctures connected by the arcs \(\gamma_j\) in the \(\Sigma_2\)-complexes.

\begin{lemma}\label{NoCircuitLemma} Let \(G\) be the graph created by combining the partial flow trees of the limit of \((\pi_M \circ u_\sigma)(\partial \Delta_r) - \pi_M(\Sigma_2)\) with its \(\Sigma_2\)-complexes.   Then the only circuits in \(G\) are contained entirely in a \(\Sigma_2\)-complex.\end{lemma}

\begin{proof} Suppose that \(G\) has a circuit that is not contained in a \(\Sigma_2\)-complex.   Let the edges of this circuit that lie in partial flow trees be \(e_1, e_2, ..., e_m \in M\), with \(e_1\) lying in a partial flow tree.   We will use this circuit to construct a curve \(C\) in the interior of \(\Delta_r\) that separates \(\overline{\partial \Delta_r}\) into two components, where \(\overline{I}\) denotes the closure of \(I\).

By a slight abuse of notation, we can consider \(e_k\) to have two preimages in \(\partial\Delta_r\), \(e_k^+\) and \(e_k^-\), where the \(+\) superscript denotes the edge in the slit model with higher vertical coordinate.   (That is, \(e_k^\pm\) is an interval in \(\partial\Delta_r\) so that \(\pi_M \circ u_\sigma(e_k^\pm)\) limits to \(e_k\) as \(\sigma \to 0\).)   We can assume without loss of generality that, for every \(e_k\), there are coordinates \((t, \tau)\) on \(\Delta_r\) so that:
\[
e_k^+ = \{(t, 1) | 0 \leq t \leq 1\}
\]\[
e_k^- = \{(t, 0) | 0 \leq t \leq 1\}
\]
And define:
\[
c_k = \{(t, 0.5) | 0 \leq t \leq 1\}
\]
Observe that \(\pi_M \circ u_\sigma(c_k)\) limits to the flow line \(e_k\) as \(\sigma \to 0\) for all \(k\).   Consider \(c_k, c_{k+1}\).   If \(e_k, e_{k+1}\) share a vertex \(p_k\) in a partial flow tree, then we can find a path \(g_k\) in \(W_0(\sigma)\) or \(W_1(\sigma)\) linking the end points of \(c_k, c_{k+1}\) and so that \(\pi_M \circ u_\sigma(g_k)\) limits to the vertex \(p_k\) as \(\sigma \to 0\).   If \(e_k, e_{k+1}\) are separated by a \(\Sigma_2\)-complex, then we can find a path \(g_k\) in \(W_2(\sigma)\) linking the end points of \(c_k, c_{k+1}\) and so that the \(g_k\) lies arbitrarily close to the edges in the circuit that lie in the \(\Sigma_2\)-complex.   If we use the convention that \(e_{m+1} = e_1\), so \(g_{m+1}\) links the endpoints of \(c_1, c_m\), then the union:
\[
C = c_1 \cup g_1 \cup c_2 \cup g_2 \cup ... \cup c_m \cup g_{m+1}
\]
Is a circle that lies in the interior of \(\Delta_r\).   \(C\) does not intersect itself: by definition, none of the edges or vertices in the circuit are repeated, so there can be no intersection between \(c_i\) or \(g_i\) that lie in \(W_0(\sigma), W_1(\sigma)\) and some other part of the circle, and the requirement that \(g_i\) that lie in \(W_2(\sigma)\) lie arbitrarily close to the edges of the \(\Sigma_2\)-complex ensures that they will not intersect either.   Therefore, we have a circle \(C\) that separates \(\Delta_r\) into two components, \(D^1, D^2\).   If we return to the disk model of \(\Delta_r\), we see that \(D^1 \cap \overline{\partial\Delta_r}, D^2 \cap \overline{\partial\Delta_r}\) is a separation of \(\overline{\partial\Delta_r}\) into two components, which means that \(\overline{\partial\Delta_r}\) is disconnected, and therefore that \(\Delta_r\) is not a disk.   This is a contradiction.   Therefore, \(G\) can contain no circuits that do not lie entirely in \(\Sigma_2\)-complexes.
\end{proof}

We are now ready to prove that every partial flow tree in the limit of \(\pi_M \circ u_\sigma(\partial\Delta_r)\) restricted to \(M-\pi_M(\Sigma_2)\) as \(\sigma\to 0\) has a single positive puncture.

\begin{lemma}\label{GFGlemma} If \(u:(\Delta_r, \partial\Delta_r) \to (T^*M, \pi_\mathbb{C}(\Lambda))\) is a pseudoholomorphic disk with a single positive puncture, then there exists a sequence \(\sigma \to 0\) such that \(\pi_M \circ u_\sigma(\partial \Delta_r)\) restricted to \(M - \pi_M(\Sigma_2)\) converges to a collection of partial flow trees as \(\sigma \to 0\) such that each partial flow tree has a single positive puncture.\end{lemma}

\begin{proof}Let \(G\) be the graph created by combining the partial flow trees with the \(\Sigma_2\)-complexes.   Suppose that we have some partial flow tree \(F_i:\Gamma_i \to M\) with two positive punctures, \(p_1\) and \(p_2\).   We will assume for notational simplicity that these are special positive punctures, but the argument is strictly analogous even if one is a positive puncture at a Reeb chord.   Since a tree is connected by definition, we can find a sequence of edges in \(F_i(\Gamma_i)\) linking \(p_1\) to \(p_2\).   Therefore, in order for \(G\) to not contain a circuit outside of a \(\Sigma_2\)-complex - which is prohibited by Lemma~\ref{NoCircuitLemma} - deleting \(F_i(\Gamma_i)\) from \(G\) must produce at least two separate components, one which has \(p_1\) in its boundary and one which has \(p_2\).   We label these \(G_1\) and \(G_2\).   Since there is only a single positive puncture at a Reeb chord anywhere in the disk, only one at most of these two components can contain that positive puncture at a Reeb chord; assume without loss of generality that \(G_2\) does.

Let \(N_\epsilon(p_1) \subset M\) denote an \(\epsilon\)-neighborhood of \(p_1\), and choose \(\epsilon\) small enough that it does not intersect \(\pi_M \circ u_\sigma(\partial\Delta_r)\) except for the arcs adjacent to \(p_1\).   Assume \(\sigma\) is small enough that \(\pi_M \circ u_\sigma(\partial\Delta_r)\) intersects \(\partial N_\epsilon(p_1)\) in four places, two of which limit to \(\partial G_1 \cap N_\epsilon(p_1)\) as \(\sigma \to 0\), and two limiting to \(F_i(\Gamma_i) \cap N_\epsilon(p_1)\).   Let \(p_+(\sigma), p_-(\sigma)\) be the preimages in \(\partial \Delta_r\) of the points where the arcs of the disk limiting to \(G_1\) intersect \(\partial N_\epsilon(p_1)\), with \(p_+(\sigma)\) having the higher \(z\) coordinate in the lift of \(\partial\Delta_r\) to \(J^1(M)\).   For \(\sigma\) small enough, we can find an arc \(\gamma(\sigma) \subset u_\sigma^{-1}(N_\epsilon(p_1))\) linking \(p_+(\sigma)\) to \(p_-(\sigma)\), oriented from \(p_+(\sigma)\) to \(p_-(\sigma)\).   Let \(I(\sigma)\) be the component of \(\partial \Delta_r\) which limits to \(G_1\) as \(\sigma \to 0\).   We can then separate \(\Delta_r\) along \(\gamma(\sigma)\) to produce a lens \(E(\sigma)\) with boundary \(\gamma(\sigma) \cup I(\sigma)\), so that \(\pi_M \circ u_\sigma(\partial E(\sigma))\) limits to \(G_1\) as \(\sigma \to 0\), and \(\Delta_r - E(\sigma)\) limits to \(G - G_1\).   Let \(\tilde{u}_\sigma:E(\sigma) \to T^*M\) be the restriction of \(u_\sigma\) to \(E(\sigma)\).

By Stokes' Theorem, the symplectic area of \(\tilde{u}_\sigma(E(\sigma))\) is given by:
\[
A(\sigma) = -\sum_{i=1}^N \sigma \alpha_i + \int_{\tilde{u}_\sigma(\gamma(\sigma))} \beta
\]
Where \(\alpha_i\) are the actions of the Reeb chords so that there are punctures in \(\partial E(\sigma)\) limiting to them, and \(\beta\) is the tautological one-form on \(T^*M\).   There exists a lift \(\tilde{\gamma}(\sigma)\) of \(\tilde{u}_\sigma(\gamma(\sigma))\) to \(J^1(M)\), so that:
\[
\pi_\mathbb{C}(\tilde{\gamma}(\sigma)) = \tilde{u}_\sigma(\gamma(\sigma))
\]\[
\alpha \big|_{\tilde{\gamma}(\sigma)} = 0
\]
Where the \(z\) coordinate of \(\tilde{\gamma}(\sigma)\) is given by integrating \(\beta\) along \(\tilde{u}_\sigma(\gamma(\sigma))\).   Therefore, the restriction of \(\beta\) to \(\tilde{u}_\sigma(\gamma(\sigma))\) equals \(dz\), where \(z:\tilde{u}_\sigma(\gamma(\sigma)) \to \mathbb{R}\) is the \(z\) coordinate of the corresponding point in \(\tilde{\gamma}(\sigma)\).   Therefore, by Stokes' Theorem:
\[
\int_{\tilde{u}_\sigma(\gamma(\sigma))} \beta = z(p_-(\sigma)) - z(p_+(\sigma)) < 0
\]
Therefore, \(A(\sigma) < 0\), which is a contradiction.   Therefore, there exists no partial flow tree \(F_j:\Gamma_j \to M\) with more than one positive puncture.\end{proof}

}
\section{Proof of Morse Lemmas}
\label{sec:MorseLemmas}

In this section we prove~\ref{lemma3.1},~\ref{lemma3.2},~\ref{lemma3.3}, and~\ref{lemma3.4}.   First, recall that we define:
\[
\mathcal{F}(f_1, \delta) = \left\{f_2: M \to \mathbb{R} | f_2\mbox{ is Morse}, |f_1 - f_2|_{C^1} < \delta\right\}
\]

In addition, we need the well-known tubular flow theorem:

\begin{theorem}[Tubular Flow Theorem]\label{tubularflowtheorem} Let \(V\) be a vector field on an \(n\)-dimensional manifold \(M\).   Then, at any point \(p \in M\) such that \(V_p \neq 0\), there exists a neighborhood \(U\) of \(p\) and a diffeomorphism \(\Psi:(-1,1)^n \to U\) such that the trajectories of \(V\) in \(U\) are mapped to the trajectories of \(\partial_{x_1}\) in \((-1,1)^n\).
\end{theorem}

\begin{proof}  See \cite{PdM}, Theorem 2.1.1.\end{proof}

Then:

\textbf{Lemma~\ref{lemma3.1}:} \emph{Let \(f_1: M \to \mathbb{R}\) be a Morse function, where \(M\) is a compact manifold.   For any \(\epsilon > 0\), there exists \(\delta > 0\) such that, for any \(f_2 \in \mathcal{F}(f_1, \delta)\), there is a bijection between the critical points of \(f_1\) and \(f_2\), the ascending manifold of every critical point of \(f_2\) lies within \(\epsilon\) of the ascending manifold for the corresponding critical point of \(f_1\), and the descending manifold of every critical point of \(f_2\) lies within \(\epsilon\) of the descending manifold of the corresponding critical point of \(f_1\).}

\begin{proof}  It is well known that, if \(f_1\) is Morse, then there exists \(\delta'\) such that all \(f_2 \in \mathcal{F}(f_1, \delta')\) have the same number of critical points as \(f_1\), and each critical point of \(f_2\) has the same Morse index as the corresponding critical point of \(f_1\).   Restrict \(\delta\) to \(\delta < \delta'\).

Now, since \(f_1\) is smooth, \(|\nabla f_1|:M \to \mathbb{R}\) is continuous, and \(|\nabla f_1|^{-1}(0)\) is equal to the set of critical points of \(f_1\).   Further, we know that:
\[
\sup_{p \in M} \left| |\nabla f_1(p)| - |\nabla f_2(p)| \right| \leq \left| \nabla f_1 - \nabla f_2\right|_{C^0} \leq |f_1 - f_2|_{C^1}
\]
Therefore, if \(f_2 \in \mathcal{F}(f_1, \delta)\), then \(|\nabla f_2|^{-1}(0) \subset |\nabla f_1|^{-1}(0, \delta)\).   Therefore, if \(\delta > 0\) is small enough, the critical points of \(f_2\) will lie arbitrarily close to the corresponding critical points of \(f_1\).   Let \(p_1, ..., p_m\) denote the critical points of \(f_1\), and let \(p'_i\) denote the critical point of \(f_2\) corresponding to \(p_i\).   Pick open ball-shaped Morse neighborhoods \(\psi_i:N_i \to M\) for every critical point \(p_i\), and assume that \(\delta\) is small enough that \(p'_i \in \psi_i(N_i)\) for every \(f_2 \in \mathcal{F}(f_1, \delta)\) and every critical point \(p'_i\).   Define:
\[
Q = M - \bigcup_i \psi_i(N_i)
\]
Observe that \(Q\) is compact.

By the Tubular Flow Theorem, for every point \(p\) that is not a critical point, we can find a radius \(\rho(p)\) such that there is an open neighborhood of \(p\) of radius \(\rho(p) < \frac{1}{2}\epsilon\) that is diffeomorphic to \((-1, 1)^n\), and such that the diffeomorphism carries flow lines of \(\nabla f_1\) to flow lines of \(\partial_{x_1}\).   Let \(U(p)\) denote the tubular flow neighborhood of \(p\), and observe that \(U(p)\) is an infinite open cover of the compact submanifold \(Q\).   We can therefore find a finite subcover \(U_1, ..., U_m\).   Let \(\phi_1:[-1,1]^n \to U_1, ..., \phi_m:[-1,1]^n\to U_m\) denote the diffeomorphisms.   Equip each copy of \([-1,1]^n\) with the pullback of the metric of \(M\), rather then the standard Euclidean metric.

Pick an arbitrary flow line \(\gamma:[0, l]\to M\) parameterized by arc length, so that \(\gamma(0) = p_i, \gamma(l) = p_j\).   Assume without loss of generality that \(\psi_i(N_i), U_1, ..., U_k, \psi_j(N_j)\) form an open cover of the image of \(\gamma\) (see Figure~\ref{fig:coverOfFlowline}).

\begin{figure}\begin{center}
\includegraphics{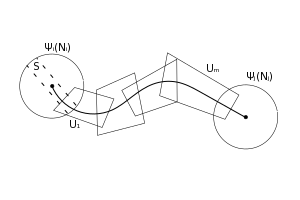}
\caption{Open Cover of \(\gamma\)\label{fig:coverOfFlowline}}\end{center}
\end{figure}

Let \(q_0\) be the point where the pullback of \(\gamma\) to \(N_i\) intersects \(\partial N_i\).   \(\psi_i^{-1}(U_1)\) will be open in \(N_i\), so there is some disk of radius \(\epsilon_0\) around \(q_0\) in \(\partial N_i\) which is contained within \(\partial N_i \cap \psi_i^{-1}(U_1)\).   Define the set:
\begin{equation}\label{eqn:defineS}
S = \left\{x \in N_i | x_{k_i + 1}^2 + ... + x_n^2 < (\epsilon_0)^2\right\}
\end{equation}
Where \(k_i\) is the Morse index of \(p_i\).   The boundary of \(S\) can be considered as the union of two (possibly disconnected, possibly empty) overlapping parts:
\begin{equation}\label{eqn:defineVW}
V = \left\{x \in N_i | x_{k_i + 1}^2 + ... + x_n^2 = (\epsilon_0)^2\right\}
\end{equation}\[
W = \left\{x \in \partial N_i | x_{k_i + 1}^2 + ... + x_n^2 \leq (\epsilon_0)^2\right\}
\]
Schematically, this takes the form shown in figure~\ref{fig:schematicOfSVW}.

\begin{figure}\begin{center}
\includegraphics{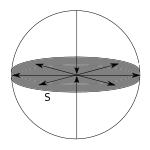}
\includegraphics{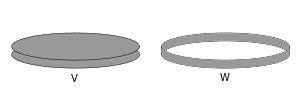}
\caption{\(S, V\), and \(W\)\label{fig:schematicOfSVW}}\end{center}
\end{figure}

Note that, by construction, \(q_0 \in W\).   We now define a pair of vector fields, \(R_V\) and \(R_W\), on \(V, W\) respectively, as seen in Figure~\ref{fig:RVandRW}:
\[
R_V = -x_{k_i+1}\partial_{x_{k_i+1}} - .... - x_n\partial_{x_n}
\]\[
R_W = x_1\partial_{x_1} + ... x_{k_i}\partial_{x_{k_i}}
\]

\begin{figure}\begin{center}
\includegraphics{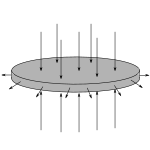}
\caption{\(R_V\) and \(R_W\)\label{fig:RVandRW}}\end{center}
\end{figure}

Observe that, using the ambient metric \(g\) of \(N_i \subset \mathbb{R}^n\):
\[
\left.\nabla \psi_i^*f_1\right|_{V} \cdot R_V = -x_{k_i+1}^2 - ... - x_n^2 < 0
\]\[
\left.\nabla \psi_i^*f_1\right|_{W} \cdot R_W = -x_1^2 - ... - x_{k_i}^2 < 0
\]

What this means is that the flow lines of \(\nabla \psi_i^*f_1\) may enter \(S\) through \(V\) and leave through \(W\), but not the other way around.

Now observe that, if we pull \(f_2\) back to \(N_i\) with \(\psi_i\):
\[
\left.\nabla \psi_i^*f_2\right|_{V} \cdot R_V = 
\left.\nabla \psi_i^*f_1\right|_{V} \cdot R_V + \left.(\nabla \psi_i^*f_2 - \psi_i^*\nabla f_1)\right|_{V} \cdot R_V
\]\[
\left.\nabla \psi_i^*f_2\right|_{W} \cdot R_W = 
\left.\nabla \psi_i^*f_1\right|_{W} \cdot R_W + \left.(\nabla \psi_i^*f_2 - \nabla \psi_i^*f_1)\right|_{W} \cdot R_W
\]
And we know that:
\[
\left|\nabla \psi_i^*f_2 - \nabla \psi_i^*f_1\right| \leq |\psi_i|_{C^1} |f_2 - f_1|_{C^1}
\]
Therefore, if \(\delta\) is small enough:
\begin{equation}\label{eqn:nablaf2RVandRW}
\left.\nabla \psi_i^*f_2\right|_{V} \cdot R_V < 0\mbox{ and }
\left.\nabla \psi_i^*f_2\right|_{W} \cdot R_W < 0
\end{equation}
For all \(f_2 \in \mathcal{F}(f_1, \delta)\).   And, for \(\delta\) small enough, \(p'_i \in S\).   What this means is that the flow lines of \(\nabla \psi_i^*f_2\) for any \(f_2 \in \mathcal{F}(f_1, \delta)\) will enter \(S\) through \(V\) and leave through \(W\), but not vice-versa, if \(\delta\) is small enough.   This means, first, that \(\mathcal{D}_{f_2}(p'_i) \cap \partial N_i \subset W\).   In addition, it means that \(\mathcal{A}_{f_2}(p'_i) \cap \partial N_i \not\subset W\).

Define \(Z = [\mathcal{D}_{f_2}(p'_i) \cap W] \in H_{k_i-1}(W)\), where \(H_*(W)\) is the singular homology of \(W\) with \(\mathbb{Z}\) coefficients.   Since \(W\) is diffeomorphic to \(S^{k_i-1} \times [-1,1]^{n-k_i+1}\), we know that \(H_{k_i-1}(W) \tilde{=} \mathbb{Z}\).   If \(Z\) is trivial as an element of the homology of \(H_{k_i-1}\), then it bounds some disk \(Y \subset W\), so that \(\partial Y = Z\).   Since the descending manifold of \(p'_i\) is transverse to the ascending manifold at \(p'_i\), and intersects it at no other point, that implies that \(\mathcal{A}_{f_2}(p'_i) \cap \partial N_i \subset Y \subset W\), which is a contradiction.   Therefore, \(Z\) cannot be trivial.   This, in turn, implies that for any choice of \((x_1, ..., x_k, 0, ..., 0) \in W\), there exists some point \((x_1, ..., x_k, x_{k+1}, ..., x_n) \in \mathcal{D}_{f_2}(p'_i) \cap W\).

Therefore, if \(\delta\) is small enough to ensure equation~\ref{eqn:nablaf2RVandRW} holds, then the intersection \(\psi_i^{-1}(U_1) \cap \mathcal{D}_{f_2}(p'_i)\) will be non-empty, and contain at least some open \(k_i\)-disk.   Call the image of this disk in \(M\) \(D_1\).


\begin{figure}
\begin{center}\includegraphics{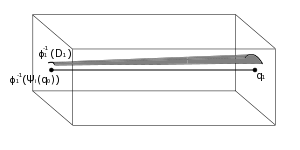}
\caption{Preimage of \(U_1\)\label{fig:U1}}\end{center}
\end{figure}

Now consider \(\phi_1^{-1}(D_1)\), as shown in figure~\ref{fig:U1}.   From the Tubular Flwo Theorem, we know that, in \([-1,1]^n\), the pullback of \(\nabla f_1\) has the form:
\[
\nabla (\phi_1^* f_1) = \lambda_1(x) \partial_{x_1}
\]
Where \(\lambda_1(x) > 0\).   Since \([-1,1]^n\) is compact, we can find \(l_1 > 0\) such that \(\lambda_1(x) > l_1\) for all \(x\).   Let \(q_1\) be the point where \(\phi_1^{-1} \circ \gamma\) leaves \([-1,1]^n\).   We also know that:
\begin{equation}\label{eqn:nablaph1f2}
\nabla (\phi_1^* f_2) = \nabla (\phi_1^* f_1) + \nabla (\phi_1^* (f_2 - f_1))
\end{equation}\begin{equation}\label{eqn:boundnablaphi2f2minusf1}
\left|\nabla (\phi_1^* (f_2 - f_1))\right| \leq \delta||\phi_1||_{C^1}
\end{equation}
Now, pick an arbitrary point in \(\phi_1^{-1}(D_1)\), and let \(\gamma^1:[0,T_1] \to [-1,1]^n\) be the flow line beginning at that point for \(\nabla (\phi_1^*f_2)\), parameterized by \((\gamma^1)' = \nabla (\phi_1^*f_2)\), such that \(\gamma^1(T_1) \in \partial [-1,1]^n\).   Then:
\begin{equation}\label{eqn:dgamma1dt}
\frac{d\gamma^1}{dt} = \lambda(\gamma^1(t))\partial_{x_1} + \nabla \phi_1^*(f_2 - f_1)
\end{equation}
Let \(\gamma^1_i\) denote the \(x_i\) coordinate of \(\gamma^1\).   Then equations~\ref{eqn:nablaph1f2},~\ref{eqn:boundnablaphi2f2minusf1}, and~\ref{eqn:dgamma1dt} together imply that:
\[
\frac{d\gamma^1_1}{dt} > l_1 - ||\phi_1||_{C^1}\delta
\]\[
\left|\frac{d\gamma^1_i}{dt}\right| < ||\phi_1||_{C^1}\delta \mbox{ for }i \neq 1
\]
Therefore, since the width of \([-1,1]^n\) is 2, we can bound \(T_1\) by:
\[
T_1 < \frac{2}{l_1 - ||\phi_1||_{C^1}\delta}
\]
And, for \(i \neq 1\):
\[
\left|\gamma^1_i(T_1) - \gamma^1_i(0)\right| < ||\phi_1||_{C^1}\delta T_1 < 
\frac{2||\phi_1||_{C^1}\delta}{l_1 - ||\phi_1||_{C^1}\delta}
\]
Therefore, for any choice of \(\epsilon_1 > 0\), if \(\epsilon_0\) and \(\delta\) are small enough, then \(\gamma^1(T_1)\) will be within \(\epsilon_1\) of \(q_1\).   Since this is an open condition, if \(\epsilon_0\) and \(\delta\) are small enough, there will be some \(k_i\)-disk in \(\phi_1^{-1}(\mathcal{D}_{f_2}(p'_i))\) within \(\epsilon_1\) of \(q_1\).   Since \(q_1 \in \phi_1^{-1}(U_2)\), this means that, if \(\epsilon_1\) is small enough, this \(k_i\)-disk will lie within \(\phi_1^{-1}(U_2)\).   Call the image of this disk in \(M\) \(D_2\).

Now repeat this process for \(D_2\) in \(\phi^{-1}_2(U_2)\).   If \(\epsilon_1\) and \(\delta\) are small enough, there will be some \(D_3\) within \(\epsilon_2\) of \(q_2\), which is the point where \(\gamma\) leaves \(U_2\).   Then repeat the process for \(D_3\) in \(\phi^{-1}_3(U_3)\), and so on, until we reach an open \(k_i\)-disk \(D_m \subset \psi_j(N_j)\).   Then \(\psi_j^{-1}(D_m)\) will be a \(k_i\)-disk in \(N_j\).   Analogous to \(N_i\), we can define:
\[
S' = \left\{x \in N_j | x_{1}^2 + ... + x_{k_j}^2 < \epsilon_{m+1}^2\right\}
\]\[
V' = \left\{x \in N_j | x_{1}^2 + ... + x_{k_j}^2 = \epsilon_{m+1}^2\right\}
\]\[
W' = \left\{x \in \partial N_j | x_{1}^2 + ... + x_{k_j}^2 \leq \epsilon_{m+1}^2\right\}
\]
And, if \(\delta\) is small enough, then \(p'_j \in S'\), the ascending manifold of \(p'_j\) intersects \(\partial N_j\) in \(W'\), and the descending manifold of \(p'_j\) intersects \(\partial S'\) in \(V'\).   Since the pullback of \(\gamma\) lies within \(\mathcal{A}_{f_1}(p_j)\), which is within \(S'\), then if \(\epsilon_m\) is small enough, \(\psi_j^{-1}(D_m)\) will lie within \(S'\).   Generically, the intersection of \(\psi_j^{-1}(D_m)\) with \(\mathcal{A}_{f_2}(p'_j)\) will be \((k_i - k_j)\)-dimensional.

Therefore, for any choice of flow line \(\gamma\) from \(p_i\) to \(p_j\) for \(f_1\), we can find some \(\delta_\gamma\) such that if \(\delta < \delta_\gamma\), then for any \(f_2 \in \mathcal{F}(f_1, \delta)\) there exists some flow line from \(p'_i\) to \(p'_j\) for \(f_2\) that lies within \(\epsilon\) of \(\gamma\).   Since the set of flow lines emerging from \(p_i\) is diffeomorphic to \(S^{k_i - 1}\), we can find some \(\delta_i\) such that if \(\delta < \delta_i\), then for any flow line \(\gamma\) emerging from \(p_i\) and for any \(f_2 \in \mathcal{F}(f_1, \delta)\) there exists some flow line from \(p'_i\) to \(p'_j\) for \(f_2\) that lies within \(\epsilon\) of \(\gamma\).   And since there are only finitely many critical points of \(f_1\), we can find \(\delta\) such that, if \(f_2 \in \mathcal{F}(f_1, \delta)\), this holds for any critical point.   This concludes the proof.\end{proof}

\textbf{Lemma~\ref{lemma3.2}:} \emph{Let \(f:M \to \mathbb{R}\) be a Morse function, where \(M\) is a closed manifold.   Let \(Q\) be a compact codimension-0 subset of \(M\) that includes no critical points of \(f\).   Then for any \(\epsilon > 0\) there exists \(\delta > 0\) such that for any critical point \(q\) and any points \(p_1, p_2 \in Q\), if:
\[
d(p_1, p_2) < \delta
\textnormal{, and}
\]\[
p_1, p_2\textnormal{ lie in the same component of }\mathcal{A}_f(q)
\]
Then:
\[
d(\mathcal{D}_f(p_1), \mathcal{D}_f(p_2)) < \epsilon
\]
And, for any \(\epsilon > 0\), there exists \(\delta > 0\) such that, for any points \(p_1, p_2 \in Q\), if
\[
d(p_1, p_2) < \delta
\textnormal{, and}
\]\[
p_1, p_2\textnormal{ lie in the same component of }\mathcal{D}_f(q)
\]
Then:
\[
d(\mathcal{A}_f(p_1), \mathcal{A}_f(p_2)) < \epsilon
\]}

\begin{proof}  We begin by proving that, given a \emph{specific} point \(p_1 \in Q\), we can find \(\delta\) depending on \(p_1\) such that if \(d(p_1, p_2) < \delta\), then \(d(\mathcal{D}_f(p_1), \mathcal{D}_f(p_2)) < \epsilon\)..   We will then use compactness to extend the result to all of \(Q\).   We prove the statement for the descending manifold, since \(\mathcal{A}_f(p) = \mathcal{D}_{-f}(p)\).

Assume \(p_1\) is in the ascending manifold of \(q\).   Let \(\psi_0: N_0 \to M\) be a Morse neighborhood of \(q\), so that \(N_0\) is an open \(n\)-ball and the image \(\psi_0(N_0)\) is contained in an open ball of radius \(\epsilon_0 \leq \frac{1}{2}\epsilon\) centered at \(q\).   Let \(x_1, ..., x_n\) be the coordinates of \(N_0 \subset \mathbb{R}^n\).   By definition, \(\psi_0^*f = f(q_1) \pm x_1^2 \pm ... \pm x_n^2\).   Therefore:
\[
\nabla (\psi_0^* f) = \pm 2x_1\partial_{x_1} \pm ... \pm 2x_n\partial_{x_n}
\]
This means that if \(p \in N_0\) is in the ascending manifold of \(q = 0\), its descending manifold in \(N_0\) is simply a straight line starting at \(p\) and ending at 0.

Therefore, for every point in the image of \(\psi_0(N_0)\) that lies in the ascending manifold of \(q\), its descending manifold is contained in \(\psi_0(N_0)\).

Now, let \(\gamma:[0, l] \to M\) be a parameterization of the descending manifold of \(p_1\) by arc length, \(\gamma(0) = p_1, \gamma(l) = q\).   By the Tubular Flow Theorem, for every point \(\gamma(t), 0 \leq t < l\), we can find a local neighborhood \(\tilde{\psi}_t: N'_t \to M\) of \(\gamma(t)\), where \(N'_t\) is diffeomorphic to \((-1, 1)^n\), where the flow lines of \(\nabla f\) are mapped to lines with constant \(x_2, ..., x_n\) coordinates, and where the image \(\tilde{\psi}_t(N'_t)\) is contained in a \(\frac{1}{2}\epsilon\)-neighborhood of the image of \(\gamma\).   Then, with the addition of \(\psi_0(N_0)\), \(\tilde{\psi}_t(N'_t)\) form an infinite cover of the image of \(\gamma\), as shown in figure~\ref{fig:gammaAndItsCovering}.   Since the image of \(\gamma\) is compact, we can find a finite subcover \(N_0, N'_{t_1}, ..., N'_{t_k}, t_1 > t_2 > ... > t_k\) - that is, \(\psi_1(N_1)\) is close to the critical point \(q_1\), while \(\psi_k(N_k)\) contains our starting point \(p_1\).   We write \(N_i = N'_{t_i}\) and let \(\psi_i\) denote the embedding \(\tilde{\psi}_{t_i}:N'_{t_i} \to M\) for succinctness.

\begin{figure}
\begin{center}\includegraphics{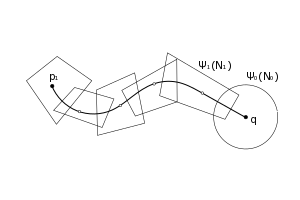}
\caption{\(\gamma\) and its covering by tubular flow neighborhoods and \(\psi_0(N_0)\)\label{fig:gammaAndItsCovering}}\end{center}\end{figure}

Now, consider the preimage \(\psi_1^{-1}(\psi_0(N_0))\).   As shown in Figure~\ref{fig:tubularNeighborhood}, let \(W_1\) denote the projection of \(\psi_1^{-1}(\psi_0(N_0))\) to \((-1,1)^{n-1}\), dropping the \(x_1\) coordinate, and let \(V_1\) denote the ascending manifold of \(\psi_1^{-1}(\psi_0(N_0))\) in \(N_1\).   Since the pullback of the vector field \(\nabla f\) to \(N_1\) is \(\lambda(x)\partial_{x_1}, \lambda(x) > 0\), \(V_1\) will be diffeomorphic to \((-1, 1) \times W_1\):
\[
V_1 = \{(x_1, x_2, ..., x_n) : (x_2, ..., x_n) \in W_1\}
\]
Note that \(V_1\) will contain \(\psi_1^{-1}(\mathcal{D}_f(p_1))\).   Note also that, since \(\psi_0\) is a diffeomorphism onto its image and its domain is open, \(\psi_0(N_0)\) is open.   Therefore \(\psi^{-1}(\psi_0(N_0))\) is open, so \(W_1 \subset (-1,1)^{n-1}\) is open, so \(V_1\) is open.

\begin{figure}
\begin{center}\includegraphics{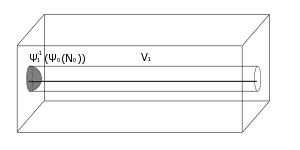}
\caption{\(N_1, W_1, V_1,\) and \(\gamma\)\label{fig:tubularNeighborhood}}\end{center}
\end{figure}

Now consider \(\psi_2^{-1}(\psi_1(V_1))\).   Let \(W_2\) be the projection of \(\psi_2^{-1}(\psi_1(V_1))\) that drops the \(x_1\) coordinate, let \(V_2\) be the ascending manifold of \(\psi_2^{-1}(\psi_1(V_1))\) in \(N_2\), and observe that \(V_2 = (-1, 1) \times W_2\), that \(V_2\) contains \(\psi_2^{-1}(\mathcal{D}_f(p_1)\), and that \(V_2\) is open.   Continue repeating this process until you reach \(V_k\).

Now consider \(\psi_k(V_k)\).   Since \(\psi_k\) is a diffeomorphism onto its image, \(\psi_k(V_k)\) will be open.   Since \(V_k\) contains the \(\psi_k^{-1}(\mathcal{D}_f(p_1))\), \(\psi_k(V_k)\) contains \(p_1\).   Therefore, \(\psi_k(V_k)\) contains an open ball \(B_\delta(p_1)\) of radius \(\delta\) around \(p_1\).

Suppose \(p_2 \in B_\delta(p_1)\), and \(p_2\) lies in the same component of \(\mathcal{D}_f(q)\) as \(p_1\).   Then \(\psi_k^{-1}(p_2) \in V_k\), so its descending manifold in \(N_k\) lies in \(V_k\).   We can continue forward from here, showing that the preimage in \(N_i\) of the descending manifold \(\mathcal{D}_f(p_2)\) lies in \(V_i\) for all \(i\).   Then, once we reach \(N_0\), since \(p_2 \in \mathcal{D}_f(q)\), the descending manifold terminates in a curve to \(q\).   Since \(N_i\) lie in an \(\epsilon\)-neighborhood of \(\mathcal{D}_f(p_1)\), we may conclude that, for any \(p\in Q\), we can find \(\delta(p)\) such that, if:
\[
d(p, p_2) < \delta(p)\mbox{, and}
\]\[
p, p_2\mbox{ are in the same component of }\mathcal{D}_f(q)
\]
Then:
\[
d(\mathcal{D}_f(p), \mathcal{D}_f(p_2)) < \frac{1}{2}\epsilon
\]
Let \(B_{\frac{1}{2}\delta(p)}(p)\) denote an open ball of radius \(\frac{1}{2}\delta(p)\) around \(p\), where \(p \in Q\).   These balls form an infinite open cover of \(Q\).   Since \(Q\) is compact, we can find a finite subcover \(B_{\frac{1}{2}\delta(p'_1)}(p'_1), ..., B_{\frac{1}{2}\delta(p'_m)}(p'_m)\).   Define:
\[
\delta = \frac{1}{2}\min_i \delta(p'_i)
\]
Suppose \(p_1, p_2 \in Q\) and \(d(p_1, p_2) < \delta\).   \(p_1\) must be contained in some ball \(B_{\frac{1}{2}\delta(p'_k)}(p'_k)\).   Then, the distance from \(p_2\) to \(p'_k\) is:
\[
d(p_2, p'_k) \leq d(p_2, p_1) + d(p_1, p'_k) \leq \delta + \frac{1}{2}\delta(p'_k) \leq \delta(p'_k)
\]
So \(p_2 \in B_{\frac{1}{2}\delta(p'_k)}(p'_k)\).   Therefore:
\[
d(\mathcal{D}_f(p_1), \mathcal{D}_f(p'_k)) < \frac{1}{2}\epsilon
\]\[
d(\mathcal{D}_f(p_2), \mathcal{D}_f(p'_k)) < \frac{1}{2}\epsilon
\]
Which in turn implies that:
\[
d(\mathcal{D}_f(p_1), \mathcal{D}_f(p_2)) < \epsilon
\]
This concludes the proof.\end{proof}

\textbf{Lemma~\ref{lemma3.3}:} \emph{Let \(f_1:M \to \mathbb{R}\) be a Morse function, where \(M\) is a compact manifold.   For any choice of \(\epsilon > 0\) and any compact codimension-0 submanifold \(Q \subset M\) that lies in the descending manifold of maxima of \(f_1\) and contains no critical points of \(f_1\), there exists \(\delta\) such that, for any generic choice of \(f_2 \in \mathcal{F}(f_1, \delta)\) and any \(p \in Q\), the ascending manifold of \(p\) for \(\nabla f_2\) will lie within an \(\epsilon\)-neighborhood of the ascending manifold of \(p\) for \(\nabla f_1\).}

\begin{proof}  The proof of this lemma is essentially analogous to the proof of lemma~\ref{lemma3.1}.   Pick any point \(p \in Q\), and let \(\gamma:[0,l] \to M\) be the flow line from \(p\) to its maximum \(q\) for \(-\nabla f_1\) parameterized by arc length, so that \(\gamma(0) = p, \gamma(l) = q\).   Let \(\Psi:U(q) \to M\) be a Morse neighborhood of \(q\) of radius less then \(\epsilon\).   Find a cover of \(\mbox{Im }\gamma - \Psi(U(q))\) by Tubular Flow neighborhoods, \(\phi_1:[-1,1]^n \to U_1, ..., \phi_m:[-1,1]^n \to U_m\).   Restrict \(\delta\) to be small enough that the critical points of any \(f_2 \in \mathcal{F}(f_1, \delta)\) will be in one-to-one correspondence with the critical points of \(f_1\), and the critical point of \(f_2\) corresponding to \(q\) will lie inside \(\Psi(U(q))\).   Note that we require \(f_2\) to always be generic, to ensure that \(p\) will still lie in the descending manifold of the maximum of \(f_2\) corresponding to \(q\).

Consider \(-\nabla(\phi_1^* f_1)\).   By the Tubular Flow Theorem, this will equal:
\[
\nabla (-\phi_1^* f_1) = \lambda_1(x) \partial_{x_1}
\]
For some function \(\lambda_1:[-1,1]^n \to \mathbb{R}, \lambda_1(x) > 0\).   Since \([-1,1]^n\) is compact, we can find \(l_1 > 0\) such that \(\lambda_1(x) > l_1\) for all \(x\).   Let \(q_1\) be the point where \(\phi_1^{-1} \circ \gamma\) leaves \([-1,1]^n\).   We know that, for \(f_2 \in \mathcal{F}(f_1, \delta)\):
\begin{equation}\label{eqn:lemma3-3eq1}
\nabla (\phi_1^* f_2) = \nabla (\phi_1^* f_1) + \nabla (\phi_1^* (f_2 - f_1))
\end{equation}\begin{equation}\label{eqn:lemma3-3eq2}
\left|\nabla (\phi_1^* (f_2 - f_1))\right| \leq \delta ||\phi_1||_{C^1}
\end{equation}

Define \(\gamma_{f_2}:[0,T_1]\to[-1,1]^n\) to be the flow line beginning at \(\phi_1^{-1}(p)\) and ending on \(\partial [-1,1]^n\) for \(-\nabla (\phi_1^* f_2)\), parameterized by \((\gamma_{f_2})' = -\nabla (\phi_1^*f_2)\).   Then:
\begin{equation}\label{eqn:lemma3-3eq3}
-\frac{d\gamma_{f_2}}{dt} = \lambda(\gamma_{f_2}(t))\partial_{x_1} + \nabla (\phi_1^*(f_2 - f_1))
\end{equation}
Let \((\gamma_{f_2})_i\) denote the \(x_i\) coordinate of \(\gamma_{f_2}\).   Then equations~\ref{eqn:lemma3-3eq1},~\ref{eqn:lemma3-3eq2}, and~\ref{eqn:lemma3-3eq3} together imply that:
\[
-\frac{d(\gamma_{f_2})_1}{dt} > l_1 - ||\phi_1||_{C^1}\delta
\]\[
\left|-\frac{d(\gamma_{f_2})_i}{dt}\right| < ||\phi_1||_{C^1}\delta \mbox{ for }i \neq 1
\]
Therefore, since the width of \([-1,1]^n\) is 2, for any \(f_2 \in \mathcal{F}(f_1, \delta)\) we can bound \(T_1\) by:
\[
T_1 < \frac{2}{l_1 - ||\phi_1||_{C^1}\delta}
\]
And, for \(i\neq 1\):
\[
|(\gamma_{f_2}(T_1))_i - (\gamma_{f_2}(0))_i| < ||\phi_1||_{C^1}\delta T_1 < \frac{2||\phi_1||_{C^1}\delta}{l_1 - ||\phi_1||_{C^1}\delta}
\]
Since the pullback of \(\gamma(t)\) to \([-1,1]^n\) has constant \(x_2, ..., x_n\) coordinates for all \(t\), we may conclude that, for any choice of \(\epsilon_1 > 0\), if \(\delta\) is small enough then \(\gamma_{f_2}(T_1)\) will lie within \(\epsilon_1\) of \(q_1\) for any \(f_2 \in \mathcal{F}(f_1, \delta)\) - and, indeed, all points of \(\gamma_{f_2}\) will lie within \(\epsilon_1\) of \(\gamma\) within \(U_1\).   And, if \(\epsilon_1\) is small enough, \(\phi_1(\gamma_{f_2}(T_1))\) will lie inside \(\phi_2([-1,1]^n)\) for any choice of \(f_2 \in \mathcal{F}(f_1, \delta)\).

Now define \(D_1\) to be a closed \(\epsilon_1\)-ball around \(\phi_1(q_1)\) in \(M\), and consider \(\phi_2^{-1}(D_1) \subset [-1,1]^n\).   Define \(q_2\) to be the point in \([-1,1]^n\) where \(\phi_2^{-1} \circ \gamma\) leaves the cube.   For any point \(p_2 \in \phi_2^{-1}(D_1)\), define \(\gamma_{f_2}^{p_2}:[0,T_2] \to [-1,1]^n\) to be the flow line of \(-\nabla \phi_2^*f_2\) starting at \(p_2\) and ending on \(\partial [-1,1]^n\).   By an analogous argument we can again bound:
\[
|(\gamma_{f_2}^{p_2})_i(T_2) - (\gamma_{f_2}^{p_2})_i(0)| < \frac{2||\phi_2||_{C^1}\delta}{l_2 - ||\phi_2||_{C^1}\delta}
\]
Where:
\[
l_2 < \lambda_2(x)\mbox{ for all }x\in [-1,1]^n
\]\[
-\nabla (\phi_2^*f_1) = \lambda_2(x)\partial_{x_1}
\]
Since \(D_1\) is compact, for any choice of \(\epsilon_2 > 0\), for any point \(p_2 \in \phi_2^{-1}(D_1)\), and for any \(f_2 \in \mathcal{F}(f_1, \delta)\), if \(\delta\) is small enough then the flow line of \(-\nabla (\phi_2^*f_2)\) starting at \(p_2\) and ending on \(\partial [-1,1]^n\) will be within \(\epsilon_1 + \epsilon_2\) of \(q\).   Therefore, for any \(\epsilon_1, \epsilon_2 > 0\) and any \(f_2 \in \mathcal{F}(f_1, \delta)\), if \(\delta\) is small enough then the flow line of \(-\nabla f_2\) beginning at \(p\) and ending on \(\phi_2(\partial[-1,1]^n)\) will be within \(\epsilon_1 + \epsilon_2\) of \(\gamma\).

Proceeding in this way, we can ultimately show that, if \(\delta\) is small enough, then for any \(f_2 \in \mathcal{F}(f_1, \delta)\), the flow line of \(-\nabla f_2\) beginning at \(p\) and ending on \(\phi_m(\partial[-1,1]^n)\) will lie within \(\epsilon_1 + ... + \epsilon_m\) of \(\gamma\).   If we choose \(\epsilon_1+... + \epsilon_m\) to be small enough, then the flow line will end on \(\partial( \Psi(U(q)))\), where \(\Psi:U(q) \to M\) is the Morse neighborhood of the maximum \(q\).

Now, consider \(U(q)\).   The pullback of \(f_1\) to \(U(q)\) is:
\[
\Psi^*f_1 = f_1(q) - x_1^2 - ... - x_n^2
\]
Therefore, if we define \(R = x_1\partial_{x_1} + ... + \partial_{x_n}\), then in the ambient metric of \(U(q)\):
\[
(R \cdot \nabla(\Psi^*f_1))|_{\partial U(q)} < 0
\]
And, if we choose \(\delta\) small enough, we can therefore ensure that:
\[
(R \cdot \nabla(\Psi^*f_2))|_{\partial U(q)} < 0
\]
Therefore, if \(\delta\) is small enough, the ascending manifold of any point on \(\partial U(q)\) for any \(\nabla f_2, f_2 \in \mathcal{F}(f_1, \delta)\), will lie inside \(U(q)\).   And since \(U(q)\) has radius less then \(\epsilon\), the image of this ascending manifold in \(M\) must lie within \(\epsilon\) of \(\gamma\).

Therefore, for any \(p \in Q\) we can find \(\delta(p)\) such that, for any generic \(f_2 \in \mathcal{F}(f_1, \delta(p))\), the ascending manifold of \(p\) for \(f_2\) lies within \(\epsilon\) of the ascending manifold of \(p\) for \(f_1\).   And, since \(Q\) is compact, we can therefore find \(\delta\) such that \(\delta(p) < \delta\) for all \(p\).   This concludes the proof.\end{proof}

\textbf{Lemma~\ref{lemma3.4}:}\emph{ Let \(\Lambda \subset J^1(M)\) be a front-generic Legendrian submanifold whose front projection is defined by sheet functions \(f_1:U_1 \to \mathbb{R}, ..., f_m:U_m \to \mathbb{R}\), where \(U_1, ..., U_m \subset M\).   Let \(\hat{\Lambda} \subset J^1(M)\) be a second front-generic Legendrian submanifold whose front projection is defined by sheet functions \(\hat{f}_1:U_1 \to \mathbb{R}, ..., \hat{f}_m:U_m \to \mathbb{R}\).   Then, for any choice of \(\epsilon > 0\), there exists \(\delta > 0\) such that, if:
\[
\left|\left| f_i - \hat{f}_i\right|\right|_{C^1} < \delta\textnormal{ for all }i
\]
Then, there exists a bijection between the rigid gradient flow trees \(\Lambda\) with one positive Reeb chord and the rigid gradient flow trees of \(\hat{\Lambda}\) with one positive Reeb chord, such that a tree \(\Gamma\) of \(\Lambda\) shares the same Reeb chords with its corresponding tree \(\hat{\Gamma}\) of \(\hat{\Lambda}\), and such that the projection of \(\hat{\Gamma}\) to \(M\) lies within an \(\epsilon\)-neighborhood of the projection of \(\Gamma\) to \(M\).}

\begin{proof}  Restrict \(\delta\) to be small enough that the Reeb chords of \(\hat{\Lambda}\) are in one-to-one correspondence with the Reeb chords \(\Lambda\).   Consider some rigid gradient flow tree \(\Gamma\) of \(\Lambda\), which has only one positive Reeb chord and whose vertices are only Reeb chords and \(Y^0\) vertices.   Begin by considering the case where \(\Gamma\) has no 2-valent vertices at Reeb chords.

\(\Gamma\) is equipped with a natural direction, given by the direction of the gradient flow functions.   Label each vertex of \(\Gamma\) with the number of edges of the longest directed path from that vertex to a negative Reeb chord - so, e.g., Reeb chords would be labeled with 0.   Let \(V_i\) denote the set of vertices of \(\Gamma\) which are labeled with \(i\).   For a \(Y^0\) vertex \(v \in V_i\), define \(w_{v}^1, w_{v}^2\) to be the pair of vertices in \(\Gamma\) such that there is an edge in \(\Gamma\) connecting \(v\) to them, directed from \(v\) to \(w_{v}^1, w_{v}^2\).   Define \(f_{w_{v}^j}, g_{w_{v}^j}\) to be the sheet height functions such that the edge from \(v\) to \(w_{v}^j\) flows on the gradient flow \(-\nabla (f_{w_{v}^j} - g_{w_{v}^j})\).   Let \(\hat{f}_{w_v^j}, \hat{g}_{w_v^j}\) denote the corresponding sheet height functions of \(\hat{\Lambda}\).

We will define sets \(Y(v), Z(v)\) iteratively.   We begin by defining, for any \(v \in V_0\), \(Y(v) = \{v\}\).   Then, for any \(v\) - not just in \(V_0\) - we define:
\[
Z(v) = \mathcal{A}_{f_{v} - g_{v}}(Y(v))
\]
And for \(v \notin V_0\), we define:
\[
Y(v) = Z(w_{v}^1) \cap Z(w_{v}^2)
\]

Now, for any Reeb chord of \(\Lambda\) designated by \(v \in V_0\), there is a corresponding Reeb chord of \(\hat{\Lambda}\).   We denote this Reeb chord by \(\hat{v}\), and we define:
\[
\widehat{Y}(v) = \{\hat{v}\}
\]
Then, for \(v \notin V_0\), we define:
\[
\widehat{Z}(\mathcal{A}_{\hat{f}_{v} - \hat{g}_{v}}(\widehat{Y}(v))
\]\[
\widehat{Y}(v) = \widehat{Z}(w_{v}^1) \cap \widehat{Z}(w_{v}^2)
\]
Now, consider any \(v\in V_0\).   By Lemma~\ref{lemma3.1}, if \(\delta > 0\) is small enough, then \(\widehat{Z}(v)\) lies within an \(\epsilon_0\)-neighborhood of \(Z(v)\) for all \(v \in V_0\), for any choice of \(\epsilon_0 < \epsilon\).

Then, let \(N_{\epsilon_0}(Z(v))\) denote an \(\epsilon_0\)-neighborhood of \(Z(v)\).   For \(v \in V_1\), consider \(N_{\epsilon_0}(Z(w_{v}^1)) \cap N_{\epsilon_0}(Z(w_{v}^2))\).   If \(\epsilon_0\) is small enough, then for any choice of \(\delta_1 > 0\), every point in \((N_{\epsilon_0}(Z(w_{v}^1)) \cap N_{\epsilon_0}(Z(w_{v}^2))\) will lie within a \(\delta_1\)-neighborhood of \(Z(w_{v}^1) \cap Z(w_{v}^2) = Y(v)\).   Since \(\widehat{Z}(w_{v}^1) \subset N_{\epsilon_0}(Z(w_{v}^1)),\widehat{Z}(w_{v}^1) \subset N_{\epsilon_0}(Z(w_{v}^1))\), we know \(\widehat{Y}(v)\) is non-empty, and \(\widehat{Y}(v) \subset N_{\epsilon_0}(Z(w_{v}^1)) \cap N_{\epsilon_0}(Z(w_{v}^2))\).   From this we conclude that if \(\epsilon_0\) is small enough, then \(\widehat{Y}(v)\) lies within a \(\delta_1\)-neighborhood of \(Y(v)\).

Now consider \(\widehat{Z}(v)\) for \(v \in V_1\).   We know that \(Y(v)\) must be at least codimension-1, because if they are codimension-0, then \(\Gamma\) is not rigid; and since \(\delta\) is small enough that our Reeb chords correspond, this means \(\widehat{Y}(v)\) must also be at least codimension-1.   Therefore, we can pick \(Q \subset M\) containing \(Y(v)\) but not containing any Reeb chords of \(f_{v_i} - g_{v_i}\).   Therefore, by Lemma~\ref{lemma3.3}, if \(\delta_1\) is small enough, then ascending manifold of \(\widehat{Y}(v)\) for \(-\nabla (f_{v} - g_{v})\) will lie within \(\frac{1}{2}\epsilon_1\) of \(\widehat{Z}(v)\).   And by Lemma~\ref{lemma3.2}, if \(\delta_1\) is small enough, the ascending manifold of \(\widehat{Y}(v)\) for \(-\nabla (f_{v} - g_{v})\) will lie within \(\frac{1}{2}\epsilon_1\) of \(Z(v)\).   Therefore, if \(\delta_1\) is small enough, then \(\widehat{Z}(v)\) will lie within \(\epsilon_1\) of \(Z(v)\).

Now, let \(v \in V_2\), and consider \(N_{\epsilon_1}(Z(w_{v}^1)) \cap N_{\epsilon_1}(Z(w_{v}^2))\).   Note that \(w_{v}^1, w_{v}^2 \in V_1 \cup V_0\).   If \(\epsilon_1\) is small enough, then for any choice of \(\delta_2 > 0\), every point in \((N_{\epsilon_0}(Z(w_{v}^1)) \cap N_{\epsilon_0}(Z(w_{v}^2))\) will lie within \(\delta_2\) of \(Z(w_{v}^1) \cap Z(w_{v}^2) = Y(v)\).   Since \(\widehat{Y}(v) \subset N_{\epsilon_1}(Z(w_{v}^1)) \cap N_{\epsilon_1}(Z(w_{v}^2))\), we conclude that if \(\epsilon_1\) is small enough, then \(\widehat{Y}(v)\) is non-empty and lies within a \(\delta_2\)-neighborhood of \(Y(v_i)\).

We keep repeating this process until we reach \(Z(v_a), \widehat{Z}(v_a)\), where \(a\) is the positive Reeb chord of \(\Gamma\), and \(v_a\) is the \(Y^0\) vertex connected by an edge to \(a\).   By this process, we show that \(\widehat{Z}(v_a)\) lies within an \(\epsilon_m\)-neighborhood of \(Z(v_a)\).   Let \(\hat{a}\) denote the Reeb chord of \(\hat{\Lambda}\) corresponding to \(a\).

Now define \(X(a), \widehat{X}(a)\) to be:
\[
X(a) = \mathcal{D}_{-(f_{v_a} - g_{v_a})}(a) \cap Z(v_a)
\]\[
\widehat{X}(a)= \mathcal{D}_{-(\hat{f}_{v_a} - \hat{g}_{v_a})}(\hat{a}) \cap \widehat{Z}(v_a)
\]
Since \(\Gamma\) is rigid, \(X(a), \widehat{X}(a)\) must be one-dimensional, and \(X(a)\) must be equal to the image in the base space of the edge of \(\Gamma\) emerging from \(a\).   By Lemma~\ref{lemma3.1}, if \(\delta\) is small enough, then \(\mathcal{D}_{-(f_{v_a} - g_{v_a})}(a)\) will lie within \(\epsilon_{m+1}\) of \(\mathcal{D}_{-(\hat{f}_{v_a} - \hat{g}_{v_a})}(\hat{a})\).   And we already know that if \(\delta, \epsilon_1,..., \epsilon_m\) are small enough, then \(Z(v_a)\) will lie within \(\epsilon_{m+1}\) of \(\widehat{Z}(v_a)\).   Therefore, if \(\epsilon_{m+1}\) is small enough, then \(X(a)\) will lie within \(\delta_m\) of \(\widehat{X}(a)\).   And, since \(X(a), \widehat{X}(a)\) are one-dimensional, we know that \(\partial X(a) = \{a, v_a\}\) and \(\partial \widehat{X}(a) = \{\hat{a}, \hat{v}_a\}\).

Now define \(X_i(v_a), \widehat{X}_i(v_a)\) to be:
\[
X_i(v_a) = \mathcal{D}_{-(f_{w_{v_a}^i} - g_{w_{v_a}^i})}(v_a) \cap Z(w_{v_a}^i)
\]\[
\widehat{X}_i(v_a) = \mathcal{D}_{-(\hat{f}_{w_{v_a}^i} - \hat{g}_{w_{v_a}^i})}(\hat{v}_a) \cap \widehat{Z}(w_{v_a}^i)
\]

Once again, since \(\Gamma\) is rigid, we can conclude that \(X_i(v_a), \widehat{X}_i(v_a)\) are one-dimensional.   We know that \(Z(w_{v_a}^i)\) lies within \(\epsilon_{m-1}\) of \(\widehat{Z}(w_{v_a}^i)\).   And, by Lemma~\ref{lemma3.2}, if \(\delta, \delta_m\) are small enough, \(\mathcal{D}_{-(f_{w_{v_a}^i} - g_{w_{v_a}^i})}(v_a)\) will lie within \(\epsilon_{m+2}\) of \(\mathcal{D}_{-(\hat{f}_{w_{v_a}^i} - \hat{g}_{w_{v_a}^i})}(\hat{v}_a)\).   Therefore, if \(\delta, \epsilon_1,..., \epsilon_m, \epsilon_{m+1}\) are small enough, \(\widehat{X}_i(v_a)\) will lie within \(\delta_{m+1}\) of \(X_i(v_a)\).   Repeat this process until we reach the Reeb chords.   The trace of the \(\widehat{X}_i(v)\) form a unique rigid gradient flow tree \(\hat{\Gamma}\) of \(\hat{\Lambda}\) with Reeb chords corresponding to the Reeb chords of \(\Gamma\), and lying within an \(\epsilon\)-neighborhood of \(\Gamma\).

Now consider the case where \(\Gamma\) has a two-valent Reeb chord.   We can break \(\Gamma\) into sub-trees \(\Gamma_1, ..., \Gamma_k\) at every two-valent Reeb chord.   We can then repeat the process.

Therefore, for every rigid \(\Gamma\) there exists some \(\delta_\Gamma\) such that if \(||f_i - \hat{f}_i||_{C^1} < \delta\) for all \(i\), then there exists a unique rigid gradient flow tree \(\hat{\Gamma}\) of \(\hat{\Lambda}\) that lies within an \(\epsilon\)-neighborhood of \(\Gamma\).   Since there are only finitely-many rigid flow trees, we can therefore find \(\delta\) such that, if \(||f_i - \hat{f}_i||_{C^1} < \delta\) for all \(i\), then \(\hat{\Gamma}\) lies within an \(\epsilon\)-neighborhood of \(\Gamma\) for all \(\Gamma\).   This completes the proof.\end{proof}

\end{document}